\documentclass{amsart}
\usepackage[T1]{fontenc}
\usepackage{amsmath}
\usepackage{amsfonts} 
\usepackage{amsthm}
\usepackage[mathscr]{euscript}
\usepackage{amssymb}
\usepackage{tikz-cd}
\usepackage{derivative}
\usepackage{hyperref}
\usepackage{dirtytalk}
\usepackage{array}
\usepackage{diagbox}

\usepackage[backend=biber,
style=numeric,giveninits=true]{biblatex} 
\theoremstyle{plain}
\newtheorem{thm}{Theorem}[section]
\newtheorem{lem}[thm]{Lemma}
\newtheorem{obs}[thm]{Observation}
\newtheorem{prop}[thm]{Proposition}
\newtheorem{cor}[thm]{Corollary}
\newtheorem{conj}[thm]{Conjecture}

\newtheorem{case}{Case}
\newtheorem*{claim*}{Claim}

\theoremstyle{definition}
\newtheorem{defi}[thm]{Definition}

\theoremstyle{remark}
\newtheorem{rem}{Remark}[thm]

\numberwithin{equation}{section}

\newcommand{\Rvar}[1]{\mathscr{#1}}
\newcommand{\var}[1]{{#1}}

\newcommand{\R}{\mathbb{R}}
\newcommand{\C}{\mathbb{C}}
\newcommand{\K}{\mathbb{K}}
\newcommand{\Hq}{\mathbb{H}}
\newcommand{\Z}{\mathbb{Z}}

\newcommand{\F}{\mathbb{F}}
\newcommand{\RR}{\mathcal{R}}

\newcommand{\sph}{\mathbb{S}}

\newcommand{\alg}{\mathrm{alg}}
\newcommand{\Gr}{\mathrm{Gr}}
\newcommand{\Th}{\mathrm{Th}}
\newcommand{\PT}{\mathrm{PT}}
\newcommand{\pt}{\mathrm{pt}}
\newcommand{\cofib}{\mathrm{cofib}}

\DeclareMathOperator{\codim}{codim}

\DeclareMathOperator{\inte}{int}
\DeclareMathOperator{\Ext}{Ext}
\DeclareMathOperator{\tr}{tr}
\DeclareMathOperator{\supp}{supp}
\DeclareMathOperator{\im}{im}

\title{Invariants of real affine varieties based on their complexifications}
\author{Juliusz Banecki}

\address{Faculty of Mathematics and Computer Science,
Jagiellonian University, ul. Lojasiewicza 6, 30-348 Krakow, Poland}
\email{juliusz.banecki@student.uj.edu.pl}
\subjclass[2020]{14P25, 19L47, 55N22}
\keywords{real algebraic set, algebraic approximation, Real-orientation}

\begin{document}
\begin{abstract}
We introduce a new family of invariants of real algebraic sets defined in terms of the topology of their complexifications and compute some of these invariants for spheres. This allows us to completely classify topological isomorphism classes of algebraic vector bundles over products of two spheres. We also obtain new results concerning both the existence and nonexistence of regular maps from products of spheres into spheres. Additionally, we show that the newly defined invariants provide obstructions to weak algebraic approximation of smooth submanifolds of real algebraic sets, disproving a conjecture of Kucharz and Kurdyka. 
\end{abstract}
\maketitle
\setcounter{tocdepth}{1}
\tableofcontents
\section{Introduction}
The paper is concerned with certain realisation problems in real algebraic geometry. A formal introduction, which sets up the language that we use is contained in Section \ref{sec:preliminaries}. In the current section we choose to avoid using specialised terminology. To keep things simple, our objects of interest are \emph{real algebraic sets}, i.e. subsets of $\R^n$ for some $n\geq 0$ given as zero sets of polynomials. A \emph{regular map} is a map between real algebraic sets, all of whose coordinates can be written as quotients of polynomials with nonvanishing denominators. 
\subsection{Motivation}
A typical problem in real algebraic geometry deals with finding algebraic models for various topological data. A prime example of a result of this sort is the classical \emph{Nash-Tognoli theorem}:
\begin{thm}[\cite{tognoliSuCongetturaDi1973}]
Let $M$ be a smooth closed connected manifold without boundary. Then, there exists a nonsingular real algebraic set $X$, which as a smooth manifold is diffeomorphic to $M$.
\end{thm}
Here, \emph{nonsingular} means a certain algebraic criterion, which implies that $X$ is a smooth submanifold of $\R^n$. For the definition of this notion see \cite[Definition 3.3.4]{bochnakRealAlgebraicGeometry1998} or Section \ref{sec:preliminaries:real}.

A real algebraic set $X$ with properties as above is called \emph{a real algebraic model of $M$}. The Nash-Tognoli theorem has since been generalised in many ways, including versions allowing for simultaneous construction of models for $M$ and a finite number of its submanifolds \cite{akbulutRelativeNashTheorem1981}, or versions which allow the defining equations for $X$ to be taken with coefficients in $\mathbb Q$ \cite{ghiloniNashTognoliTheoremRationals2025}.

Some manifolds admit real algebraic models which are of particularly simple nature. For example, the sphere $\sph^n$ can be identified with the subset of $\R^{n+1}$ given by the single equation
\begin{equation*}
    \sum_i x_i^2=1.
\end{equation*}
Likewise, the Grassmannian $\mathrm{Gr}_{k}(\R^n)$ may be identified with the subspace of the space $\mathcal M_{n\times n}$ of $n\times n$ matrices defined by
\begin{equation*}
    \mathrm{Gr}_{k}(\R^n)=\{P\in \mathcal M_{n\times n}:P^\intercal =P,P^2=P,\tr P=k\}.
\end{equation*}
Identifying $\mathcal M_{n\times n}$ with $\R^{n^2}$ we obtain the standard real algebraic model of $\mathrm{Gr}_{k}(\R^n)$.

Given two real algebraic sets $X$ and $Y$, one may pose a different kind of question regarding algebraic realisation of topological data: which homotopy classes of continuous maps between $X$ and $Y$ are represented by regular maps? A typical result concerning this question is the following one:
\begin{thm}[{\cite[Corollary 2.7]{bochnakAlgebraicApproximationMappings1987}}]
For every triple $(m,n,k)$ of positive integers with $n\geq k$, every map from (the standard real algebraic model of) the sphere $\sph^m$ into (the standard real algebraic model of) the Grassmannian $\Gr_k(\R^n)$ is homotopic to a regular one.
\end{thm}
On the contrary, the question whether for all $n,k\geq 0$ every continuous map between $\sph^n$ and $\sph^k$ is homotopic to a regular one is a notoriously difficult open question.

The setting of real algebraic geometry is quite flexible allowing for results like the Nash-Tognoli theorem. On the other hand, there do exist many negative results concerning algebraic realisation of topological data as well. A concrete low-dimensional example is the following:
\begin{thm}[{\cite[Theorem 3.1]{bochnakRealizationHomotopyClasses1987}}]
Every regular map from $\sph^1\times \sph^1$ into $\sph^2$ is homotopic to the constant map.
\end{thm}
To obtain such negative results, it is typical to consider certain invariants of real algebraic sets, which can obstruct the existence of a regular map in a given homotopy class or the existence of an algebraic model of a manifold satisfying certain additional conditions. Many different kinds of obstructions have been introduced for these purposes. As an example, there are the invariants
\begin{equation*}
    H_\ast ^\alg(X;\Z/2),H^\ast_\alg(X;\Z/2),H^{2\ast}_{\C-\alg}(X;\Z),H^{\ast}_\C(X;\Z)
\end{equation*}
(see \cite[Chapter II.7]{akbulutTopologyRealAlgebraic1992} and \cite{buchnerAlgebraicVectorBundles1987,ozanHomologyRealAlgebraic2001}) based on singular homology and cohomology. They have seen much use in real algebraic geometry as they are defined in elementary ways and are usually computable without much effort. On the other hand, there are some interesting problems where these invariants are too weak to be applied. For example, they have no chance of obstructing the existence of regular maps between spheres of different dimensions, as such maps induce the trivial morphism in singular homology.

A somewhat stronger family of invariants is based on algebraic K-theory. To introduce them, let us fix some notation. For a real algebraic set $X$ and an $\R$-algebra $\K$ equal to $\R$, $\C$ or $\Hq$ denote by $\RR(X,\K)$ the ring of regular functions from $X$ to $\K$ with pointwise addition and multiplication, where $\C$ and $\Hq$ are treated as real algebraic sets via the identifications $\C\cong \R^2$ and $\Hq\cong \R^4$. One can then consider the algebraic K-groups
\begin{equation*}
    K_0(\RR(X,\R)),K_0(\RR(X,\C)),K_0(\RR(X,\Hq)).
\end{equation*}
These invariants are generally stronger than the ones based on singular homology, but they are much more difficult to compute. For example, from the 1980s up until now it has been an open problem to compute the group $K_0(\RR(\sph^2\times \sph^2,\Hq))$.

The main goal of the current paper is to introduce a new broad family of invariants of real algebraic sets, which can be used as obstructions in various problems connected with realisation of homotopy classes by regular maps and with finding algebraic models of smooth manifolds. Their precise definitions are provided in Section \ref{sec:invariants}; in the following subsection we give only a very brief sketch explaining the idea lying behind their construction. 

\subsection{Sketch of construction of the new invariants}
First, we would like to recall the definition of the invariant $H^\ast_\C(X;\Z)$ introduced by Ozan \cite{ozanHomologyRealAlgebraic2001}, and then generalised by Dołęga \cite{dolegaComplexificationCohomologyReal2005}. Informally, the invariant is constructed as follows
\begin{defi}\label{defi:introHC}
Let $X$ be a nonsingular compact real algebraic set. Using well-known theory one can find a smooth complex projective variety $X_\C$ defined over the real numbers, whose real locus is Zariski dense in $X_\C$ and isomorphic to $X$ (see \cite[Theorem 2.3.7]{mangolteRealAlgebraicVarieties2020}). Then, for $d\geq 0$ the group $H^d_\C(X;\Z)$ is defined as the pullback of $H^d(X_\C;\Z)$ through the inclusion $i:X\hookrightarrow X_\C$:
\begin{equation*}
    H^d_\C(X;\Z):=\im(H^d(X_\C;\Z)\xrightarrow{i^\ast} H^d(X;\Z)).
\end{equation*}
\end{defi}
It is a nontrivial fact, but this definition does not depend on the particular choice of the variety $X_\C$. An essential ingredient of the proof of this fact is that the complex variety $X_\C$ is orientable.

The main idea behind our new family of invariants is to extend Definition \ref{defi:introHC} to a more general setting, where instead of singular cohomology one considers some stronger generalised cohomology theory $E^\ast$. For such an invariant to be properly defined, we still need an orientability condition, which turns out to hold for every \emph{complex-oriented} cohomology theory (see \cite{adamsStableHomotopyGeneralised1974} or Section \ref{sec:preliminaries:cohomology} for the definition). For every such generalised cohomology theory $E^\ast$ we define a corresponding invariant $E^\ast_\C$ of real algebraic sets. As an example, this includes a new invariant $KU^\ast_\C(X)$, which is approximately as strong as $K_0(\RR(X,\C))$, but which is better suited for computations. We do not have to stop at K-theory, however, as this way we can define invariants based on complex cobordism or on other intermediate complex-oriented cohomology theories.

We take this idea further to define even stronger invariants. The idea behind this strengthening is to endow the complex variety $X_\C$ with the $\Z/2$-action induced by the conjugation. In $\Z/2$-equivariant algebraic topology, there is the corresponding notion of a \emph{Real-oriented} cohomology theory, which fits our framework perfectly. Similarly to the complex-oriented case, this allows us to define the invariant $K\R_\C^\star(X)$ based on Atiyah's Real K-theory \cite{atiyahTheoryReality1966}, which can be thought of as a more approachable approximation to the algebraic K-groups $K_0(\RR(X,\R))$ and $K_0(\RR(X,\Hq))$. We can still go further than that and define invariants based on stronger Real-oriented cohomology theories, including the Real cobordism of Landweber \cite{landweberConjugationsComplexManifolds1968a}. Such invariants seem to be exceptionally strong, though we are not yet able to perform any meaningful computations with theories stronger than K-theory due to their complexity.

\subsection{Results concerning products of spheres}\label{sec:introduction:products}
After we define the invariants in Section \ref{sec:invariants}, we focus mostly on computations involving the functors $KU_\C^\ast$ and $K\R_\C^\star$. We are able to fully compute the groups $KU^\ast_\C(\sph^n)$, $K\R^\star_\C(\sph^n)$ for all $n$. From this, we deduce new interesting theorems concerning products of spheres, which we describe below.

First of all, for the first time we are able to completely describe the algebraic K-groups
\begin{equation*}
    K_0(\RR(\sph^n\times \sph^m,\K)),
\end{equation*}
where $n,m\geq 1$ and $\K=\R,\C$ or $\Hq$. For our convenience, we formulate our results in terms of reduced K-theory. To explain them clearly, we recall that there are natural comparison maps from the algebraic K-theory of a real algebraic set to its topological K-groups:
\begin{align*}
    \widetilde K_0(\RR(\sph^n\times \sph^m,\R))&\rightarrow \widetilde{KO}^0(\sph^n\times \sph^m),\\
    \widetilde K_0(\RR(\sph^n\times \sph^m,\C))&\rightarrow \widetilde{KU}^0(\sph^n\times \sph^m),\\
    \widetilde K_0(\RR(\sph^n\times \sph^m,\Hq))&\rightarrow \widetilde{KSp}^0(\sph^n\times \sph^m).
\end{align*}
These maps are monomorphisms \cite[Theorem 12.3.6]{bochnakRealAlgebraicGeometry1998}. This way we may treat the algebraic K-groups as subgroups of the corresponding topological K-groups. These subgroups classify topological isomorphism classes of algebraic vector bundles over $\sph^n\times \sph^m$; a topological vector bundle over the product of spheres is isomorphic to an algebraic one if and only if the K-theory class it represents lies in the image of the respective comparison map \cite[Proposition 12.3.5]{bochnakRealAlgebraicGeometry1998}.

To describe these subgroups, we first note the following consequence of topological K-theory satisfying Eilenberg–Steenrod axioms for generalised cohomology theory (cf. Section \ref{sec:invariantsProduct}):
\begin{obs}
There are splittings
\begin{align*}
    \widetilde{KO}^0(\sph^n\times \sph^m)&\cong \widetilde{KO}^0(\sph^{n+m})\oplus \widetilde{KO}^0(\sph^n)\oplus \widetilde{KO}^0(\sph^m),\\
    \widetilde{KU}^0(\sph^n\times \sph^m)&\cong \widetilde{KU}^0(\sph^{n+m})\oplus \widetilde{KU}^0(\sph^n)\oplus \widetilde{KU}^0(\sph^m),\\
    \widetilde{KSp}^0(\sph^n\times \sph^m)&\cong \widetilde{KSp}^0(\sph^{n+m})\oplus \widetilde{KSp}^0(\sph^n)\oplus \widetilde{KSp}^0(\sph^m).
\end{align*}
\end{obs}

For unitary K-theory, the following description of $\widetilde K_0(\RR(\sph^n\times \sph^m,\C))$ follows from the work of Bochnak and Kucharz \cite{bochnakRealizationHomotopyClasses1987}, although to our best knowledge it has never been stated in such an explicit form:
\begin{thm}\label{thm:introKU}
Let $n$ and $m$ be positive integers. Let $G\subset \widetilde{KU}^0(\sph^{n+m})$ be the zero group if both $n$ and $m$ are odd, and otherwise be the entire group $\widetilde{KU}^0(\sph^{n+m})$. Then, the image of $\widetilde K_0(\RR(\sph^n\times \sph^m,\C))$ in $\widetilde{KU}^0(\sph^n\times \sph^m)$ through the comparison map is the group
\begin{equation*}
    G\oplus \widetilde{KU}^0(\sph^n)\oplus \widetilde{KU}^0(\sph^m).
\end{equation*}
\end{thm}

In the remaining cases of orthogonal and quaternionic K-theory, the description turns out to be more complicated than that. To be able to state it in a concise way, we introduce an auxiliary function $\varphi:\Z\times \Z\rightarrow\{1,2,\infty\}$. We define it in such a way that $\varphi(n,m)$ depends only on the congruence of $n$ and $m$ modulo $8$ according to the following table:
\begin{center}
\begin{tabular}{|l|l|l|l|l|l|l|l|l|} 
\hline
\diagbox{n}{m} & 0 & 1     & 2 & 3     & 4 & 5     & 6 & 7      \\ 
\hline
0           & 1    & 1        & 1    & 1        & 1    & 1        & 1    & 1         \\ 
\hline
1           & 1    & 1        & 1    & $\infty$ & 1    & 1        & 1    & $\infty$  \\ 
\hline
2           & 1    & 1        & 1    & 1        & 1    & 1        & 2    & 2         \\ 
\hline
3           & 1    & $\infty$ & 1    & 1        & 1    & $\infty$ & 2    & 2         \\ 
\hline
4           & 1    & 1        & 1    & 1        & 1    & 1        & 1    & 1         \\ 
\hline
5           & 1    & 1        & 1    & $\infty$ & 1    & 1        & 1    & $\infty$  \\ 
\hline
6           & 1    & 1        & 2    & 2        & 1    & 1        & 1    & 1         \\ 
\hline
7           & 1    & $\infty$ & 2    & 2        & 1    & $\infty$ & 1    & 1         \\
\hline
\end{tabular}
\end{center}
For orthogonal and quaternionic K-theory our result can be described as follows:
\begin{thm}\label{thm:introKOKSp}
Let $n$ and $m$ be positive integers. Let $G\subset \widetilde{KO}^0(\sph^{n+m})$ be the subgroup of index $\varphi(n,m)$ (it is unique, as $\widetilde{KO}^0(\sph^{n+m})$ is cyclic according to the Bott periodicity theorem). Then, the image of $\widetilde K_0(\RR(\sph^n\times \sph^m,\R))$ in $\widetilde{KO}^0(\sph^n\times \sph^m)$ is the group
\begin{equation*}
    G\oplus \widetilde{KO}^0(\sph^n)\oplus \widetilde{KO}^0(\sph^m).
\end{equation*}

Similarly, let $G'\subset \widetilde{KSp}^0(\sph^{n+m})$ be the subgroup of index $\varphi(n,m+4)$. Then, the image of $\widetilde K_0(\RR(\sph^n\times \sph^m,\Hq))$ in $\widetilde{KSp}^0(\sph^n\times \sph^m)$ is the group
\begin{equation*}
    G'\oplus \widetilde{KSp}^0(\sph^n)\oplus \widetilde{KSp}^0(\sph^m).
\end{equation*}
\end{thm}

For comparison, we include a table of the order of the groups $\widetilde{KO}^{0}(\sph^{n+m})$:
\begin{center}
\begin{tabular}{|l|l|l|l|l|l|l|l|l|} 
\hline
\diagbox{n}{m} & 0 & 1     & 2 & 3     & 4 & 5     & 6 & 7      \\ 
\hline
0           & $\infty$    & 2        & 2    & 1        & $\infty$    & 1        & 1    & 1         \\ 
\hline
1              & 2        & 2    & 1        & $\infty$    & 1        & 1    & 1        & $\infty$ \\ 
\hline
2                  & 2    & 1        & $\infty$    & 1        & 1    & 1        & $\infty$ & 2 \\ 
\hline
3               & 1        & $\infty$    & 1        & 1    & 1        & $\infty$ & 2 & 2\\ 
\hline
4                  & $\infty$    & 1        & 1    & 1        & $\infty$ & 2 & 2& 1\\
\hline
5               & 1        & 1    & 1        & $\infty$ & 2 & 2& 1& $\infty$\\
\hline
6                   & 1    & 1        & $\infty$ & 2 & 2& 1& $\infty$& 1\\
\hline
7              & 1        & $\infty$ & 2 & 2& 1& $\infty$& 1& 1\\         
\hline
\end{tabular}
\end{center}

Besides mere computation of algebraic K-groups, our results have applications in the more concrete classical problem of classifying homotopy classes of maps from $\sph^n\times \sph^m$ to $\sph^{n+m}$ represented by regular maps. Before presenting our results, we briefly review the current state of the art concerning this problem.

First of all, using algebraic K-theory Bochnak and Kucharz were able to show the following:
\begin{thm}[{\cite[Theorem 3.1]{bochnakRealizationHomotopyClasses1987}}]\label{thm:BKNoMap}
Let $n,m$ be odd positive integers. Then, every regular map from $\sph^n\times \sph^m$ to $\sph^{n+m}$ is homotopic to the constant map.
\end{thm}
On the contrary, if at least one of the numbers $n$ or $m$ is even, then there is a very simple construction of a regular (indeed polynomial) map $f:\sph^n\times \sph^m\rightarrow \sph^{n+m}$ of topological degree two \cite[Theoreme 14]{lodayApplicationsAlgebriquesTore1973}. Combining this with \cite[Theorem 1.2]{baneckiAlgebraicHomotopyClasses2024} we obtain the following:
\begin{prop}\label{prop:dichotomy}
If at least one of the numbers $n$ or $m$ is even, then exactly one of the following holds true:
\begin{enumerate}
    \item either every integer is realised as the topological degree of a regular map from $\sph^n\times \sph^m$ to $\sph^{n+m}$,
    \item or the set of degrees which arise this way coincides with the set of even integers.
\end{enumerate}
\end{prop}

So far, it has been unknown whether the second alternative could ever hold true. Prior to this work, the only result beyond Proposition \ref{prop:dichotomy} applicable when one of the dimensions is even was the following:
\begin{thm}[{\cite[Theoreme 12]{lodayApplicationsAlgebriquesTore1973}}]\label{thm:Loday}
Let $n$ and $m$ be natural numbers. Write $n$ as
\begin{equation*}
    n=2^{4a+b}(2c+1),
\end{equation*}
where $a\geq 0$, $ 3\geq b\geq 0$ and $c\geq 0$ are integers. Assume that $m$ is strictly smaller than the \emph{Radon-Hurwitz} number $\rho(n):=8a+2^b$. Then, every map $f:\sph^n\times \sph^m \rightarrow \sph^{n+m}$ is homotopic to a regular one.
\end{thm}
For small values of $m$ this provides many examples of pairs $(n,m)$ with $n$ even for which the first alternative in Proposition \ref{prop:dichotomy} holds true. In the long run, however, the inequality forces $n$ to be exponentially larger than $m$, and the pairs treated by Theorem \ref{thm:Loday} become rather sparse.

Motivated, among other reasons, by the lack of known pairs for which the second alternative in Proposition \ref{prop:dichotomy} is known to hold true, in 2022 Bochnak and Kucharz posed the following conjecture:
\begin{conj}[{\cite[Conjecture II]{bochnakApproximationMapsReal2022}}]\label{conj:BK}
Let $n$ be a positive odd integer, and let $X$ be a nonsingular compact connected real algebraic set of dimension $n$. Then, every continuous map from $X$ to $\sph^n$ is homotopic to a regular one.
\end{conj}
The conjecture is known to hold true for $n=1$; prior to this work it was open for all other values of $n$.

We can now pass to the presentation of our new results. Using the functor $K \R_\C^\star(X)$ we are able to show the following:
\begin{thm}\label{thm:23mod4noMap}
Let $n$ and $m$ be positive integers. Assume that one of them is congruent to $2$ modulo $4$, and that the other one is congruent to $2$ or $3$ modulo $4$. Then, a continuous map $f:\sph^n\times \sph^m\rightarrow \sph^{n+m}$ is homotopic to a regular one if and only if $\deg f$ is even.
\end{thm}
This provides the first family of examples of pairs $(n,m)$ for which the second alternative in Proposition \ref{prop:dichotomy} holds true. Moreover, this disproves Conjecture \ref{conj:BK} for all $n$ congruent to $1$ modulo $4$, $n>1$. 

In the other direction, in the current paper we were also able to generalise Theorem \ref{thm:Loday}. The precise result is slightly technical and is explained in Section \ref{sec:products:bilinear}. Instead of stating it here, we present the following corollary of our construction:
\begin{thm}
Let $n$ and $m$ be two positive integers. Assume that $n\leq 8$ and that the number $\binom{n+m}{n}$ is odd. Then, every continuous map from $\sph^n\times \sph^m$ into $\sph^{n+m}$ is homotopic to a regular one.
\end{thm}
Combining this with Theorems \ref{thm:BKNoMap} and \ref{thm:23mod4noMap} we obtain the following:
\begin{cor}
Assume that $n\leq 3$. Then, the number $\binom{n+m}{n}$ is odd if and only if every continuous map from $\sph^n\times \sph^m$ into $\sph^{n+m}$ is homotopic to a regular one.
\end{cor}
This uncovers a pattern in the problem considered. Quite naturally, we formulate the following conjecture:
\begin{conj}\label{conj:main}
Let $n$ and $m$ be positive integers. Then, the following conditions are equivalent:
\begin{enumerate}
    \item every continuous map from $\sph^n\times \sph^m$ into $\sph^{n+m}$ is homotopic to a regular one,
    \item the number $\binom{n+m}{n}$ is odd.
\end{enumerate}
\end{conj}
We provide some more evidence for this conjecture in Section \ref{sec:products:bilinear}. 

It seems reasonable to believe that computations of the invariants $E^\star_\C(\sph^n)$ for Real-oriented cohomology theories $E^\star$ stronger than $K\R^\star$ will bring new pairs for which there is no regular map $f:\sph^n\times \sph^m\rightarrow \sph^{n+m}$ of odd topological degree. Therefore, the next step in verification of the conjecture would be to compute these invariants.

\subsection{Applications in the algebraic approximation problem}\label{sec:introduction:approx}
The newly defined invariants admit another application in the problem of \emph{algebraic approximation of submanifolds}, which we recall here.

\begin{defi}
Let $X$ be a nonsingular real algebraic set, and let $M\subset X$ be a smooth compact submanifold of $X$. We say that $M$ \emph{admits an algebraic approximation in $X$} if for every neighbourhood $\mathcal U$ of the inclusion $i:M\hookrightarrow X$ in the $\mathcal C^\infty$-topology there is a map $j:M\hookrightarrow X$ in $\mathcal U$ whose image $j(M)$ is a nonsingular Zariski closed subset of $X$.

We say that $M$ \emph{admits a weak algebraic approximation in $X$} if for every neighbourhood $\mathcal U$ of the inclusion $i:M\hookrightarrow X$ in the $\mathcal C^\infty$-topology there is a map $j:M\hookrightarrow X$ in $\mathcal U$ whose image $j(M)$ is equal to the nonsingular locus of a Zariski closed subset of $X$.
\end{defi}
In the definition, and in this entire subsection, a \emph{manifold} is implicitly assumed to be without boundary, unless stated otherwise.
\begin{rem}
The second part of the definition makes use of the fact that in real algebraic geometry, the nonsingular locus of a real algebraic set $X$ may be both closed and open in the Euclidean topology in $X$. This is connected to the existence of the so-called \emph{noncentral points} (cf. \cite[Section 7.6]{bochnakRealAlgebraicGeometry1998}). 
\end{rem}

In general, deciding whether a submanifold of a given real algebraic set admits (weak) algebraic approximation may be a difficult task. Some obstructions to the validity of this condition are known, the most important of which are based on the notion of unoriented bordism, which we recall below.

Let $T$ be a topological space, let $N_1$ and $N_2$ be two closed manifolds of the same dimension $d$ and let $f_1:N_1\rightarrow T$ and $f_2:N_2\rightarrow T$ be two continuous maps. We say that $(N_1,f_1)$ and $(N_2,f_2)$ are \emph{bordant}, if there exists a smooth compact manifold with boundary $C$, and a diffeomorphism $\partial C\cong N_1\sqcup N_2$, with the property that the map $f_1\sqcup f_2$ extends to a map $f:C\rightarrow T$ defined on the entire manifold $C$. This defines an equivalence relation on the set of pairs of the form $(N,f)$, where $N$ is a manifold of dimension $d$ and $f:N\rightarrow T$ is a continuous map. The set of equivalence classes under this relation forms a group under the operation of disjoint union, denoted by $MO_d(T)$. The bordism class represented by a pair $(N,f)$ is usually denoted by $[f:N\rightarrow T]$.
\begin{defi}
Let $X$ be a real algebraic set. The subgroup of $MO_d(X)$ generated by classes of the form $[f:Y\rightarrow X]$, where $Y$ is a compact nonsingular real algebraic set and $f$ is a regular map, is denoted by $MO_d^\alg(X)$. Its elements are called \emph{algebraic bordism classes.}
\end{defi}
Whether a given bordism class is algebraic can be rephrased completely in terms of the functor $H_i^\alg(X;\Z/2)$ \cite{kucharzAlgebraicApproximationSubmanifolds2026}, which makes it verifiable in reasonable cases.

One obstruction to the existence of a (weak) algebraic approximation of $M$ is given in terms of the unoriented bordism class of the inclusion $i:M\hookrightarrow X$:
\begin{obs}\label{obs:approx=>bordism}
Let $X$ be a nonsingular real algebraic set, and let $M\subset X$ be a smooth compact submanifold of $X$. Assume that $M$ admits a weak algebraic approximation in $X$. Then, the bordism class of the inclusion $[i:M\hookrightarrow X]$ is algebraic.
\end{obs}
Until only recently, it had been unknown whether the algebraicity of the bordism class of the inclusion was a sufficient condition for the existence of a (weak) algebraic approximation of $M$ in general. In 2024, the following positive result was obtained by Benoist:
\begin{thm}[{\cite[Theorem 0.6]{benoistSubvarietiesNonsingularReal2024}}]\label{thm:BenoistApprox}
Let $X$ be a nonsingular compact real algebraic set of dimension $n$. Let $M\subset X$ be its smooth compact submanifold of dimension $d$, the bordism class of whose inclusion $[i:M\hookrightarrow X]$ is algebraic. Assume that $2d<n$. Then $M$ admits an algebraic approximation in $X$. 
\end{thm}
On the other hand, in the same paper Benoist was able to show that the assumption about the inequality $2d<n$ is essential:
\begin{thm}[{\cite[Theorem 0.7]{benoistSubvarietiesNonsingularReal2024}}]\label{thm:BenoistCounterex}
Let $c$ be a positive integer, with the property that $c+1$ has precisely two ones when written in binary. Let $d$ and $n$ be positive integers, such that $d\geq c$ and $n=d+c$. Then, there exist a nonsingular compact real algebraic set $X$ of dimension $n$ and a compact submanifold $M\subset X$ of dimension $d$, such that the following conditions are satisfied:
\begin{enumerate}
    \item the bordism class of the inclusion $[i:M\hookrightarrow X]$ is algebraic,
    \item if $j:Y\hookrightarrow X$ is the inclusion of a nonsingular algebraic subset of $X$ of dimension $d$, then $(Y,j)$ is \emph{not} bordant to $(M,i)$; in particular $M$ does not admit an algebraic approximation in $X$.
\end{enumerate}
\end{thm}

Despite this negative result, until now it has been unknown whether algebraicity of the bordism class of the inclusion of a submanifold is sufficient for the existence of a \emph{weak} algebraic approximation outside the dimension range $2d<n$. In \cite{kucharzConjecturesContinuousRational2016} it was conjectured that it is always sufficient, irrespective of the dimension. It is worth mentioning that a slightly weaker positive result is known: if the bordism class of the inclusion $[M\hookrightarrow X]$ is algebraic, then the manifold $M\times \{0\}$ admits a weak algebraic approximation in $X\times \R$ \cite[Theorem 2.10.8]{akbulutTopologyRealAlgebraic1992}.

In Section \ref{sec:obstructionsApproximation} we show that the invariants introduced in the current paper are capable of obstructing the existence of a weak algebraic approximation of submanifolds. This way, we are able to disprove the conjecture from \cite{kucharzConjecturesContinuousRational2016}:
\begin{thm}\label{thm:newCounterex}
There exist a nonsingular compact real algebraic set $X$ of dimension $6$ and a smooth compact submanifold $M$ of $X$ of dimension $4$, such that the following conditions are satisfied:
\begin{enumerate}
    \item the bordism class of the inclusion $[i:M\hookrightarrow X]$ is zero, yet
    \item $M$ does not admit a weak algebraic approximation in $X$.
\end{enumerate}
\end{thm}
Furthermore, this example for the first time shows that it is not enough to know the bordism class of the inclusion of a submanifold to determine whether it admits an algebraic approximation:
\begin{cor}
There exists a nonsingular compact real algebraic set $X$, and two smooth compact submanifolds $M,N\subset X$ of the same dimension satisfying the following properties:
\begin{enumerate}
    \item the inclusions $i:M\hookrightarrow X$ and $j:N\hookrightarrow X$ are bordant to each other,
    \item $M$ does not admit an algebraic approximation in $X$, but
    \item $N$ \emph{does} admit an algebraic approximation in $X$.
\end{enumerate}
\begin{proof}
Let $X$ and $M\subset X$ be as in Theorem \ref{thm:newCounterex}. It remains to find any nonsingular algebraic subset $N\subset X$ of codimension $c=\codim M$, with the property that the bordism class of the inclusion $[j:N\hookrightarrow X]$ is zero. This can be done as follows. Choose any regular map $f:X\rightarrow \R^c$ whose image has nonempty interior in $\R^c$. Applying Sard's theorem, without loss of generality we may assume that $0\in \R^c$ lies in the image of $f$ and is a regular value of $f$, and that $0\in \R^{c-1}$ is a regular value of the composition $\pi \circ f$, where $\pi:\R^c\rightarrow \R^{c-1}$ is the projection onto the first $c-1$ coordinates. Set $N:=f^{-1}(0)$. Then $N$ is a nonsingular algebraic subset of $X$. Moreover, the bordism class of its inclusion is zero, as it is the boundary of the smooth submanifold $C:=f^{-1}\{(0,\dots,0,t):t\geq 0\}$ of $X$.
\end{proof}
\end{cor}
\begin{rem}
The obstruction to weak algebraic approximation behind the construction of the example from Theorem \ref{thm:newCounterex}, which is in fact based on the classical functor $H^{\ast}_{\C}(X;\Z)$, fails to obstruct the existence of a weak algebraic approximation of a submanifold $M$ if the normal bundle to $M$ in $X$ is trivial. This leaves \cite[Conjecture B(p)]{kucharzConjecturesContinuousRational2016} open.

However, as explained before, the new stronger invariants introduced in this paper can be used in the role of such obstructions instead. In theory, this could lead to a disproof of \cite[Conjecture B(p)]{kucharzConjecturesContinuousRational2016}.
\end{rem}
\section{Preliminaries}\label{sec:preliminaries}
The paper is intended to be available both to real algebraic geometers and algebraic topologists. For this reason, we include two preliminary Subsections \ref{sec:preliminaries:real} and \ref{sec:preliminaries:cohomology}. In the former we introduce terminology and notation from real algebraic geometry. In the latter we provide a brief introduction to equivariant algebraic topology. 
\subsection{Preliminaries on real algebraic geometry}\label{sec:preliminaries:real}
There are essentially two interconnected approaches to real algebraic geometry. On the one hand, for many authors, real algebraic geometry is the study of real algebraic sets, that is, subsets of the affine space $\R^n$ defined by real polynomial equations. On the other hand, one can study the larger object consisting of all complex solutions of a set of polynomial equations with real coefficients. The paper will make use of both of these approaches, so it is essential that in this section we set up notation for both of these types of objects. This section is based on \cite[Section 2.4]{mangolteRealAlgebraicVarieties2020}, although our notation sometimes diverges from the one in \cite{mangolteRealAlgebraicVarieties2020}.

We begin by formally stating the following definition from the introduction:
\begin{defi}
A \emph{real algebraic set} $V$ is a subset of the space $\R^n$ for some $n\geq 0$ given as the zero set of a family of polynomials indexed by some index set $I$:
\begin{equation*}
    V=\{x\in \R^n:\forall_{i\in I} \;p_i(x)=0\},
\end{equation*}
where 
\begin{equation*}
    \{p_i\}_{i\in I} \subset \R[x_1,\dots,x_n].
\end{equation*}

The family of subsets of $V$ consisting of complements of algebraic sets contained in $V$ comprises a topology, called the Zariski topology.
\end{defi}
As explained in the introduction, the standard definition of a morphism in this setting is the following one:
\begin{defi}
Let $V\subset \R^n$ be a real algebraic set, and let $U\subset V$ be a Zariski open subset of $V$. A function $f:U\rightarrow \R$ is called \emph{regular} if there are polynomials $P,Q\in \R[x_1,\dots,x_n]$ such that $Q$ is different from zero at all points of $U$ and 
\begin{equation*}
    Q(x)f(x)=P(x)
\end{equation*}
holds for $x\in U$.

Given another algebraic set $W\subset \R^m$, a map $F:U\rightarrow W$ is said to be regular, if all its coordinate functions $F_1,\dots,F_m$ are regular.
\end{defi}

We can now introduce the core object of our study; a \emph{real affine variety} is essentially a real algebraic set without a preferred embedding in the affine space:
\begin{defi}\label{defi:realAffineVariety}
A real affine variety is a topological space $\var X$ endowed with a sheaf $\mathcal R_{\var X}$ of real-valued functions on $\var X$, isomorphic to a real algebraic set endowed with the Zariski topology and the sheaf of regular functions. The topology on $\var X$ is by analogy called the Zariski topology on $\var X$, and local sections of $\RR_{\var X}$ are called regular functions. By the dimension of $\var X$ we mean its dimension as a Noetherian topological space with the Zariski topology.
\end{defi}
Given a Zariski open subset $\var U$ of a real affine variety $\var X$, the ring $\RR_{\var X}(\var U)$ is usually denoted by $\RR(\var U)$ for short.
\begin{defi}
A regular map between two real affine varieties $\var X$ and $\var Y$ is a map $F:\var X\rightarrow \var Y$ continuous in the Zariski topology, with the property that for every Zariski open set $U\subset \var Y$ and every regular function $\varphi\in \RR(U)$, the composition $\varphi\circ F\vert_{F^{-1}(U)}$ is regular.
\end{defi}
It is verified easily that if $V\subset \R^n,W\subset \R^m$ are real algebraic sets, then a map $F:V\rightarrow W$ is regular as a map of real algebraic sets, if and only if it is regular as a map of real affine varieties.

Every real affine variety carries also the finer Euclidean topology, which coincides with the standard topology when $\var X=V\subset \R^n$ is a real algebraic set. The variety $\var X$ is said to be \emph{complete}, if it is compact in the Euclidean topology.

An example of a complete real affine variety which is most relevant to us is the unit sphere in $\R^{n+1}$, given by the single equation $\sum_i x_i^2=1$. We denote it by $\sph^n$.

The class of real affine varieties may at first seem too small to perform standard constructions from algebraic geometry, and it would seem natural to extend this notion and define quasi-projective varieties. Unlike in the complex case, this is unnecessary, as the following holds true:
\begin{prop}[{\cite[Theorem 3.4.4 and Proposition 3.2.10]{bochnakRealAlgebraicGeometry1998}}]\text{} 
\begin{enumerate}
    \item The real projective space $\R P^n$, endowed with the topology and the sheaf of real-valued functions obtained by gluing the local structures of real affine varieties on the affine charts is a real \emph{affine} variety in the sense as before. 
    \item If $\var X$ is a real affine variety and $U\subset \var X$ is a Zariski open set, then the pair $(U,\RR_{\var X}\vert_U)$ is a real \emph{affine} variety as well.
\end{enumerate}
Thus, every real quasi-projective variety is automatically affine.
\end{prop}

The notion of nonsingularity of a real affine variety is defined as follows:
\begin{defi}
Let $\var X$ be a real affine variety and let $x\in \var X$ be its point. We say that $\var X$ is nonsingular at $x$, if the stalk of $\RR_{\var X}$ at $x$ is a regular local ring of dimension equal to the dimension of $\var X$. We say that $\var X$ is nonsingular if it is nonsingular at each of its points. 
\end{defi}
\begin{rem}
Equivalently, $\var X$ is nonsingular in this sense if and only if its irreducible components are nonsingular, pairwise disjoint, and all have the same dimension.
\end{rem}
A real algebraic set which is nonsingular as a real affine variety is an analytic submanifold of $\R^n$. Unlike in the complex case, however, the converse is not true in general.

Let us now start introducing a different point of view, which accounts for complex points of real algebraic sets. We take a classical approach to complex algebraic geometry, which identifies an algebraic variety with its space of closed points. We also restrict ourselves to the class of quasi-projective varieties. Therefore, a quasi-projective complex variety for us is what is left of a (not necessarily irreducible) Zariski locally closed subset of the projective space after forgetting the particular choice of the embedding. A formal definition could be given following the lines of Definition \ref{defi:realAffineVariety}, but there is no point in repeating the details. A quasi-projective complex variety is said to be nonsingular if its irreducible components are smooth, pairwise disjoint, and all have the same dimension.

\begin{defi}
A quasi-projective $\R$-variety $\Rvar X$ is a pair $(\Rvar X(\C),\sigma)$, where $\Rvar X(\C)$ is a complex quasi-projective variety $\Rvar X(\C)$ and $\sigma:\Rvar X(\C)\rightarrow \Rvar X(\C)$ is an involution continuous in the Zariski topology, which is anti-regular, in the sense that for every open set $U\subset \Rvar X(\C)$ and every regular function $\varphi\in \Rvar O_{\Rvar X(\C)}(U)$, the function $\bar{\varphi}\circ \sigma$ is regular on $\sigma^{-1}(U)$, where the bar denotes complex conjugation.

A morphism (or a regular map) between two quasi-projective $\R$-varieties $\Rvar X$ and $\Rvar Y$ is a regular map between the underlying complex varieties, which is equivariant with respect to the $\mathbb Z/2$-actions induced by the involutions.

A quasi-projective $\R$-variety is said to be nonsingular, if the underlying complex quasi-projective variety is nonsingular. Similarly, it is said to be projective (resp. affine) if the underlying complex variety is projective (resp. affine).
\end{defi}
We will denote $\R$-varieties using calligraphic letters $\Rvar X,\Rvar Y,\dots$, not to confuse them with real affine varieties denoted using standard italic letters $\var X,\var Y,\dots$.

The following nontrivial result could be considered as an alternative definition of a quasi-projective  $\R$-variety:
\begin{prop}[{\cite[Theorem 2.1.33]{mangolteRealAlgebraicVarieties2020}}]\label{prop:R-varietyEmbedding}
Let $\Rvar X$ be a quasi-projective $\R$-variety. Then, there exists a Zariski locally closed set $V\subset \C P^n$ for some $n\geq 0$, which is invariant under the complex conjugation map $\sigma:\C P^n\rightarrow \C P^n$, and such that $(V,\sigma\vert_V)$ is isomorphic to $\Rvar X$.
\end{prop}

Given a quasi-projective $\R$-variety $\Rvar X$, we denote by $\Rvar X(\C)$ the underlying complex quasi-projective variety. It is a consequence of Proposition \ref{prop:R-varietyEmbedding}, that the set of fixed points of the involution carries a natural structure of a real affine variety, which we denote by $\Rvar X(\R)$. 

Given an $\R$-variety $\Rvar X$ and a Zariski open subset $U\subset \Rvar X$ invariant under $\sigma$, we denote by $\mathcal O_{\Rvar X}(U)$ the ring of morphisms of $\R$-varieties $f: U\rightarrow \C$, where $\C$ is endowed with the complex conjugation. In other words, $\mathcal O_{\Rvar X}(U)$ consists of regular maps $f:\Rvar U(\C)\rightarrow \C$ which additionally satisfy $f\circ \sigma =\bar f$.

It is rather obvious that every real affine variety $\var X$ arises as the set of real points $\Rvar X(\R)$ of some quasi-projective $\R$-variety $\Rvar X$. Hironaka's resolution of singularities theorem implies the following more interesting fact:
\begin{thm}[{\cite[Theorem 2.3.7]{mangolteRealAlgebraicVarieties2020}}]
Let $\var X$ be a nonsingular complete real affine variety. Then, there exists a nonsingular projective $\R$-variety $\Rvar X$, with the property that $\Rvar X(\R)$ is Zariski dense in $\Rvar X(\C)$ and isomorphic to $\var X$.

Such a pair $(\Rvar X,\varphi)$, where $\varphi:\var X\rightarrow \Rvar X(\R)$ is a biregular isomorphism, is called a \emph{nonsingular projective complexification of $\var X$}.
\end{thm}
We often omit the isomorphism $\varphi$ in the notation and identify $\var X$ with $\Rvar X(\R)$. Such a complexification is by no means unique; if $\var X$ is of dimension at least two then blowing up a conjugate pair of complex points of $\Rvar X$ gives rise to a different complexification of $\var X$. The main goal of this paper is to remedy this by introducing a family of invariants of real affine varieties, which are defined in terms of their complexifications, but which do not depend on the particular choices of them.

The following proposition relating regular maps between real affine varieties and regular maps between their complexifications is crucial for a construction we will perform:
\begin{prop}\label{prop:resolutionOfComplexification}
Let $\var X$ and $\var Y$ be two nonsingular complete real affine varieties. Let $\Rvar X$ and $\Rvar Y$ be some nonsingular projective complexifications of $\var X$ and $\var Y$ respectively. Let $f:\var X\rightarrow \var Y$ be a regular map. Then, there exists a nonsingular projective complexification $\Rvar X'$ of $\var X$ and two regular maps of $\R$-varieties $\pi:\Rvar X'\rightarrow \Rvar X$ and $g:\Rvar X'\rightarrow \Rvar Y$ such that the diagram
\begin{center}
\begin{tikzcd}
\Rvar X & \Rvar X' \arrow[l, "\pi"'] \arrow[r, "g"]                  & \Rvar Y                     \\
  & X \arrow[r, "f"] \arrow[lu, hook] \arrow[u, hook] & Y \arrow[u, hook]
\end{tikzcd}
\end{center}
is commutative. Moreover, there is a Zariski open neighbourhood $U$ of $\var X$ in $\Rvar X$ invariant under the involution such that $\pi\vert_{\pi^{-1}(U)}$ is an isomorphism onto $U$.
\begin{proof}
The map $f$, by definition being a quotient of polynomials, extends to a regular map $h:U\rightarrow \Rvar Y$, where $U$ is a Zariski open neighbourhood of $\var X$ in $\Rvar X$ invariant under involution. Using the theorem on resolution of indeterminacy of rational maps we can apply a sequence of equivariant blowing-ups on $\Rvar X$, whose centres do not intersect $U$, to obtain a nonsingular projective $\R$-variety $\Rvar X'$ together with a morphism $\pi:\Rvar X'\rightarrow \Rvar X$, with the property that $\pi\vert_{\pi^{-1}(U)}$ is an isomorphism onto $U$ and that $h \circ \pi\vert_{\pi^{-1}(U)}$ extends to a regular map $g:\Rvar X'\rightarrow \Rvar Y$. The triple $(\Rvar X',\pi,g)$ satisfies the desired properties.
\end{proof}
\end{prop}

\subsection{Preliminaries on equivariant algebraic topology}\label{sec:preliminaries:cohomology}
In this section we review the theory of equivariant and Real-oriented spectra. A general classical reference on equivariant topology is the book \cite{mayEquivariantHomotopyCohomology1996}. The current introduction focuses on more specific topics, which are explained more thoroughly in \cite[Section 2]{huRealorientedHomotopyTheory2001}. Another useful reference using more classical terminology is \cite{arakiOrientationsTcohomologyTheories1979}.

A \emph{Real space} is a pair $\Rvar T=(\Rvar T(\C),\sigma)$, where $T(\C)$ is a Hausdorff topological space and $\sigma$ is a continuous involution acting on $\Rvar T$. It is a standard practice to capitalize the word \say{Real} to avoid confusion, especially once we introduce \emph{Real vector bundles}. A prime example of a Real space is the underlying topological space $\Rvar X(\C)$ of an $\R$-variety $\Rvar X$ with the endowed involution. A subset of a Real space $\Rvar T$, which is invariant under $\sigma$ naturally carries an induced structure of a Real space. A morphism between two Real spaces is a continuous map equivariant with respect to the $\Z/2$-action of the involutions. A based Real space is a Real space with a chosen base point lying in the fixed point locus of the involution (of course, in general it may happen that the locus is empty). Given a Real space $\Rvar T$, by $\Rvar T_+$ we denote the disjoint union of $\Rvar T$ and a base point. Many standard topological constructions, most importantly including the smash product, can be naturally carried over in the category of based Real spaces.

Given two nonnegative integers $p$ and $q$, we denote by $\R^{p+q\tau}$ the Real space with underlying topological space $\R^{p+q}$, and the involution given by 
\begin{equation*}
    \sigma(x_1,\dots,x_{p+q})=(x_1,\dots,x_p,-x_{p+1},\dots, -x_{p+q}).
\end{equation*}
We denote by $B^{p+q\tau}$ the closed unit ball centred at the origin (in the Euclidean metric) in $\R^{p+q\tau}$ with the induced Real structure, and by $\sph^{p+q\tau}$ the based Real space arising as the topological quotient of $B^{p+q\tau}$ by its boundary. Alternatively, $\sph^{p+q\tau}$ is isomorphic to the unit sphere in $\R^{p+1+q\tau}$. Whenever $q=0$, we may write $\sph^p$ instead of $\sph^{p+0\tau}$.

We shall remark here that there are different conventions concerning the index $p+q\tau$. Some authors concerned with the study of $\R$-varieties (see for example \cite{dossantosBigradedEquivariantCohomology2009}) adopt the motivic notation, whose index $(p+q,q)$ corresponds to our $p+q\tau$. In Atiyah's notation from \cite{atiyahTheoryReality1966}, the sphere $\sph^{p,q}$ corresponds to our $\sph^{q-1+p\tau}$. We choose to stick with notation parallel to the one from \cite{huRealorientedHomotopyTheory2001}, which is more common in algebraic topology.

The $(p+q\tau)$-th suspension of a based Real space $\Rvar T$ by definition is its smash product with the sphere $\sph^{p+q\tau }$:
\begin{equation*}
    \Sigma^{p+q\tau} \Rvar T:=\Rvar T\wedge \sph^{p+q\tau}.
\end{equation*}
One can provide a similar definition for the loop space:
\begin{equation*}
    \Omega^{p+q\tau} \Rvar T:=F(\sph^{p+q\tau},\Rvar T),
\end{equation*}
where the right hand side denotes the space of based equivariant maps from the sphere to $\Rvar T$ with suitable topology and involution.

In equivariant topology, one can develop the theory of spectra and generalised cohomology theories similarly as it is done in the classical case. There are different ways of doing that, but the one which interests us concerns \emph{$\Z/2$-equivariant cohomology theories indexed over the complete universe}. We call such cohomology theories \emph{equivariant cohomology theories} for short. Roughly speaking such an equivariant cohomology theory is a collection of functors from the category of based Real spaces to the category of abelian groups numbered by all symbols of the form $p+q\tau$ for $p,q\in \Z$, subject to certain conditions. The conditions are analogous to the usual Eilenberg–Steenrod axioms for a reduced generalised cohomology theory, except for the fact that the suspension axiom is replaced by the following strengthening: there are natural isomorphisms
\begin{equation*}
    \widetilde E^{\star+r+s\tau}(\Sigma^{r+s\tau} \Rvar T)\cong \widetilde E^{\star}(\Rvar T)
\end{equation*}
for all $r,s\in \Z$. 
\begin{rem}
Throughout the paper we use the following convention from \cite{huRealorientedHomotopyTheory2001}: when an index ranges through all pairs of the form $p+q\tau$ it is denoted by $\star$. When the index ranges through $\Z$, the standard symbol $\ast$ is used instead. Thus, if $\Rvar T$ is a Real space and $\widetilde E^\star$ is an equivariant cohomology theory, then $\widetilde E^\star(\Rvar T)$ is the doubly graded group
\begin{equation*}
    \widetilde E^\star(\Rvar T):=\bigoplus_{p,q\in \Z}\widetilde E^{p+q\tau}(\Rvar T),
\end{equation*}
while $\widetilde E^\ast(\Rvar T)$ is the $\Z$-graded subgroup 
\begin{equation*}
    \widetilde E^\ast(\Rvar T):=\bigoplus_{p\in \Z}\widetilde E^{p+0\tau}(\Rvar T)    
\end{equation*}
of $\widetilde E^\star(\Rvar T)$.
\end{rem}
We will not be giving a more precise definition of an equivariant cohomology theory, which in a broader context can be found in \cite[Chapter XIII]{mayEquivariantHomotopyCohomology1996}. In the equivariant setting, every equivariant cohomology theory is representable by a (suitably defined) equivariant spectrum. As usual, from a reduced equivariant cohomology theory one defines a nonreduced one by taking $E^\star(\Rvar T):=\widetilde{E}^\star(\Rvar T_+)$ for a Real space $X$ and $E^\star(\Rvar T,\Rvar S):=\widetilde{E}^\star(\cofib(\Rvar S\hookrightarrow \Rvar T))$ for a pair of Real spaces $(\Rvar T,\Rvar S)$. For such a pair, there arises a family of long exact sequences, one for every $q\in \Z$:
\begin{equation*}
    \dots \rightarrow F^{p+q\tau}(\Rvar T,\Rvar S)\rightarrow F^{p+q\tau}(\Rvar T)\rightarrow F^{p+q\tau}(\Rvar S)\rightarrow F^{p+1+q\tau}(\Rvar T, \Rvar S)\rightarrow \dots
\end{equation*}

What interests us are equivariant cohomology theories with an additional feature called a \emph{Real orientation}. To introduce it, let us first recall the classical notion of a \emph{complex-oriented cohomology theory}.

\begin{defi}
A complex orientation on a (nonequivariant) multiplicative cohomology theory $E^\ast$ is a choice of a class $c_1^E\in \widetilde E^2(\C P^\infty)$, which maps to the unit $1\in E^0(\pt)$ under the composition
\begin{equation*}
    \widetilde E^2(\C P^\infty)\xrightarrow{i^\ast} \widetilde E^2(\sph^2)\cong\widetilde E^0(\sph^0)\cong E^0(\pt),
\end{equation*}
where $i:\sph^2\cong \C P^1\hookrightarrow \C P^\infty$ is the inclusion, and $\widetilde E^2(\sph^2)$ is identified with $\widetilde E^0(\sph^0)$ through the suspension isomorphism. The theory $E^\ast$ is called \emph{complex-orientable}, if it admits some complex orientation. It is called \emph{complex-oriented} after one fixes a particular complex orientation on $E^\ast$. 
\end{defi}
The most important example of a complex-oriented cohomology theory is complex cobordism $MU^\ast$. It satisfies the following universal property: there is a one-to-one correspondence between complex orientations on a cohomology theory $E^\ast$ and morphisms of ring spectra $MU \rightarrow E$. 

Recall, that a real vector bundle $\xi$ of dimension $d$ over a base space $B$ is $E$-\emph{oriented}, if there is given a class $u\in E^d(\xi,\xi\backslash B)$ whose restriction to every fibre is a generator of the free $E^0(\pt)$-module 
\begin{equation*}
    E^d(\R^d,\R^d\backslash\{0\})\cong \widetilde E^d(\sph^d)\cong E^0(\pt),
\end{equation*}
where $B$ is identified with the zero section of $\xi$. The class $u$ is called a \emph{Thom class} of $\xi$. Of course, in general a given vector bundle may be oriented in multiple ways.

The essential property of complex-oriented cohomology theories which interests us is that in a complex-oriented cohomology theory $E$ every vector bundle with complex structure over a compact base is $E$-oriented; in fact the Thom class is naturally given as the pullback of the universal Thom class of the tautological bundle over $B\mathrm{U}(n)$. From this, it follows that every compact complex manifold $M$ is $E$-oriented (see for example \cite[Part III, Section 10]{adamsStableHomotopyGeneralised1974}). This gives a fundamental class $[M]\in E^{2d}(M)$ (where $d$ is the complex dimension of $M$). The cap product with $[M]$ induces the Poincar\'e duality isomorphism:
\begin{equation*}
    E^\ast(M)\xrightarrow{\frown [M]}E_{2d-\ast} (M).
\end{equation*}
We refer the reader to \cite{adamsStableHomotopyGeneralised1974} for a broader introduction on complex-oriented cohomology theories. 

Let us now go back to the equivariant case. To introduce the analogous definition of a \emph{Real-oriented equivariant cohomology theory}, we first need to consider an equivariant version of a complex structure on a vector bundle:
\begin{defi}
Given a Real space $\Rvar T$, a \emph{Real vector bundle} over $\Rvar T$ is a complex vector bundle $p:F\rightarrow  \Rvar T$ endowed with an involution of the total space $\sigma_F:F\rightarrow F$, which satisfies
\begin{equation*}
    p\circ \sigma_F=\sigma \circ p,
\end{equation*}
and such that when restricted to the fibre $p^{-1}(x)$ of a point $x\in \Rvar T$, it is an anti-linear isomorphism onto $p^{-1}(\sigma(x))$. Morphisms of Real vector bundles are defined in the obvious manner.
\end{defi}
The classifying space of Real line bundles is the infinite complex projective space $\C P^\infty$ with involution induced by the complex conjugation on $\C^{\infty}$. From now on, whenever we consider the finite $\C P^n$ or infinite $\C P^\infty$ projective spaces as Real spaces, they are assumed to be endowed with this involution. 

With these notions in mind we can define Real orientations:
\begin{defi}
A \emph{Real orientation} on an equivariant multiplicative cohomology theory $E^\star$ is a choice of a class $c_1^E\in \widetilde E^{1+\tau}(\C P^\infty)$, which maps to the unit $1\in E^0(\pt)$ under the composition
\begin{equation*}
    \widetilde E^{1+\tau}(\C P^\infty)\xrightarrow{i^\ast} \widetilde E^{1+\tau}(\sph^{1+\tau})\cong\widetilde E^0(\sph^0)\cong E^0(\pt),
\end{equation*}
where $i:\sph^{1+\tau}\cong \C P^1\hookrightarrow \C P^\infty$ is the inclusion, and $\widetilde E^{1+\tau}(\sph^{1+\tau})$ is identified with $\widetilde E^0(\sph^0)$ through the suspension isomorphism. The theory $E^\star$ is called \emph{Real-orientable} if it admits some Real orientation. It is called \emph{Real-oriented} after one fixes a particular Real orientation on $E^\star$. 
\end{defi}
An example of a Real-oriented cohomology theory is the \emph{Real cobordism} $M\R^\star$ of Landweber \cite{landweberConjugationsComplexManifolds1968a}. It is universal in a similar sense as before: Real orientations on an equivariant cohomology theory $E^\star$ correspond to morphisms of equivariant ring spectra $M\R\rightarrow E$. 

As in the complex-oriented case, in the Real-oriented case every Real vector bundle of (complex) dimension $d$ over a compact space carries a natural Thom class in degree $d+d\tau$, given as the pullback of the universal Thom class of the tautological bundle over $B\mathrm U(n)$ with suitable involution \cite[Theorem 2.25]{huRealorientedHomotopyTheory2001}. 

Let $\Rvar X$ be a nonsingular projective $\R$-variety of (complex) dimension $d$. Locally near each of its real points $\Rvar X$ is diffeomorphic to the open unit ball in $\R^{d+d\tau}$, hence $\Rvar X$ is an $\R^{d+d\tau}$-manifold in the sense of \cite[Chapter X, Definition 2.5]{mayEquivariantHomotopyCohomology1996}. As the tangent bundle to $\Rvar X$ is a Real vector bundle, it follows that $\Rvar X$ is oriented in the sense of \cite[Chapter XVI, Definition 9.3]{mayEquivariantHomotopyCohomology1996}. This means that it carries a fundamental class $[\Rvar X]\in E_{d+d\tau}(\Rvar X)$, the cap product with which induces the Poincar\'e duality isomorphism
\begin{equation*}
    E^\star(\Rvar X)\xrightarrow{\frown [\Rvar X]} E_{d+d\tau-\star}(\Rvar X).
\end{equation*}

The Real-oriented cohomology theory which we make most use of is Real K-theory of Atiyah \cite{atiyahTheoryReality1966}, which we introduce now. Let $\Rvar T$ be a Real space. Isomorphism classes of Real vector bundles over $\Rvar T$ form a semigroup under the Whitney sum. Similarly as it is done in nonequivariant K-theory, by assigning a formal inverse to each of its elements we obtain a group denoted by $K\R^0(\Rvar T)$. This turns out to be the zeroth cohomology group in a certain Real-oriented equivariant cohomology theory, which we denote by $K\R^\star$ \cite[Theorem 2.8]{huRealorientedHomotopyTheory2001}. This theory is doubly periodic with periods $1+\tau$ and $8$. That is, $K\R^{p+q\tau}(\Rvar T)$ is naturally isomorphic with $K\R^{p+1+(q+1)\tau}(\Rvar T)$ and with $K\R^{p+8+q\tau}(\Rvar T)$.

If $\Rvar T$ is endowed with the trivial involution, the isomorphism class of a Real bundle over $\Rvar T$ is determined by its fixed-point real subbundle. Therefore, in this case one has the natural isomorphism $K\R^{p+q\tau}(\Rvar T)\cong KO^{p-q}(\Rvar T)$ for $p,q\in \Z$. 
\subsection{Coefficient rings in K-theory}\label{sec:preliminaries:coefficients}
Since we will perform explicit calculations in K-theory, we need to recall how the coefficient rings $KU^\ast(\pt),KO^\ast(\pt),K\R^\star(\pt)$ and the coefficient group $KSp^\ast(\pt)$ look like, and fix notation for their generators. All the results of this subsection can be found in \cite[Chapter 1]{strickland1992bott} or \cite{atiyahTheoryReality1966}.
\begin{prop}\label{prop:coefficientRings}
In the case of unitary K-theory, one has
\begin{equation*}
    KU^\ast(\pt)=\Z[\beta,\beta^{-1}],
\end{equation*}
where $|\beta|=-2$. In the case of orthogonal K-theory, the coefficient ring has the more complicated structure
\begin{equation*}
    KO^\ast(\pt)=\Z[\eta,\alpha,\lambda,\lambda^{-1}]/(2\eta,\eta^3,\alpha\eta,\alpha^2-4\lambda),
\end{equation*}
where $|\eta|=-1,|\alpha|=-4,|\lambda|=-8$. The doubly-graded ring $K\R^\star(\pt)$ can be described as 
\begin{equation*}
    K\R^\star(\pt)=KO^\ast(\pt)[\sigma,\sigma^{-1}]=\Z[\eta,\alpha,\lambda,\lambda^{-1},\sigma,\sigma^{-1}]/(2\eta,\eta^3,\alpha\eta,\alpha^2-4\lambda),    
\end{equation*}
where $|\eta|=-1,|\alpha|=-4,|\lambda|=-8,|\sigma|=-1-\tau$.
\end{prop}

There is a natural transformation $r:KU^\ast\rightarrow K\R^\ast$ between the functors
\begin{equation*}
    \Rvar T\mapsto KU^\ast(\Rvar T(\C)),\quad \Rvar T\mapsto K\R^\ast(\Rvar T)
\end{equation*}
from the category of Real spaces $\Rvar T=(\Rvar T(\C),\sigma)$ to $\Z$-graded groups. On the level of vector bundles it can be described as follows: to a complex bundle $F$ over $\Rvar T(\C)$ it associates the bundle $E:=F\oplus \sigma^\ast\bar F$, with the anti-linear involution $\sigma_E$ permuting the two summands \cite[p. 371]{atiyahTheoryReality1966}. When restricted to the subcategory of spaces with trivial involution, this is the usual transformation $r:KU^\ast\rightarrow KO^\ast$, induced by forgetting the complex structure on a vector bundle.

\begin{prop}\label{prop:realificationImage}
The induced map of rings
\begin{equation*}
    KU^\ast(\pt)\rightarrow KO^\ast(\pt)
\end{equation*}
is \emph{not} a ring homomorphism; it is only a homomorphism of graded groups, which can be described as follows:
\begin{equation*}
    r(\beta^n)=\begin{cases}
        2\lambda^{k} & \text{if }n=4k,\\
        \eta^2\lambda^{k} & \text{if }n=4k+1,\\
        \alpha\lambda^{k} & \text{if }n=4k+2,\\
        0 & \text{if }n=4k+3.
    \end{cases}
\end{equation*}

In particular, the image of this map is the ideal of $KO^\ast(\pt)$ generated by $2,\eta^2,\alpha$.
\end{prop}

Unlike $KO$ and $KU$, quaternionic K-theory $KSp$ is not a multiplicative cohomology theory, since due to noncommutativity of the quaternions one cannot tensor two $\Hq$-vector bundles to get a $\Hq$-bundle. However, since the tensor product $\otimes_\R:\R\times  \Hq\rightarrow \Hq$ is defined, there is a map of spectra $KO\wedge KSp\rightarrow KSp$ turning $KSp$ into a $KO$-module. The coefficient group of $KSp$ can then be described as follows:
\begin{prop}\label{prop:KOtoKSp}
There is an element $\theta\in KSp^{-4}(\pt)$ (coming from the tautological $\Hq$-vector bundle over $\Hq P^1\cong \sph^4$), with the property that multiplication by $\theta$ induces a natural isomorphism of functors
\begin{equation*}
    KO^\ast(-)\cong KSp^{\ast-4}(-).
\end{equation*}
In particular, we have that $KSp^\ast(\pt)$ is the free $KO^\ast(\pt)$-module with one generator $\theta$, with $|\theta|=-4$.
\end{prop}

There is another product $KSp\wedge KSp\rightarrow KO$. On the level of vector bundles it can be described as follows: since $\Hq$ is isomorphic with $\Hq^{\mathrm{op}}$ via the conjugation, to (left) $\Hq$-vector bundles $E$ and $F$ one can associate the $\R$-bundle
\begin{equation*}
    E^{\mathrm{op}} \otimes_\Hq F,
\end{equation*}
where $E^{\mathrm{op}}$ is the right $\Hq$-bundle induced by the isomorphism $\Hq\cong \Hq^{\mathrm{op}}$. The following holds true:
\begin{prop}\label{prop:KSptoKO}
One has $\theta^2=\lambda \in KO^{-8}(\pt)$. Since $\lambda$ is invertible, it follows that multiplication by $\theta$ induces also the natural isomorphism
\begin{equation*}
    KSp^\ast(-)\cong KO^{\ast-4}(-).
\end{equation*}
\end{prop}

Besides the realification map, we will make use also of the transformation $h:KU^\ast\rightarrow KSp^\ast$ induced by treating $\Hq$ as a right $\C$-module and tensoring a complex vector bundle with $\Hq$. 

\begin{prop}\label{prop:quaternionification}
On the coefficient groups the transformation is described as follows:
\begin{equation*}
    h(\beta^n)=\begin{cases}
        \lambda^{k-1}\alpha\theta & \text{if }n=4k,\\
        0 & \text{if }n=4k+1,\\
        2\lambda^{k}\theta & \text{if }n=4k+2,\\
        \eta^2\lambda^k\theta & \text{if }n=4k+3.
    \end{cases}
\end{equation*}
In particular, the image of the composition
\begin{equation*}
    KU^\ast(\pt)\xrightarrow{h} KSp^\ast(\pt)\xrightarrow{\cdot\theta} KO^{\ast-4}(\pt)
\end{equation*}
is again the ideal $(2,\eta^2,\alpha)$.
\end{prop}
\section{Invariants of real affine varieties}\label{sec:invariants}
In this section we define new invariants of real affine varieties based on complex- and Real-oriented cohomology theories, and derive some of their basic properties.
\subsection{Definitions}
The definitions will be based on the following theorem:
\begin{thm}\label{thm:morphismExists}
Let $\var{X}$ and $\var{Y}$ be two nonsingular complete real affine varieties and let $\Rvar{X}$ and $\Rvar{Y}$ be some nonsingular projective complexifications of $\var{X}$ and $\var{Y}$ respectively. Denote the inclusions by $i:\var{X}\hookrightarrow \Rvar{X}$ and $j:\var{Y}\hookrightarrow \Rvar{Y}$. Let $f:\var{X}\rightarrow \var{Y}$ be a regular map.

Let $E^\ast$ be a complex-oriented cohomology theory. Then, there is a homomorphism of $E^\ast(\pt)$-modules $\varphi:E^\ast(\Rvar{Y}(\C))\rightarrow E^\ast(\Rvar{X}(\C))$ making the diagram below commutative:
\begin{center}
\begin{tikzcd}
E^\ast(\Rvar{X}(\C)) \arrow[d, "i^\ast"] & E^\ast(\Rvar{Y}(\C)) \arrow[d, "j^\ast"] \arrow[l, "\varphi"'] \\
E^\ast(\var{X})                      & E^\ast(\var{Y}) \arrow[l, "f^\ast"']                      
\end{tikzcd}
\end{center}

Similarly, let $F^\star$ be a Real-oriented equivariant cohomology theory. Then, there is a homomorphism of $F^\star(\pt)$-modules $\psi:F^\star(\Rvar{Y})\rightarrow F^\star(\Rvar{X})$ making the diagram below commutative:
\begin{center}
\begin{tikzcd}
F^\star(\Rvar{X}) \arrow[d, "i^\ast"] & F^\star(\Rvar{Y}) \arrow[d, "j^\ast"] \arrow[l, "\psi"'] \\
F^\star(\var{X})                  & F^\star(\var{Y}) \arrow[l, "f^\ast"']               
\end{tikzcd}
\end{center}
Here $\Rvar{X}$ and $\Rvar{Y}$ are to be understood as Real spaces with the corresponding conjugations, while $\var{X}$ and $\var{Y}$ are to be understood as Real spaces with the trivial involution.
\begin{proof}
We consider only the case of a Real-oriented cohomology theory; the other case is similar. Denote the real dimension of $\var{X}$ (and hence the complex dimension of $\Rvar{X}(\C)$) by $d$.

Applying Proposition \ref{prop:resolutionOfComplexification} we obtain a nonsingular projective complexification $\Rvar{X}'$ of $\var{X}$ and two regular maps of $\R$-varieties $\pi:\Rvar{X}'\rightarrow \Rvar{X}$ and $g:\Rvar{X}'\rightarrow \Rvar{Y}$ such that the diagram
\begin{center}
\begin{tikzcd}
\Rvar{X} & \Rvar{X}' \arrow[l, "\pi"'] \arrow[r, "g"]                  & \Rvar{Y}                     \\
  & \var{X} \arrow[r, "f"] \arrow[lu, hook] \arrow[u, hook] & \var{Y} \arrow[u, hook]
\end{tikzcd}
\end{center}
is commutative. We claim that the diagram
\begin{center}
\begin{tikzcd}
F^\star(\Rvar{X}) \arrow[rd] & F^\star(\Rvar{X}') \arrow[d] \arrow[l, "\pi_!"'] & F^\star(\Rvar{Y}) \arrow[d] \arrow[l, "g^\ast"'] \\
                     & F^\star(\var{X})                             & F^\star(\var{Y}) \arrow[l, "f^\ast"']     
\end{tikzcd}
\end{center}
is also commutative, so that we can define $\psi:=\pi_!\circ g^\ast$. Here $\pi_!$ is the Gysin homomorphism, defined as
\begin{equation*}
    \pi_!=D_{\Rvar{X}}^{-1}\circ \pi_\ast\circ D_{\Rvar{X}'},
\end{equation*}
where
\begin{equation*}
    D_{\Rvar{X}}:F^\star(\Rvar{X})\rightarrow F_{d+d\tau-\star} (\Rvar{X}),\quad D_{\Rvar{X}'}:F^\star(\Rvar{X}')\rightarrow F_{d+d\tau-\star} (\Rvar{X}')
\end{equation*}
are the Poincar\'{e} duality isomorphisms induced by taking cap product with the fundamental classes. 

What requires verification is commutativity of the left triangle. To prove it, let us start by choosing a smooth function $\xi:\Rvar{X}\rightarrow \R^{\geq 0}$ with zero set equal to $\var{X}$. After substituting $\xi+\xi\circ \sigma$ for $\xi$ we may assume that it satisfies $\xi\circ \sigma=\xi$, where $\sigma$ is the conjugation on $\Rvar{X}$. Let $U$ be the neighbourhood of $\var{X}$ from Proposition \ref{prop:resolutionOfComplexification}. Using Sard's theorem, we find a small real number $\varepsilon$ such that $M:=\xi^{-1}([0,\varepsilon])$ is a smooth compact manifold with boundary, contained in $U$. By construction it is invariant under $\sigma$ and contains $\var{X}$. 

As explained in Section \ref{sec:preliminaries:cohomology}, the manifolds $\Rvar{X}$ and $\Rvar{X}'$ carry fundamental classes $[\Rvar{X}]\in F_{d+d\tau}(\Rvar{X})$ and $[\Rvar{X}']\in F_{d+d\tau}(\Rvar{X}')$. The manifold $M$ is a $\R^{d+d\tau}$-manifold with boundary, and its tangent bundle is the restriction of the $F$-oriented tangent bundle to $U$. It follows that $M$ is also oriented in the sense of \cite[Chapter XVI, Definition 9.3]{mayEquivariantHomotopyCohomology1996}, and it carries a fundamental class $[M,\partial M]\in F_{d+d\tau}(M,\partial M)$. 

Define $A:=\Rvar{X}\backslash \inte M$ and $B:=\Rvar{X}'\backslash \pi^{-1}(\inte M)$. Denote by $[\Rvar{X},A]$ and $[\Rvar{X}',B]$ the images of $[\Rvar{X}]$ and $[\Rvar{X}']$ in $F_{d+d\tau}(\Rvar{X},A)$ and $F_{d+d\tau}(\Rvar{X}',B)$ respectively. Since the tangent bundle to $U$ is the restriction of the tangent bundle to $\Rvar{X}$ it follows that the class $[\Rvar{X},A]$ maps to the fundamental class $[M,\partial M]$ under the excision isomorphism
\begin{equation*}
    F_{d+d\tau}(\Rvar{X},A)\cong F_{d+d\tau}(M,\partial M).
\end{equation*}
Identifying $U$ with $\pi^{-1}(U)$ by the same argument we obtain that the class $[\Rvar{X}',B]$ maps to $[M,\partial M]$ under the isomorphism 
\begin{equation*}
    F_{d+d\tau}(\Rvar{X}',B)\cong F_{d+d\tau}(M,\partial M).
\end{equation*}

Using the projection formula, one verifies that the following diagram is commutative:
\begin{center}
\begin{tikzcd}
F_{d+d\tau-\star}(\Rvar{X})  \arrow[d]                                                                            &                                                                    & F_{d+d\tau-\star}(\Rvar{X}') \arrow[ll,"\pi_\ast"'] \arrow[d]                                                                                                 \\
{F_{d+d\tau-\star}(\Rvar{X},A)}                                                                                   &                                                                    & {F_{d+d\tau-\star}(\Rvar{X}',B)} \arrow[ll]                                                                                                   \\
                                                                                                             & {F_{d+d\tau-\star}(M,\partial M)} \arrow[lu, "\cong"] \arrow[ru, "\cong"'] &                                                                                                                        \\
F^\star(\Rvar{X}) \arrow[uuu, "{\frown [\Rvar{X}]}"', bend left=70] \arrow[uu, "{\frown[\Rvar{X},A]}"'] \arrow[r] \arrow[rd] & F^\star(M) \arrow[u, "{\frown[M,\partial M]}"'] \arrow[d]              & F^\star(\Rvar{X}') \arrow[uuu, "{\frown [\Rvar{X}']}"', bend right=70] \arrow[uu, "{\frown [\Rvar{X}',B]}"'] \arrow[l] \arrow[ld] \\
                                                                                                             & F^\star(\var{X})                                                           &                                                                                                                       
\end{tikzcd}
\end{center}
The two arrows in the middle are isomorphisms due to excision. The arrows induced by taking cap products with $[\Rvar{X}],[\Rvar{X}']$ and $[M,\partial M]$ are isomorphisms as well thanks to Poincar\'{e}--Lefschetz duality (see \cite[Chapter XVI, Definition 9.3]{mayEquivariantHomotopyCohomology1996}). Chasing an element from $F^\star(\Rvar{X}')$ around the diagram one verifies that it maps to the same element of $F^\star(\var{X})$ regardless of whether we go directly or through $\pi_!$. This finishes the proof.
\end{proof}
\end{thm}
From the result we obtain the following:
\begin{cor}\label{cor:inclusion}
Let $\var{X}$ be a nonsingular complete real affine variety. Let $\Rvar{X}$ be a nonsingular projective complexification of $\var{X}$ and let $i:\var{X}\rightarrow \Rvar{X}$ be the inclusion. Let $E^\ast$ be a complex-oriented cohomology theory. Then, the $E^\ast(\pt)$-algebra $E^\ast_\C(\var{X}):=i^\ast (E^\ast(\Rvar{X}(\C)))\subset E^\ast(\var{X})$ depends only on $\var{X}$ and not on $\Rvar{X}$. Moreover, if $f:\var{X}\rightarrow \var{Y}$ is a continuous map between two nonsingular complete real affine varieties which is homotopic to a regular one, then $f^\ast(E^\ast_\C(\var{Y}))\subset E^\ast_\C(\var{X})$.

Similarly, let $F^\star$ be a Real-oriented equivariant cohomology theory. Then, the $F^\star(\pt)$-algebra $F^\star_\C(\var{X}):=i^\ast (F^\star(\Rvar{X}))\subset F^\star(\var{X})$ depends only on $\var{X}$ and not on $\Rvar{X}$. Moreover, if $f:\var{X}\rightarrow \var{Y}$ is a continuous map between two nonsingular complete real affine varieties which is homotopic to a regular one, then $f^\ast(F^\star_\C(\var{Y}))\subset F^\star_\C(\var{X})$.
\begin{proof}
Again, we only prove the second part. Let $\Rvar{X}'$ be another nonsingular projective complexification of $\var{X}$. Applying Theorem \ref{thm:morphismExists} to the identity $\var{X}\rightarrow \var{X}$ we find a homomorphism $\psi:F^\star(\Rvar{X}')\rightarrow F^\star(\Rvar{X})$ making the diagram
\begin{center}
\begin{tikzcd}
F^\star(\Rvar{X}) \arrow[rd, "i^\ast"] & F^\star(\Rvar{X}') \arrow[d, "i'^\ast"] \arrow[l, "\psi"'] \\
                                & F^\star(\var{X})                                      
\end{tikzcd}
\end{center}
commutative. It follows that $i'^\ast(F^\star(\Rvar{X}'))\subset i^\ast(F^\star(\Rvar{X}))$. By symmetry we have an equality, so $F^\star_\C(\var{X})$ is well defined.

Let now $f:\var{X}\rightarrow \var{Y}$ be a continuous map, which is homotopic to a regular one. Due to homotopy invariance of generalised cohomology, without loss of generality we may assume that $f$ itself is regular. Fix any nonsingular projective complexifications $\Rvar{X}$ and $\Rvar{Y}$ of $\var{X}$ and $\var{Y}$ respectively. Again applying Theorem \ref{thm:morphismExists} we find a homomorphism $\xi:F^\star(\Rvar{Y})\rightarrow F^\star(\Rvar{X})$ making the diagram
\begin{center}
\begin{tikzcd}
F^\star(\Rvar{X}) \arrow[d, "i^\ast"] & F^\star(\Rvar{Y}) \arrow[d, "j^\ast"] \arrow[l, "\xi"'] \\
F^\star(\var{X})                  & F^\star(\var{Y}) \arrow[l, "f^\ast"']               
\end{tikzcd}
\end{center}
commutative. It follows that $f^\ast\circ j^\ast (F^\star(\Rvar{Y}))\subset i^\ast(F^\star(\Rvar{X}))$.
\end{proof}
\end{cor}
\begin{rem}
In the case when the cohomology theory $E^\ast$ is singular cohomology with coefficients in a commutative ring with unity $R$ and $\var{X}$ is $R$-orientable, the $R$-algebras $H^\ast_\C(\var{X};R)$ were defined in \cite{ozanHomologyRealAlgebraic2001} by Ozan. The current definition is a significant generalisation of that. 
\end{rem}

\begin{rem}
The existence of the homomorphism $\xi$ making the diagram at the end of the proof of Corollary \ref{cor:inclusion} commutative is a stronger property than the inclusion $f^\ast(F^\star_\C(\var{Y}))\subset F^\star_\C(\var{X})$. Thus, even if the inclusion holds true, one can define a secondary obstruction in an appropriate $\Ext$ group. We will not pursue this line of thought in this paper though.
\end{rem}

For convenience, we extend the definition given in Corollary \ref{cor:inclusion} for pairs:
\begin{defi}
Let $\var{X}$ be a complete nonsingular real affine variety, and let $C\subset \var{X}$ be any subset. Let $E^\ast$ be a complex-oriented cohomology theory. Define 
\begin{equation*}
    E^\ast_\C(\var{X},C):=(\iota^\ast)^{-1}(E^\ast_\C(\var{X})),
\end{equation*}
where $\iota$ is the inclusion of pairs $(\var{X},\emptyset)\hookrightarrow (\var{X},C)$ and $\iota^\ast:E^\ast(\var{X},C)\rightarrow E^\ast(\var{X})$ is the induced morphism.
Similarly, if $F^\star$ is a Real-oriented equivariant cohomology theory, then define
\begin{equation*}
    F^\star_\C(\var{X},C):=(\iota^\ast)^{-1}(F^\star_\C(\var{X})).
\end{equation*}
\end{defi}
Straight from the definition we get the following
\begin{obs}\label{obs:inclusion}
Let $\var{X}$ and $\var{Y}$ be two complete nonsingular real affine varieties and let $C\subset \var{X}, D\subset \var{Y}$ be their subsets. Let $f:(\var{X},C)\rightarrow (\var{Y},D)$ be a continuous map of pairs. Assume that $f$, treated as a map between $\var{X}$ and $\var{Y}$ is homotopic to a regular one. Let $F^\star$ be a Real-oriented equivariant cohomology theory. Then,
\begin{equation*}
    f^\ast(F^\star_\C(\var{Y},D))\subset F_\C^\star(\var{X},C).
\end{equation*}

A similar inclusion holds true for a complex-oriented cohomology theory $E^\ast$.
\begin{proof}
It suffices to apply Corollary \ref{cor:inclusion} and consider the diagram
\begin{center}
\begin{tikzcd}
{F^\star(\var{Y},D)} \arrow[d, "f^\ast"] \arrow[r] & F^\star(\var{Y}) \arrow[d, "f^\ast"] \\
{F^\star(\var{X},C)} \arrow[r]                      & F^\star(\var{X})                    
\end{tikzcd}
\end{center}
\end{proof}
\end{obs}

\begin{rem}\label{rem:defOfPair}
Choosing a nonsingular projective complexification of $\var{X}$ and analysing the diagram
\begin{center}
\begin{tikzcd}
F^{\star-1}(C) \arrow[d, "\cong"] \arrow[r] & {F^\star(\Rvar{X},C)} \arrow[d] \arrow[r] & F^\star(\Rvar{X}) \arrow[d] \arrow[r] & F^{\star}(C) \arrow[d, "\cong"] \\
F^{\star-1}(C) \arrow[r]                    & {F^\star(\var{X},C)} \arrow[r]        & F^\star(\var{X}) \arrow[r]        & F^{\star}(C)                   
\end{tikzcd}
\end{center}
with exact rows one can verify that $F^\star_\C(\var{X},C)$ may alternatively be defined as $i^\ast(F^\star(\Rvar{X},C))$, where $i:(\var{X},C)\hookrightarrow (\Rvar{X},C)$ is the inclusion of pairs. A similar statement is true in the complex-oriented case.
\end{rem}

If $\var{X}$ is a complete nonsingular real affine variety with a chosen base point $x_0\in \var{X}$, then we define the reduced version of $\widetilde F_\C^\star(\var{X})$ as $\widetilde F^\star_\C(\var{X}):=F^\star_\C(\var{X},x_0)$ and similarly in the complex-oriented case. We have
\begin{obs}\label{obs:unreducedVSreduced}
Let $\var{X}$ be a complete nonsingular real affine variety with a chosen base point $x_0\in \var{X}$ and let $F^\star$ be a Real-oriented equivariant cohomology theory. Then, there is an isomorphism
\begin{equation*}
    F^\star_\C(\var{X})\cong \widetilde F_\C^\star(\var{X})\oplus F^\star (\pt).
\end{equation*}
A similar statement is true in the complex-oriented case.
\begin{proof}
Let $\Rvar{X}$ be a nonsingular projective complexification of $\var{X}$. It suffices to recall the definitions and consider the following diagram whose rows are split-exact in a compatible way:
\begin{center}
\begin{tikzcd}
0 \arrow[r] & {F^\star(\Rvar{X},x_0)} \arrow[r] \arrow[d] & F^\star(\Rvar{X}) \arrow[r] \arrow[d] & F^\star(x_0) \arrow[r] \arrow[d] & 0 \\
0 \arrow[r] & {F^\star(\var{X},x_0)} \arrow[r]        & F^\star(\var{X}) \arrow[r]        & F^\star(x_0) \arrow[r]           & 0
\end{tikzcd}
\end{center}
\end{proof}
\end{obs}

Recall from Section \ref{sec:preliminaries:coefficients}, that in the particular case of K-theory, there is the realification transformation $r:KU^\ast\rightarrow K\R^\ast$. As a consequence of its functoriality, we get
\begin{obs}\label{obs:realificationKC}
Let $\var{X}$ be a complete nonsingular real affine variety and let $C\subset \var{X}$ be its subset. Then,
\begin{equation*}
    r(KU^p_\C(\var{X},C))\subset K\R^p_\C(\var{X},C)
\end{equation*}
for all $p\in \Z$.
\begin{proof}
Fix a nonsingular projective complexification $\Rvar{X}$ of $\var{X}$. It suffices to consider the commutative diagram
\begin{center}
\begin{tikzcd}
{KU^\ast(\Rvar{X}(\C),C)} \arrow[d] \arrow[r, "r"] & {K\R^\ast(\Rvar{X}(\C),C)} \arrow[d] \\
{KU^\ast(\var{X},C)} \arrow[r, "r"]            & {K\R^\ast(\var{X},C)}           
\end{tikzcd}
\end{center}
\end{proof}
\end{obs}

\subsection{A comparison map from algebraic K-theory}
Let $\var{X}$ be a complete nonsingular real affine variety. Recall from the introduction that there is the monomorphism
\begin{equation*}
    \iota :K_0(\RR(\var{X}))\rightarrow K_0(\mathcal C(\var{X}))\cong KO^0(\var{X}),
\end{equation*}
where the last isomorphism is due to topological Serre-Swan theorem. We denote its image by
\begin{equation*}
    KO^0_{\C-\alg}(\var{X}):=\iota(K_0(\RR(\var{X}))).
\end{equation*}

There are also two analogous monomorphisms
\begin{align*}
    \iota' &:K_0(\RR(\var{X},\C))\rightarrow K_0(\mathcal C(\var{X},\C))\cong KU^0(\var{X}), \\
    \iota''&:K_0(\RR(\var{X},\Hq))\rightarrow K_0(\mathcal C(\var{X},\Hq))\cong KSp^0(\var{X}),
\end{align*}
where $\RR(\var{X},\C)$ is understood as the ring whose elements are regular maps from $\var{X}$ to $\R^2$, with pointwise addition and multiplication induced by the field structure of $\C$. The ring $\RR(\var{X},\Hq)$ is defined in a similar way. The homomorphisms $\iota'$ and $\iota''$ are injective as well. We introduce similar notation
\begin{equation*}
    KU^0_{\C-\alg}(\var{X}):=\iota'(K_0(\RR(\var{X},\C))),\quad KSp^0_{\C-\alg}(\var{X}):=\iota''(K_0(\RR(\var{X},\Hq))).
\end{equation*}
These groups have been extensively used as obstructions to the existence of regular maps, due to the following:
\begin{obs}
Let $\var{X}$ and $\var{Y}$ be two complete nonsingular real affine varieties. Let $f:\var{X}\rightarrow \var{Y}$ be a continuous map. Assume that it is homotopic to a regular one. Then, the induced morphism
\begin{equation*}
    f^\ast:KO^0(\var{Y})\rightarrow KO^0(\var{X})
\end{equation*}
satisfies 
\begin{equation*}
    f^\ast(KO^0_{\C-\alg}(\var{Y}))\subset KO_{\C-\alg}^0(\var{X}).
\end{equation*}
A similar statement is true in the complex and quaternionic case.
\begin{proof}
Without loss of generality we can assume that $f$ is regular itself. The conclusion then follows from commutativity of the following diagram:
\begin{center}
\begin{tikzcd}
K_0(\RR(\var{Y})) \arrow[r, "f^\ast"] \arrow[d]          & K_0(\RR(\var{X})) \arrow[d]          \\
K_0(\mathcal C(\var{Y})) \arrow[r, "f^\ast"] \arrow[d, "\cong"] & K_0(\mathcal C(\var{X})) \arrow[d, "\cong"] \\
KO^0(\var{Y}) \arrow[r, "f^\ast"]                        & KO^0(\var{X})                       
\end{tikzcd}
\end{center}
\end{proof}
\end{obs}

These algebraic K-groups are fundamentally connected to the invariants considered in the previous subsection due to the following consequence of the exact sequence of a localisation in algebraic K-theory:
\begin{prop}\label{prop:KCalgInKC}
Let $\var{X}$ be a complete nonsingular affine variety. Then
\begin{equation*}
    KO^0_{\C-\alg}(\var{X})\subset K\R^0_\C(\var{X})
\end{equation*}
and 
\begin{equation*}
    KU^0_{\C-\alg}(\var{X})\subset KU^0_\C(\var{X}),
\end{equation*}
where $KO^0(\var{X})$ is identified with $K\R^0(\var{X})$ through the natural isomorphism.
\begin{proof}
Let us begin with the first part. We will make use of the natural transformation between the functors
\begin{equation*}
    \Rvar{X}\mapsto K_0(\Rvar{X}),\quad \Rvar{X}\mapsto K\R^0(\Rvar{X})
\end{equation*}
defined on the category of nonsingular quasi-projective $\R$-varieties, sending an algebraic vector bundle to the underlying topological one with complex conjugation (cf. \cite{karoubiAlgebraicRealKtheory2003}). We denote this transformation by $\alpha$.

According to \cite[Lemma 12.6.3]{bochnakRealAlgebraicGeometry1998}, $\var{X}$ can be embedded as a Zariski closed subset of $\R^n$ for some $n$ in such a way that the Zariski closure $\Rvar{A}$ of $\var{X}$ in $\C^n$ is a nonsingular affine $\R$-variety. By definition, the ring $\RR(\var{X})$ of globally defined regular functions on $\var{X}=\Rvar{A}(\R)$ is the localisation of the ring $\mathcal O_{\Rvar{A}}(\Rvar{A})$ of globally defined regular functions on $\Rvar{A}$ in the multiplicative monoid of functions different from zero at all points of $\Rvar{A}(\R)$. 

Now, using Hironaka's theorem on resolution of singularities, we can find a nonsingular projective $\R$-variety $\Rvar{X}$, such that $\Rvar{A}$ embeds as a Zariski open and dense subset of $\Rvar{X}$. It follows that $\var{X}$ is Zariski dense in $\Rvar{X}(\C)$. It must then be true that $\Rvar{X}(\R)=\Rvar{A}(\R)=\var{X}$, for otherwise $\Rvar{X}(\R)$ would contain noncentral points \cite[Corollary 2.3.3]{akbulutTopologyRealAlgebraic1992} and hence be singular.

Consider now the following diagram, which is commutative due to the definitions of $\alpha$ and $\iota$:
\begin{center}
\begin{tikzcd}
K_0(\Rvar{X}) \arrow[d] \arrow[r, "\alpha"]          & K\R^0(\Rvar{X}) \arrow[d]  \\
K_0(\Rvar{A}) \arrow[d, "\cong"] \arrow[r, "\alpha"] & K\R^0(\Rvar{A}) \arrow[dd] \\
K_0(\mathcal O_{\Rvar{X}}(\Rvar{A})) \arrow[d]                &                     \\
K_0(\RR(\var{X})) \arrow[r, "\iota"]             & K\R^0(\var{X})        
\end{tikzcd}
\end{center}
Here, the isomorphism on the left is due to algebraic Serre-Swan theorem, and the bottom left arrow is induced by the inclusion of $\mathcal O_{\Rvar{X}}(\Rvar{A})$ in its localisation $\RR(\var{X})$. 

To finish the proof, we assert that the composition of the vertical arrows in the left column is surjective. Indeed, the bottom left vertical arrow is surjective due to \cite[Theorem 6.5, p. 499]{bassAlgebraicKtheory1968}, and the upper left vertical arrow is surjective due to \cite[Application II.6.4.2 and Theorem II.8.2]{weibelKbookIntroductionAlgebraic2013}. It follows that 
\begin{equation*}
    \iota(K_0(\RR(\var{X})))\subset K\R^0_\C(\var{X}).    
\end{equation*}

The second part follows by a similar argument, with the difference that instead of the transformation $\alpha$ we need to consider the natural transformation between the functors
\begin{equation*}
    \Rvar{X}(\C)\mapsto K_0(\Rvar{X}(\C)),\quad \Rvar{X}(\C)\mapsto KU^0(\Rvar{X}(\C))
\end{equation*}
defined on the category of nonsingular quasi-projective complex varieties \cite{baumRiemannRochTopologicalKtheorySingular1979}.
\end{proof}
\end{prop}
\begin{rem}
The notation $KO_{\C-\alg}^0$ and $KU_{\C-\alg}^0$ is motivated by the classical functor $H^{2\ast}_{\C-\alg}(-;\Z)$ defined in terms of the Chow ring (see \cite{buchnerAlgebraicVectorBundles1987}), which bears a similar relation to $H_\C^\ast(-;\Z)$ as algebraic K-theory does to $K\R^0_\C$ and $KU_\C^0$. It is possible that similar \say{$\C$-algebraic} invariants can be defined in connection to other complex- and Real-oriented cohomology theories as well. Perhaps the notion of an oriented cohomology theory on the category of smooth varieties over $\C$ or $\R$, and that of algebraic cobordism of Morel and Levine \cite{AlgebraicCobordism2007} could be incorporated in our current approach. 
\end{rem}
\begin{defi}\label{defi:KOCfromKRC}
For convenience, we introduce the notation $KO^p_{\C}(\var{X})\subset KO^p(\var{X})$ for the subgroup of $KO^p(\var{X})$ corresponding to the group $K\R_\C^{p+0\tau}(\var{X})$ under the isomorphism
\begin{equation*}
    K\R^{p+0\tau}(\var{X})\cong KO^p(\var{X}).
\end{equation*}
\end{defi}
The first part of the conclusion of Proposition \ref{prop:KCalgInKC} can then be restated as
\begin{equation*}
    KO^0_{\C-\alg}(\var{X})\subset KO^0_{\C}(\var{X}).
\end{equation*}

In the case of quaternionic K-theory, a similar result holds true:
\begin{prop}\label{prop:quaternionCycleClass}
Let $\var{X}$ be a complete nonsingular real affine variety. Then, the image of $KSp^0_{\C-\alg}(\var{X})$ through the isomorphism described in Proposition \ref{prop:KSptoKO}
\begin{equation*}
    KSp^0(\var{X})\xrightarrow{\cong} KO^{-4}(\var{X})
\end{equation*}
is contained in $KO^{-4}_\C(\var{X})$.
\end{prop}
The proof of Proposition \ref{prop:quaternionCycleClass} will be provided in the next subsection.

Since the tensor product of two projective modules is projective, it is clear that $KO^0_{\C-\alg}(\var{X})$ and $KU^0_{\C-\alg}(\var{X})$ comprise subrings of $KO^0(\var{X})$ and $KU^0(\var{X})$ respectively. As expected, the other products described in Section \ref{sec:preliminaries:coefficients} respect the groups $KO_{\C-\alg}^0,KU_{\C-\alg}^0$ and $KSp_{\C-\alg}^0$ as well:
\begin{obs}\label{obs:KCalgCup}
Let $\var{X}$ be a complete nonsingular real affine variety. The products described in Section \ref{sec:preliminaries:coefficients} restrict to the following products:
\begin{align*}
    KO_{\C-\alg}^0(\var{X})&\times KSp_{\C-\alg}^0(\var{X})\rightarrow KSp_{\C-\alg}^0(\var{X}), \\
    KSp_{\C-\alg}^0(\var{X})&\times KSp_{\C-\alg}^0(\var{X})\rightarrow KO_{\C-\alg}^0(\var{X}).
\end{align*}
\begin{proof}
Notice that if $N$ is a free $\RR(\var{X})$-module, and $M$ is a free $\RR(\var{X},\Hq)$-module, then the product $N\otimes_{\RR(\var{X})} M$ is a free $\RR(\var{X},\Hq)$-module. From this, one deduces that if $N$ is a projective $\RR(\var{X})$-module, and $M$ is a projective $\RR(\var{X},\Hq)$-module, then the product $N\otimes_{\RR(\var{X})} M$ is a projective $\RR(\var{X},\Hq)$-module. This proves the first part.

For the second part, recall from Section \ref{sec:preliminaries:coefficients} that on the level of vector bundles, to two (left) $\Hq$-vector bundles $E$ and $F$ over $\var{X}$ the product
\begin{equation*}
    KSp^0(\var{X})\times KSp^0(\var{X})\rightarrow KO^0(\var{X})
\end{equation*}
associates the $\R$-bundle 
\begin{equation*}
    E^{\mathrm{op}}\otimes_\Hq F,
\end{equation*}
where $E^{\mathrm{op}}$ is the right $\Hq$-vector bundle induced by the isomorphism $\Hq\cong \Hq^{\mathrm{op}}$. The same construction can be carried over in the category of modules over $\RR(\var{X},\Hq)=\RR(\var{X})\otimes_\R \Hq$; to modules $M,N$ one associates the $\RR(\var{X})$-module $L:=M^{\mathrm{op}}\otimes_{\RR(\var{X},\Hq)} N$. In this case it is again easy to verify that if the modules $M$ and $N$ are free, then $L$ is free, and hence if $M$ and $N$ are projective then $L$ is projective.
\end{proof}
\end{obs}

Note also, that for similar reasons the realification homomorphism $r:KU^0(\var{X}) \rightarrow KO^0(\var{X})$, and the quaternionification homomorphism $h:KU^0(\var{X})\rightarrow KSp^0(\var{X})$ carry $KU_{\C-\alg}^0(\var{X})$ into $KO_{\C-\alg}^0(\var{X})$ and $KSp_{\C-\alg}^0(\var{X})$ respectively.

Similarly as it was done in the topological case, for our convenience, we pose a version of the definition for pairs:
\begin{defi}
Let $\var{X}$ be a complete nonsingular real affine variety, and let $C\subset \var{X}$ be any subset. Define 
\begin{align*}
    KO^0_{\C-\alg}(\var{X},C)&:=(\iota^\ast)^{-1}(KO^0_{\C-\alg}(\var{X})),
\end{align*}
where $\iota$ is the inclusion of pairs $(\var{X},\emptyset)\hookrightarrow (\var{X},C)$ and $\iota^\ast:KO^0(\var{X},C)\rightarrow KO^0(\var{X})$ is the induced morphism. 

If $\var{X}$ has a chosen base point $x_0\in \var{X}$, define also 
\begin{equation*}
    \widetilde{KO}^0_{\C-\alg}(\var{X}):=KO^0_{\C-\alg}(\var{X},x_0).
\end{equation*}

Define $KU^0_{\C-\alg}(\var{X},C),KSp^0_{\C-\alg}(\var{X},C)$ and $\widetilde{KU}^0_{\C-\alg}(\var{X}),\widetilde{KSp}^0_{\C-\alg}(\var{X})$ similarly.
\end{defi}

\begin{obs}\label{obs:reducedKR}
Let $\var{X}$ be a complete nonsingular real affine variety with a chosen base point $x_0\in \var{X}$. Then, there is an isomorphism
\begin{equation*}
    KO^0_{\C-\alg}(\var{X})\cong \widetilde{KO}^0_{\C-\alg}(\var{X})\oplus KO^0(x_0)\cong \widetilde{KO}^0_{\C-\alg}(\var{X})\oplus \Z.
\end{equation*}
There is a similar isomorphism in the case of $KU^0_{\C-\alg}(\var{X})$ and $KSp^0_{\C-\alg}(\var{X})$.
\begin{proof}
We prove only the first part. The inclusion of $x_0$ in $\var{X}$ is a regular map, so it induces a morphism
\begin{equation}\label{reducedKReq1}
    KO^0_{\C-\alg}(\var{X})\rightarrow KO^0_{\C-\alg}(x_0),
\end{equation}
where we treat $x_0$ as a real affine subvariety of $\var{X}$. The constant map from $\var{X}$ to $x_0$ is also regular, and it induces a section of the morphism \eqref{reducedKReq1}. From this we obtain the split short exact sequence
\begin{equation*}
    0\rightarrow KO^0_{\C-\alg}(\var{X},x_0)\rightarrow KO^0_{\C-\alg}(\var{X})\rightarrow KO^0_{\C-\alg}(x_0)\rightarrow 0.
\end{equation*}
Moreover, it is clear that $KO^0_{\C-\alg}(x_0)=KO^0(x_0)$, from which the conclusion follows.
\end{proof}
\end{obs}

The map
\begin{equation*}
    \iota:K_0(\RR(\var{X}))\rightarrow KO^0(\var{X}) 
\end{equation*}
admits a reduced version
\begin{equation*}
    \tilde \iota:\widetilde K_0(\RR(\var{X}))\rightarrow \widetilde{KO}^0(\var{X}).
\end{equation*}
Having obtained the splitting in Observation \ref{obs:reducedKR} it is easy to see that the image of $\tilde \iota$ is equal to $\widetilde{KO}^0_{\C-\alg}(\var{X})$. Similar statements are true in the unitary and quaternionic case.
\subsection{General properties of the invariants on product varieties}\label{sec:invariantsProduct}
Let $\var{X}$ and $\var{Y}$ be two complete nonsingular real affine varieties, with chosen base points $x_0\in \var{X}$ and $y_0\in \var{Y}$. In this section we make some general remarks on how the invariants defined in the previous two sections behave on the variety $\var{X}\times \var{Y}$.

Let $F^\star$ be an equivariant cohomology theory. Consider the long exact sequence of the triple $(\var{X}\times \var{Y}, \var{X}\vee \var{Y},(x_0,y_0))$:
\begin{equation*}
	\dots \rightarrow  F^\star(\var{X}\times \var{Y}, \var{X}\vee \var{Y})\rightarrow \widetilde F^\star(\var{X}\times \var{Y})\rightarrow \widetilde F^\star (\var{X}\vee \var{Y})\rightarrow \dots
\end{equation*}
We claim that the map $\widetilde F^\star(\var{X}\times \var{Y})\rightarrow \widetilde F^\star (\var{X}\vee \var{Y})$ has a section. To construct it, first apply the wedge axiom to get an isomorphism
\begin{equation*}
	\widetilde F^\star (\var{X}\vee \var{Y})\cong \widetilde F^\star (\var{X})\oplus \widetilde F^\star(\var{Y}).
\end{equation*}
Now the section can be defined on each of the summands separately and then extended linearly, by
\begin{equation*}
p_1^\ast:\widetilde F^\star (\var{X}) \rightarrow\widetilde F^\star (\var{X}\times \var{Y}), \quad p_2^\ast:\widetilde F^\star (\var{Y}) \rightarrow\widetilde F^\star (\var{X}\times \var{Y}),
\end{equation*}
where $p_1:\var{X}\times \var{Y}\rightarrow \var{X}$ and $p_2:\var{X}\times \var{Y}\rightarrow \var{Y}$ are the projections. This gives rise to the naturally split exact sequence
\begin{equation}\label{eq:productIsoShortEx}
	0 \rightarrow  F^\star(\var{X}\times \var{Y}, \var{X}\vee \var{Y})\rightarrow  \widetilde F^\star(\var{X}\times \var{Y})\rightarrow \widetilde F^\star (\var{X})\oplus \widetilde F^\star(\var{Y})\rightarrow 0,
\end{equation}
and hence to the natural isomorphism
\begin{equation}\label{eq:productIsoTop}
	\widetilde F^\star(\var{X}\times \var{Y})\cong F^\star(\var{X}\times \var{Y}, \var{X}\vee \var{Y})\oplus  \widetilde F^\star (\var{X})\oplus \widetilde F^\star(\var{Y}).
\end{equation}
By the same reasoning, the same conclusion applies if we had started with a (nonequivariant) generalised cohomology theory instead.
\begin{prop}\label{prop:splittings}
Let $F^\star$ be a Real-oriented equivariant cohomology theory. Then, the isomorphism \eqref{eq:productIsoTop} restricts to an isomorphism
\begin{equation*}
	\widetilde F^\star_\C(\var{X}\times \var{Y})\cong F^\star_\C(\var{X}\times \var{Y}, \var{X}\vee \var{Y})\oplus  \widetilde F^\star_\C (\var{X})\oplus \widetilde F^\star_\C(\var{Y}).
\end{equation*}
There is a similar isomorphism
\begin{equation*}
    \widetilde E_\C^\ast(\var{X}\times \var{Y})\cong E_\C^\ast(\var{X}\times \var{Y}, \var{X}\vee \var{Y})\oplus  \widetilde E_\C^\ast (\var{X})\oplus \widetilde E_\C^\ast(\var{Y}).
\end{equation*}
where $E^\ast$ is any complex-oriented cohomology theory. There are also similar isomorphisms with the functors $\widetilde{KO}^0_{\C-\alg},\widetilde{KU}^0_{\C-\alg}$ and $\widetilde{KSp}^0_{\C-\alg}$ in place of $\widetilde F^\star_\C$.
\begin{proof}
We treat only the case of $F^\star_\C$ as all the statements have completely analogous proofs. By the definition of $F^\star_\C(\var{X}\times \var{Y}, \var{X}\vee \var{Y})$, the sequence \eqref{eq:productIsoShortEx} restricts to an exact sequence
\begin{equation*}
	0 \rightarrow  F_\C^\star(\var{X}\times \var{Y}, \var{X}\vee \var{Y})\rightarrow  \widetilde F_\C^\star(\var{X}\times \var{Y})\rightarrow \widetilde F^\star (\var{X})\oplus \widetilde F^\star(\var{Y}).
\end{equation*}
The two components of the last arrow are restrictions of the morphisms
\begin{equation*}
	i_1^\ast:\widetilde F^\star(\var{X}\times \var{Y}) \rightarrow \widetilde F^\star (\var{X}),\quad i_2^\ast:\widetilde F^\star(\var{X}\times \var{Y}) \rightarrow \widetilde F^\star (\var{Y}),
\end{equation*}
where $i_1:\var{X}\rightarrow \var{X}\times \var{Y}$ and $i_2:\var{Y}\rightarrow \var{X}\times \var{Y}$ are the inclusions. Since the inclusions are regular maps, it follows that the image of $\widetilde F_\C^\star(\var{X}\times \var{Y})$ is contained in $\widetilde F^\star_\C (\var{X})\oplus \widetilde F^\star_\C(\var{Y})$. On the other hand, the constructed section of the last arrow in \eqref{eq:productIsoShortEx} is on each summand induced by the respective projection, which is a regular map. Thus, the section maps $\widetilde F^\star_\C (\var{X})\oplus \widetilde F^\star_\C(\var{Y})$ into $\widetilde F_\C^\star(\var{X}\times \var{Y})$. It follows that the sequence
\begin{equation*}
	0 \rightarrow  F_\C^\star(\var{X}\times \var{Y}, \var{X}\vee \var{Y})\rightarrow  \widetilde F_\C^\star(\var{X}\times \var{Y})\rightarrow \widetilde F_\C^\star (\var{X})\oplus \widetilde F_\C^\star(\var{Y})\rightarrow 0
\end{equation*}
is split exact, which gives the desired isomorphism.
\end{proof}
\end{prop}

Now, suppose that $\var{Y}$ is homeomorphic to the $n$-dimensional sphere $\sph^n$. Let $F^\star$ be a Real-oriented equivariant cohomology theory. The inclusion $\var{X}\vee \var{Y}\hookrightarrow \var{X}\times \var{Y}$ is a nonequivariant cofibration. As the spaces are endowed with the trivial involution, it is also a cofibration equivariantly. Hence, we have the chain of isomorphisms
\begin{equation*}
	 F^\star(\var{X}\times \var{Y},\var{X}\vee \var{Y})\xrightarrow{\cong} \widetilde F^\star(\var{X}\wedge \sph^n)\xrightarrow{\cong} \widetilde F^{\star-n}(\var{X}),
\end{equation*}
where the second arrow is the suspension isomorphism. In this setting the following proposition relates $F_\C^\star(\var{X}\times \var{Y})$ with $F^\star_\C(\var{X})$:
\begin{prop}\label{prop:suspensionFC}
The composition of the two isomorphisms restricts to a morphism 
\begin{equation*}
	 F_\C^\star(\var{X}\times \var{Y},\var{X}\vee \var{Y})\rightarrow \widetilde F_\C^{\star-n}(\var{X}).
\end{equation*}
A similar statement is true in the case of a complex-oriented cohomology theory.
\begin{proof}
We treat only the case of a Real-oriented equivariant cohomology theory. Let $\Rvar{X}$ be any nonsingular projective complexification of $\var{X}$ and let $\Rvar{Y}$ be any nonsingular projective complexification of $\var{Y}$. By the same argument as before we can consider the following diagram, whose rows are split exact in a compatible way:
\begin{center}
\begin{tikzcd}[column sep=small]
0 \arrow[r] & {F^\star(\Rvar{X}\times \Rvar{Y},\Rvar{X}\vee \Rvar{Y})} \arrow[d] \arrow[r]   & \widetilde F^\star(\Rvar{X}\times \Rvar{Y}) \arrow[d] \arrow[r] & \widetilde F^\star(\Rvar{X}\vee \Rvar{Y}) \arrow[d] \arrow[r] & 0 \\
0 \arrow[r] & {F^\star(\var{X}\times \var{Y},\var{X}\vee \var{Y})} \arrow[r] & \widetilde F^\star(\var{X}\times \var{Y}) \arrow[r]     & \widetilde F^\star(\var{X}\vee \var{Y}) \arrow[r]     & 0
\end{tikzcd}
\end{center}
It now follows from definition that $F_\C^\star(\var{X}\times \var{Y}, \var{X}\vee \var{Y})$ is equal to the image of the arrow 
\begin{equation*}
    F^\star(\Rvar{X}\times \Rvar{Y},\Rvar{X}\vee \Rvar{Y})\rightarrow F^\star(\var{X}\times \var{Y},\var{X}\vee \var{Y}).
\end{equation*}
The inclusion of $\Rvar{X}\vee \var{Y}$ into $\Rvar{X}\times \var{Y}$ is also an equivariant cofibration, as follows from \cite[Lemma 1.1]{arakiOrientationsTcohomologyTheories1979}. The conclusion follows after considering the following commutative diagram:
\begin{center}
\begin{tikzcd}
{ F^\star(\Rvar{X}\times \Rvar{Y},\Rvar{X}\vee \Rvar{Y})} \arrow[r] \arrow[d]              & { F^\star(\var{X}\times \var{Y},\var{X}\vee \var{Y})} \arrow[dd, "\cong"] \\
{ F^\star(\Rvar{X}\times \var{Y},\Rvar{X}\vee \var{Y})} \arrow[d, "\cong"] \arrow[ru] &                                                               \\
\widetilde F^\star(\Rvar{X}\wedge \sph^n) \arrow[d, "\cong"] \arrow[r]               & \widetilde F^\star(\var{X}\wedge \sph^n) \arrow[d, "\cong"]      \\
\widetilde F^{\star-n}(\Rvar{X}) \arrow[r]                                  & \widetilde F^{\star-n}(\var{X})                                 
\end{tikzcd}
\end{center}
\end{proof}
\end{prop}

To a multiplicative equivariant cohomology theory there is associated the (reduced) cross product 
\begin{equation*}
	\widetilde F^{p+q\tau}(X)\times \widetilde F^{r+s\tau}(Y)\rightarrow \widetilde F^{t+u\tau}(X\wedge Y),
\end{equation*}
where $t:=p+r,\;u:=q+s$. We have the following:
\begin{obs}\label{obs:productsExist}
Let $\var{X}$ and $\var{Y}$ be two complete nonsingular real affine varieties with chosen base points $x_0\in \var{X},y_0\in \var{Y}$. The product
\begin{equation*}
	\widetilde F^{p+q\tau}(\var{X})\times \widetilde F^{r+s\tau}(\var{Y})\rightarrow \widetilde F^{t+u\tau}(\var{X}\wedge \var{Y})\cong F^{t+u\tau}(\var{X}\times \var{Y},\var{X}\vee \var{Y})
\end{equation*}
restricts to a product
\begin{equation*}
	\widetilde F_\C^{p+q\tau}(\var{X})\times \widetilde F_\C^{r+s\tau}(\var{Y})\rightarrow F_\C^{t+u\tau}(\var{X}\times \var{Y},\var{X}\vee \var{Y}).
\end{equation*}
A similar statement is true in the complex-oriented case. By the same construction there arise also similar products
\begin{align*}
	\widetilde {KO}_{\C-\alg}^0(\var{X})\times \widetilde {KO}_{\C-\alg}^0(\var{Y})&\rightarrow {KO}_{\C-\alg}^{0}(\var{X}\times \var{Y},\var{X}\vee \var{Y}),\\
    \widetilde {KU}_{\C-\alg}^0(\var{X})\times \widetilde {KU}_{\C-\alg}^0(\var{Y})&\rightarrow {KU}_{\C-\alg}^{0}(\var{X}\times \var{Y},\var{X}\vee \var{Y}),\\
    \widetilde {KO}_{\C-\alg}^0(\var{X})\times \widetilde {KSp}_{\C-\alg}^0(\var{Y})&\rightarrow {KSp}_{\C-\alg}^{0}(\var{X}\times \var{Y},\var{X}\vee \var{Y}),\\
    \widetilde {KSp}_{\C-\alg}^0(\var{X})\times \widetilde {KSp}_{\C-\alg}^0(\var{Y})&\rightarrow {KO}_{\C-\alg}^{0}(\var{X}\times \var{Y},\var{X}\vee \var{Y}).
\end{align*}
\begin{proof}
Let us start with the case of $\widetilde F_\C^\star$. Due to commutativity of the diagram
\begin{center}
\begin{tikzcd}
\widetilde F^{p+q\tau}(\var{X})\times \widetilde F^{r+s\tau}(\var{Y}) \arrow[d] \arrow[r] & {F^{t+u\tau}(\var{X}\times \var{Y},\var{X}\vee \var{Y})} \arrow[d] \\
F^{p+q\tau}(\var{X}) \times F^{r+s\tau}(\var{Y}) \arrow[r]                                & F^{t+u\tau}(\var{X}\times \var{Y})                          
\end{tikzcd}
\end{center}
it suffices to consider the unreduced cross product instead of the reduced one. Now, the cross product can be presented in terms of the cup product as in the diagram below:
\begin{center}
\begin{tikzcd}
F^{p+q\tau}(\var{X}) \times F^{r+s\tau}(\var{Y}) \arrow[r,"\times"] \arrow[d,"p_1^\ast\times p_2^\ast"]   & F^{t+u\tau}(\var{X}\times \var{Y}) \\
F^{p+q\tau}(\var{X}\times \var{Y}) \times F^{r+s\tau}(\var{X}\times \var{Y}) \arrow[ru,"\smile"] &                             
\end{tikzcd}
\end{center}
where the vertical arrow is induced by the projections
\begin{equation*}
    p_1:\var{X}\times \var{Y}\rightarrow \var{X},\quad p_2:\var{X}\times \var{Y}\rightarrow \var{Y}.
\end{equation*}
The conclusion follows, as the projections are regular and $F^\star_\C(\var{X}\times \var{Y})$ forms a ring under the cup product.

The same argument provides a conclusion in the case of a complex-oriented cohomology theory. The last part also follows from the same argument, keeping in mind Observation \ref{obs:KCalgCup}.
\end{proof}
\end{obs}

We can now provide the proof of Proposition \ref{prop:quaternionCycleClass}:
\begin{proof}[Proof of Proposition \ref{prop:quaternionCycleClass}]
Recall that by $\sph^n$ we denote the unit sphere in $\R^{n+1}$, considered as an affine variety. It is well-known that 
\begin{equation*}
    KSp^0_{\C-\alg}(\sph^4)=KSp^0(\sph^4)    
\end{equation*}
(the generator of $\widetilde{KSp}^0_{\C-\alg}(\sph^4)$ corresponding to $\theta\in KSp^{-4}(\pt)$ can be constructed by endowing the quaternionic projective line $\Hq P^1$ with a structure of a real affine variety and checking that it is isomorphic to $\sph^4$ \cite{bochnakAlgebraicApproximationMappings1987}).

Let $\var{X}$ be a complete nonsingular affine variety with base point $x_0\in \var{X}$. We have the commutative diagram
\begin{center}
\begin{tikzcd}
\widetilde{KSp}_{\C-\alg}^0(\var{X})\times \widetilde{KSp}_{\C-\alg}^0(\sph^4) \arrow[r] \arrow[d] & \widetilde{KSp}^0(\var{X})\times \widetilde{KSp}^0(\sph^4) \arrow[d, "\cong"] \\
{KO_{\C-\alg}^0(\var{X}\times \sph^4,\var{X}\vee \sph^4)} \arrow[d]                                & \widetilde{KSp}^0(\var{X})\times KSp^{-4}(\pt) \arrow[d]                     \\
\widetilde{KO}^{-4}_\C(\var{X}) \arrow[r]                                                             & \widetilde{KO}^{-4}(\var{X})                                                 
\end{tikzcd}
\end{center}
Here, the arrows in the left column are due to Observation \ref{obs:KCalgCup} and Proposition \ref{prop:suspensionFC}. Fixing the element of $\widetilde{KSp}_{\C-\alg}^0(\sph^4)$ corresponding to $\theta \in KSp^{-4}(\pt)$, we see that the image of $\widetilde{KSp}_{\C-\alg}^0(\var{X})$ under the isomorphism
\begin{equation*}
    \widetilde{KSp}^0(\var{X})\xrightarrow{\cong} \widetilde{KO}^{-4}(\var{X})
\end{equation*}
induced by multiplication by $\theta$ lies in $\widetilde{KO}^{-4}_\C(\var{X})$. This proves that a reduced version of Proposition \ref{prop:quaternionCycleClass} is true. The unreduced version then follows from Observation \ref{obs:unreducedVSreduced}.
\end{proof}
\section{Computation of \texorpdfstring{$KU$- and $K\R$-}{KU- and KR-}based invariants for spheres}
We can now start explicit computations of some of the newly defined invariants for spheres. In the upcoming Subsection \ref{sec:computation:spheres} we completely describe the groups $KU_{\C}^p(\sph^n)$ and $KO_{\C}^p(\sph^n)$ for all $n$ and $p$ (recall that in Definition \ref{defi:KOCfromKRC} we introduced the notation $KO_{\C}^p(X):=K\R_{\C}^p(X)$). Then, in Subsection \ref{sec:computation:algebraicKtheory} we compute algebraic K-groups of products of spheres. We remark here that the latter is not needed for applications in the problem of realisation of homotopy classes from products of spheres into spheres by regular maps. Therefore, a reader interested only in these applications might safely skip the rather computational Subsection \ref{sec:computation:algebraicKtheory}.

\subsection{The algebras $KO_{\C}^\ast(\sph^n)$ and $KU_{\C}^\ast(\sph^n)$}\label{sec:computation:spheres}
Before we begin the computation let us give a brief summary of our approach. We start by noticing that the generators of the algebras $KU_{\C}^\ast(\R P^n),KO_{\C}^\ast(\R P^n)$ can be easily given explicitly, using the fact that the real projective space $\R P^n$ admits the well-behaved complexification $\C P^n$. Then, using the fact that the map contracting $\R P^{n-1}$ to a point $\psi_n:\R P^n \rightarrow \sph^n$ is regular we deduce that 
\begin{align*}
    \psi_n^\ast( KO_{\C}^\ast(\sph^n))&\subset KO_{\C}^\ast(\R P^n),\\
    \psi_n^\ast( KU_{\C}^\ast(\sph^n))&\subset KU_{\C}^\ast(\R P^n).
\end{align*}
This gives an upper bound for how large the subalgebras $KO_{\C}^\ast(\sph^n)$ and $KU_{\C}^\ast(\sph^n)$ of $KO^\ast(\sph^n)$ and $KU^\ast(\sph^n)$ respectively can be. To obtain a lower bound we make use of the inclusions from Proposition \ref{prop:KCalgInKC} and the description of the algebraic K-groups of rings of regular functions of spheres, which follows from the work of Swan \cite{swanTopologicalExamplesProjective1977}. 

Let us then begin by considering real projective spaces. As noted in the previous paragraph, the real affine variety $\R P^n$ admits the rather simple nonsingular projective complexification $\C P^n$. For a Real-oriented equivariant cohomology theory $F^\star$, and for a complex-oriented cohomology theory $E^\ast$, the rings $F^\star(\C P^n)$ and $E^\ast(\C P^n)$ are well understood:
\begin{thm}[{\cite{adamsStableHomotopyGeneralised1974,arakiOrientationsTcohomologyTheories1979}}]
The $F^\star(\pt)$-algebra $F^\star(\C P^n)$ is isomorphic to
\begin{equation*}
    F^\star(\pt)[u]/(u^{n+1}),
\end{equation*}
where $u$ is the pullback of the Real orientation $c_1^F\in \widetilde F^{1+\tau}(\C P^{\infty})\subset F^{1+\tau}(\C P^{\infty})$ through the inclusion $\C P^n\hookrightarrow \C P^{\infty}$.

Similarly, the $E^\ast(\pt)$-algebra $E^\ast(\C P^n)$ is isomorphic to
\begin{equation*}
    E^\ast(\pt)[v]/(v^{n+1}),
\end{equation*}
where $v$ is the pullback of the complex orientation $c_1^E\in \widetilde E^2(\C P^{\infty})\subset E^2(\C P^{\infty}) $ through the inclusion $\C P^n\hookrightarrow \C P^{\infty}$.
\end{thm}

From the definition, we get
\begin{cor}
As an $F^\star(\pt)$-algebra with unity, $F_\C^\star(\R P^n)$ is generated by the pullback $u':=i^\ast u\in F^{1+\tau}(\R P^n)$ through the inclusion $\R P^n\hookrightarrow \C P^{\infty}$.

As an $E^\ast(\pt)$-algebra with unity, $E_\C^\ast(\R P^n)$ is generated by the pullback $v':=i^\ast v\in E^{2}(\R P^{n})$.
\end{cor}
This does not give us a full description of $F^\star_\C(\R P^n)$ and $E^\ast_\C(\R P^n)$, since we do not know what relations are the elements $u'$ and $v'$ subject to. 

In the case of K-theory we deduce the following:
\begin{obs}\label{obs:KCVanishesRP}
$KU^{p}_\C(\R P^n)=0$ if $p$ is odd. $K\R^{p+q\tau}_\C(\R P^n)=0$ if $p-q$ is congruent to $1,2,3$ or $5$ modulo $8$.
\begin{proof}
The first part follows from the fact that $KU^\ast_\C(\R P^n)$ is generated over $\Z$ by the elements $\beta,\beta^{-1},v'$ of degrees $-2,2,2$ respectively.

The second part follows since $K\R^\star_\C(\R P^n)$ is generated over $KO^\ast(\pt)$ by the elements $\sigma,\sigma^{-1},u'$ of degrees $-1-\tau,1+\tau,1+\tau$ respectively, and $KO^{p}(\pt)$ is zero if $p$ is congruent to $1,2,3$ or $5$ modulo $8$.
\end{proof}
\end{obs}
This is useful for us due to the following corollary:
\begin{cor}\label{cor:vanishingKCifInj}
Fix $n\geq 0$ and $p\in \Z$, such that $p$ is congruent to $1,2,3$ or $5$ modulo $8$. Consider the morphism
\begin{equation*}
    \psi_n^\ast:\widetilde{KO}^p(\sph^n)\rightarrow \widetilde{KO}^p(\R P^n)
\end{equation*}
induced by the real blowup $\psi_n:\R P^n\rightarrow \sph^n$ collapsing $\R P^{n-1}$ to a point. Assume that it is injective. Then, $\widetilde{KO}_\C^p(\sph^n)=0$.

Similarly, let $q\in \Z$ be an odd integer. Assume that the induced morphism
\begin{equation*}
    \psi_n^\ast:\widetilde{KU}^q(\sph^n)\rightarrow \widetilde{KU}^q(\R P^n)
\end{equation*}
is injective. Then, $\widetilde{KU}_\C^q(\sph^n)=0$.
\begin{proof}
The map $\psi_n$ is regular. It follows from Corollary \ref{cor:inclusion} that 
\begin{equation*}
    \psi_n^\ast(\widetilde{KO}_\C^p(\sph^n))\subset \widetilde{KO}_\C^p(\R P^n).
\end{equation*}
Since $\psi_n^\ast$ is injective by assumption and $\widetilde{KO}_\C^p(\R P^n)=0$ thanks to Observation \ref{obs:KCVanishesRP}, we get $\widetilde{KO}_\C^p(\sph^n)=0$. The case of unitary K-theory is analogous.
\end{proof}
\end{cor}
Injectivity of $\psi_n^\ast$ can be verified for particular values of $n$ and $p$ of interest, leading to the following:
\begin{prop}\label{prop:KOCKUCvanishing}
The group $\widetilde{KU}_\C^p(\sph^n)$ is zero if both $n$ and $p$ are odd.

The group $\widetilde{KO}_\C^p(\sph^n)$ is zero in each of the following cases:
\begin{enumerate}
    \item $n$ congruent to $1$ modulo $4$ and $p$ congruent to $n$ modulo $8$,  \label{vanishingp1}
    \item $n$ congruent to $2$ modulo $4$ and $p$ congruent to $n-1$ modulo $8$,    \label{vanishingp2}
    \item $n$ congruent to $3$ modulo $4$ and $p$ congruent to $3$ modulo $8$,\label{vanishingp3}
    \item $n$ congruent to $3$ modulo $4$ and $p$ congruent to $n-2$ modulo $8$. \label{vanishingp4}
\end{enumerate}
\begin{proof}
Let us consider the first part. Consider the long exact sequence of the pair $(\R P^n,\R P^{n-1})$:
\begin{equation*}
    \dots\rightarrow \widetilde{KU}^{p-1}(\R P^{n-1})\rightarrow \widetilde{KU}^{p}(\sph^n)\xrightarrow{\psi_n^\ast} \widetilde{KU}^{p}(\R P^n)\rightarrow \dots
\end{equation*}
We have $\widetilde{KU}^{p}(\sph^n)\cong \Z$, while $\widetilde{KU}^{p-1}(\R P^{n-1})$ is torsion according to \cite[Proposition II.2.7.7]{atiyahKtheory1967}. This shows that $\psi_n^\ast$ is injective, so the conclusion follows from Corollary \ref{cor:vanishingKCifInj}.

The second part follows similarly by considering the long exact sequence 
\begin{equation}\label{exactseq1}
    \dots\rightarrow \widetilde{KO}^{p-1}(\R P^{n-1})\rightarrow \widetilde{KO}^{p}(\sph^n)\xrightarrow{\psi_n^\ast} \widetilde{KO}^{p}(\R P^n)\rightarrow \dots
\end{equation}
After plugging in the groups $\widetilde{KO}^{\ast}(\R P^n),\widetilde{KO}^{\ast}(\R P^{n-1})$, which were computed in \cite{fujiiK0groupsProjectiveSpaces1967}, one concludes that $\psi_n^\ast$ is injective in these cases. 

More explicitly, the cases \eqref{vanishingp1} and \eqref{vanishingp3} follow as $\widetilde{KO}^{p}(\sph^n)\cong \Z$ while $\widetilde{KO}^{p-1}(\R P^{n-1})$ is torsion. The case \eqref{vanishingp2} follows as the three terms of the exact sequence \eqref{exactseq1} starting from $\widetilde{KO}^{p}(\sph^n)$ look like
\begin{center}
\begin{tikzcd}
\widetilde{KO}^{p}(\sph^n) \arrow[r] \arrow[d, "\cong"] &  \widetilde{KO}^{p}(\R P^n) \arrow[r] \arrow[d, "\cong"] &  \widetilde{KO}^{p}(\R P^{n-1}) \arrow[d, "\cong"] \\
\Z/2 \arrow[r]                                          & \Z/2 \arrow[r]                                           & \Z                                                
\end{tikzcd}
\end{center}
The remaining case \eqref{vanishingp4} follows as these three terms look like
\begin{center}
\begin{tikzcd}
\widetilde{KO}^{p}(\sph^n) \arrow[r] \arrow[d, "\cong"] &  \widetilde{KO}^{p}(\R P^n) \arrow[r] \arrow[d, "\cong"] &  \widetilde{KO}^{p}(\R P^{n-1}) \arrow[d, "\cong"] \\
\Z/2 \arrow[r]                                          & \Z/2\oplus \Z/2 \arrow[r]                                & \Z/2                                              
\end{tikzcd}
\end{center}
\end{proof}
\end{prop}

This gives an upper bound for how large the groups $\widetilde{KO}_\C^\ast(\sph^n)$ and $\widetilde{KU}_\C^\ast(\sph^n)$ can be. A lower bound is a consequence of the following classical result:
\begin{thm}[{\cite[Theorem 11.1]{swanTopologicalExamplesProjective1977}}]\label{thm:KCalgSphere}
We have 
\begin{align*}
    \widetilde{KU}^0_{\C-\alg}(\sph^n)&=\widetilde{KU}^0(\sph^n),\\
    \widetilde{KO}^0_{\C-\alg}(\sph^n)&=\widetilde{KO}^0(\sph^n),\\
    \widetilde{KSp}^0_{\C-\alg}(\sph^n)&=\widetilde{KSp}^0(\sph^n)
\end{align*}
for all $n\geq 0$.
\end{thm}
From Propositions \ref{prop:KCalgInKC} and \ref{prop:quaternionCycleClass} we obtain the following:
\begin{cor}\label{cor:KCsphereFromKCalg}
We have 
\begin{align*}
    \widetilde{KU}^0_{\C}(\sph^n)&=\widetilde{KU}^0(\sph^n),\\
    \widetilde{KO}^0_{\C}(\sph^n)&=\widetilde{KO}^0(\sph^n),\\
    \widetilde{KO}^{-4}_{\C}(\sph^n)&=\widetilde{KO}^{-4}(\sph^n)
\end{align*}
for all $n\geq 0$.
\end{cor}

We can now begin computation, starting from unitary K-theory:
\begin{thm}\label{thm:KUCSphere}
For $n\geq 0$, we have
\begin{equation*}
    \widetilde{KU}^\ast_\C(\sph^n)=\begin{cases}
        \widetilde{KU}^\ast(\sph^n) & \text{if $n$ is even,}\\
        0 & \text{if $n$ is odd.}
    \end{cases}
\end{equation*}
\begin{proof}
Suppose that $n$ is odd. From Proposition \ref{prop:KOCKUCvanishing} we get $\widetilde{KU}_\C^p(\sph^n)=0$ for $p$ odd. For $p$ even, one has 
\begin{equation*}
    \widetilde{KU}_\C^p(\sph^n)\subset\widetilde{KU}^p(\sph^n)\cong KU^{p-n}(\pt)\cong 0,
\end{equation*}
so $\widetilde{KU}_\C^p(\sph^n)=0$ for all $p$ as claimed.

Suppose that $n$ is even. From Corollary \ref{cor:KCsphereFromKCalg} we get $\widetilde{KU}^0_\C(\sph^n)=\widetilde{KU}^0(\sph^n)$. Since $KU^0_\C(\sph^n)$ is a $KU^\ast(\pt)$-algebra and $KU^\ast(\pt)$ contains the invertible element $\beta$ of degree $-2$, it follows that
\begin{equation}\label{KUCSphereCompEq1}
    \widetilde{KU}^p_\C(\sph^n)=\widetilde{KU}^p(\sph^n)
\end{equation}
for all $p$ even. For $p$ odd we have $\widetilde{KU}^p(\sph^n)\cong KU^{p-n}(\pt)=0$, so \eqref{KUCSphereCompEq1} trivially holds true in this case as well.
\end{proof}
\end{thm}

To describe our results in the case of Real K-theory, we introduce the following notation:
\begin{defi}\label{defi:In}
For $p,n\in \Z,n\geq 0$ denote by $I^n_{(p)}\subset KO^p(\pt)$ the image of $\widetilde{KO}_\C^{n+p}(\sph^n)$ under the suspension isomorphism
\begin{equation*}
    \widetilde{KO}^{n+p}(\sph^n)\xrightarrow{\cong}\widetilde{KO}^p(\sph^0)\xrightarrow{\cong} KO^p(\pt).
\end{equation*}
Since $\widetilde{KO}_\C^{\ast}(\sph^n)$ is a graded $KO^\ast(\pt)$-module, the sum
\begin{equation*}
    \bigoplus_{p=-\infty}^\infty I^n_{(p)}\subset KO^\ast(\pt)
\end{equation*}
is a homogeneous ideal in the ring $KO^\ast(\pt)$, which we denote by $I^n$. 
\end{defi}
The following result describes the $KO^\ast(\pt)$-algebra $KO^\ast_\C(\sph^n)$ completely:
\begin{thm}\label{thm:KOCSphere}
The ideal $I^n$ is described as follows:
\begin{enumerate}
    \item $I^n=(1)$ if $n$ is congruent to $0$ modulo $4$,
    \item $I^n=(\eta)$ if $n$ is congruent to $1$ modulo $4$,
    \item $I^n=(2,\eta^2,\alpha)$ if $n$ is congruent to $2$ modulo $4$,
    \item $I^n=(0)$ if $n$ is congruent to $3$ modulo $4$.
\end{enumerate}
\end{thm}
In the proof of Theorem \ref{thm:KOCSphere} we will make use of the following lemma:
\begin{lem}\label{lem:IvanishesImpliesIContained}
Let $J\subset KO^\ast(\pt)$ be a homogeneous ideal in the graded ring $KO^\ast(\pt)$. For $k\in \Z$ denote by $J_{(k)}$ its $k$-th homogeneous part. Then, the following statements are true
\begin{enumerate}
    \item if $J_{(0)}=0$, then $J\subset (\eta)$,
    \item if $J_{(-1)}=0$ then $J\subset (2,\eta^2,\alpha)$,
    \item if $J_{(0)}=0$ and $J_{(-2)}=0$ then $J=(0)$,
    \item if $J_{(-4)}=(0)$ and $J_{(-2)}=0$ then $J=(0)$.
\end{enumerate}
\begin{proof}
All the four statements are a simple consequence of the structure of the coefficient ring $KO^\ast(\pt)$. 

Indeed, if $J_{(0)}=0$ then also $J_{(4)}=0$ since multiplication by $\alpha$ is a monomorphism from $J_{(4)}$ to $J_{(0)}$. Since the coefficient ring $KO^\ast(\pt)$ contains the invertible element $\lambda$ of degree $-8$, it follows that every nonzero homogeneous element of $J$ is of degree congruent to $-1$ or $-2$ modulo $8$. Every such element is divisible by $\eta$, so $J\subset (\eta)$.

Now assume that $J_{(-1)}=0$. The group $J_{(0)}$ is then contained in the kernel of the homomorphism $KO^0(\pt)\rightarrow KO^{-1}(\pt)$ induced by multiplication by $\eta$. It follows that every element of $J_{(0)}$ is divisible by two, and more generally every element of $J$ of degree divisible by $8$ is divisible by $2$. Every other nonzero homogeneous element of $J$ is of degree congruent to either $-2$ or $-4$ modulo $8$, so it is divisible either by $\eta^2$ or $\alpha$. Hence $J\subset (2,\eta^2,\alpha)$.

Assume that $J_{(0)}=0$ and $J_{(-2)}=0$. By the already covered case $J\subset (\eta)$; in particular $J_{(-4)}=0$. Moreover, multiplication by $\eta$ induces an injection from $J_{(-1)}$ to $J_{(-2)}$, so $J_{(-1)}=0$ as well. Hence, $J_{(k)}=0$ for every $0\geq k \geq -7$, from which it follows that $J=(0)$.

Finally assume that $J_{(-4)}=0$ and $J_{(-2)}=0$. Multiplication by $\alpha$ induces an injection from $J_{(0)}$ to $J_{(-4)}$, so $J_{(0)}=0$. It now follows that $J=(0)$ from the previously covered case.
\end{proof}
\end{lem}
\begin{proof}[Proof of Theorem \ref{thm:KOCSphere}]
We split the proof into $8$ cases, depending on the congruence of $n$ modulo $8$.
\begin{case}
$n$ congruent to $0$ modulo $8$.
\end{case}
According to Corollary \ref{cor:KCsphereFromKCalg}, we have $\widetilde{KO}^0_\C(\sph^n)=\widetilde{KO}^0(\sph^n)$. Hence $I^n_{(-n)}=KO^{-n}(\pt)$. It follows that $I^n$ contains the invertible element $\lambda^{n/8}$, so $I^n=(1)$.
\begin{case}
$n$ congruent to $4$ modulo $8$
\end{case}
This is similar to the previous case. According to Corollary \ref{cor:KCsphereFromKCalg}, we have $\widetilde{KO}^{-4}_\C(\sph^n)=\widetilde{KO}^{-4}(\sph^n)$. Hence $I^n_{(-n-4)}=KO^{-n-4}(\pt)$. It follows that $I^n$ contains the invertible element $\lambda^{(n+4)/8}$, so $I^n=(1)$.
\begin{case}
$n$ congruent to $1$ modulo $8$
\end{case}
According to Corollary \ref{cor:KCsphereFromKCalg}, we have $\widetilde{KO}^0_\C(\sph^n)=\widetilde{KO}^0(\sph^n)$. Hence $I^n_{(-n)}=KO^{-n}(\pt)$. We deduce that $\lambda^{(n-1)/8} \eta \in I^n$, so $\eta \in I^n$.

On the other hand, according to point \eqref{vanishingp1} of Proposition \ref{prop:KOCKUCvanishing}, we have $\widetilde{KO}^{n}_\C(\sph^n)=0$, so $I^n_{(0)}=0$. From Lemma \ref{lem:IvanishesImpliesIContained} we deduce that $I^n\subset (\eta)$. Hence, $I^n=(\eta)$.
\begin{case}
$n$ congruent to $5$ modulo $8$
\end{case}
This is similar to the previous case. According to Corollary \ref{cor:KCsphereFromKCalg}, we have $\widetilde{KO}^{-4}_\C(\sph^n)=\widetilde{KO}^{-4}(\sph^n)$. Hence $I^n_{(-n-4)}=KO^{-n-4}(\pt)$. We deduce that $\lambda^{(n+3)/8} \eta \in I^{n}$, so $\eta \in I^n$.

On the other hand, according to point \eqref{vanishingp1} of Proposition \ref{prop:KOCKUCvanishing}, we have $\widetilde{KO}_\C^{n}(\sph^n)=0$, so $I^n_{(0)}=0$. From Lemma \ref{lem:IvanishesImpliesIContained} we deduce that $I^n\subset (\eta)$. Hence, $I^n=(\eta)$.
\begin{case}
$n$ congruent to $2$ modulo $8$.
\end{case}

Since $n$ is even, according to Theorem \ref{thm:KUCSphere} we have $\widetilde{KU}^\ast_\C(\sph^n)=\widetilde{KU}^\ast(\sph^n)$. From the commutative diagram
\begin{center}
\begin{tikzcd}
\widetilde{KU}^p(\sph^n) \arrow[r, "r"] \arrow[d, "\cong"] & \widetilde{KO}^p(\sph^n) \arrow[d, "\cong"] \\
KU^{p-n}(\pt) \arrow[r, "r"]                                 & KO^{p-n}(\pt)                                
\end{tikzcd}
\end{center}
and Observation \ref{obs:realificationKC}, we deduce that the image of the realification homomorphism $r:KU^\ast(\pt)\rightarrow KO^\ast(\pt)$ is contained in $I^n$. As described in Proposition \ref{prop:realificationImage}, this image is the ideal generated by $2,\eta^2,\alpha$. Hence, $2,\eta^2,\alpha\in I^n$.

On the other hand, according to Proposition \ref{prop:KOCKUCvanishing} we have $\widetilde{KO}^{n-1}_\C(\sph^n)=0$. Hence $I^n_{(-1)}=0$. According to Lemma \ref{lem:IvanishesImpliesIContained} this forces $I^n\subset (2,\eta^2,\alpha)$. Hence, $I^n=(2,\eta^2,\alpha)$.
\begin{case}
$n$ congruent to $6$ modulo $8$.
\end{case}
This is similar to the previous case. Since $n$ is even, according to Theorem \ref{thm:KUCSphere} we have $\widetilde{KU}^\ast_\C(\sph^n)=\widetilde{KU}^\ast(\sph^n)$. As before, this implies that $2,\eta^2,\alpha \in I^n$.

On the other hand, according to Proposition \ref{prop:KOCKUCvanishing} we have $I^n_{(-1)}=0$. Again, this forces $I^n\subset (2,\eta^2,\alpha)$. Hence, $I^n=(2,\eta^2,\alpha)$.
\begin{case}
$n$ congruent to $3$ modulo $8$.
\end{case}
According to Proposition \ref{prop:KOCKUCvanishing}, we have $\widetilde{KO}^n_\C(\sph^n)=0$ and $\widetilde{KO}^{n-2}_\C(\sph^n)=0$. This means that $I^n_{(0)}=0$ and $I^n_{(-2)}=0$. Lemma \ref{lem:IvanishesImpliesIContained} implies that $I^n=(0)$.
\begin{case}
$n$ congruent to $7$ modulo $8$.
\end{case}
According to Proposition \ref{prop:KOCKUCvanishing}, we have $\widetilde{KO}^{n-4}_\C(\sph^n)=0$ and $\widetilde{KO}^{n-2}_\C(\sph^n)=0$. This means that $I^n_{(-4)}=0$ and $I^n_{(-2)}=0$. Lemma \ref{lem:IvanishesImpliesIContained} implies that $I^n=(0)$.
\end{proof}
\setcounter{case}{0}
\begin{cor}\label{cor:KOCSphereIndex}
The index
\begin{equation*}
    [K\R^{p+q\tau}(\sph^n):K\R^{p+q\tau}_\C(\sph^n)]
\end{equation*}
is equal to $\varphi(n,q-p)$, where $\varphi:\Z\times \Z\rightarrow\{1,2,\infty\}$ is the function described in Section \ref{sec:introduction:products}. 

Note that this describes the subgroup $K\R^{p+q\tau}_\C(\sph^n)$ of $K\R^{p+q\tau}(\sph^n)$ completely, as the group
\begin{equation*}
    K\R^{p+q\tau}(\sph^n)\cong KO^{p-q}(\sph^n)\cong KO^{p-q-n}(\pt)
\end{equation*}
is cyclic.
\begin{proof}
Since the coefficient ring of $K\R^\star$ contains the invertible element $\sigma$ of degree $-1-\tau$, we have
\begin{multline*}
    [K\R^{p+q\tau}(\sph^n):K\R^{p+q\tau}_\C(\sph^n)]=[KO^{p-q}(\sph^n):KO^{p-q}_\C(\sph^n)]= \\
    =[KO^{p-q-n}(\pt):I^n_{(p-q-n)}].
\end{multline*}
Knowing the structure of the coefficient ring $KO^\ast(\pt)$, we deduce the conclusion from Theorem \ref{thm:KOCSphere}.
\end{proof}
\end{cor}

\subsection{Algebraic K-groups of products of spheres}\label{sec:computation:algebraicKtheory}
We can now begin the computation of algebraic K-theory of products of spheres.
\begin{thm}\label{thm:KUCalgSpheres}
Let $n,m\geq 0$ be nonnegative integers. If at least one of the numbers $n$ and $m$ is even, then,
\begin{multline*}
    KU_{\C-\alg}^0(\sph^n\times \sph^m,\sph^n\vee \sph^m)=KU_\C^0(\sph^n\times \sph^m,\sph^n\vee \sph^m)=\\
    =KU^0(\sph^n\times \sph^m,\sph^n\vee \sph^m).
\end{multline*}
On the contrary, if both of them are odd, then
\begin{equation*}
    KU_{\C-\alg}^0(\sph^n\times \sph^m,\sph^n\vee \sph^m)=KU_\C^0(\sph^n\times \sph^m,\sph^n\vee \sph^m)=0.
\end{equation*}
\begin{proof}
If one of the numbers $n,m$ is odd and one is even, then
\begin{equation*}
    KU^0(\sph^n\times \sph^m,\sph^n\vee \sph^m)\cong KU^{-n-m}(\pt)=0,
\end{equation*}
so the conclusion holds trivially.

Suppose that both $n$ and $m$ are odd. From Proposition \ref{prop:suspensionFC} we get that the image of the group $KU^0_\C(\sph^n\times \sph^m,\sph^n\vee \sph^m)$ under the isomorphism
\begin{equation*}
    KU^0(\sph^n\times \sph^m,\sph^n\vee \sph^m)\xrightarrow{\cong}\widetilde{KU}^{-m}(\sph^n)
\end{equation*}
is contained in $\widetilde{KU}_\C^{-m}(\sph^n)$. However, $\widetilde{KU}_\C^{-m}(\sph^n)=0$ due to Theorem \ref{thm:KUCSphere}, so $KU^0_\C(\sph^n\times \sph^m,\sph^n\vee \sph^m)=0$. Hence $KU^0_{\C-\alg}(\sph^n\times \sph^m,\sph^n\vee \sph^m)=0$ as well since it is contained in the former group.

Finally, suppose that both $n$ and $m$ are even. Using Observation \ref{obs:productsExist} and compatibility of the cross product with suspension, we find the following commutative diagram:
\begin{center}
\begin{tikzcd}
\widetilde{KU}_{\C-\alg}^0(\sph^n)\times \widetilde{KU}_{\C-\alg}^0(\sph^m) \arrow[r] \arrow[d, "\cong"] & {KU^0_{\C-\alg}(\sph^n\times \sph^m,\sph^n\vee \sph^m)} \arrow[d, hook] \\
\widetilde{KU}^0(\sph^n)\times \widetilde{KU}^0(\sph^m) \arrow[d, "\cong"] \arrow[r]                     & {KU^0(\sph^n\times \sph^m,\sph^n\vee \sph^m)} \arrow[d, "\cong"]        \\
KU^{-n}(\pt)\times KU^{-m}(\pt) \arrow[r]                                                                    & KU^{-n-m}(\pt)                                                                    
\end{tikzcd}
\end{center}
Here the upper left vertical arrow is a bijection due to Theorem \ref{thm:KCalgSphere}. Now, chasing the pair $(\beta^{n/2},\beta^{m/2})\in KU^{-n}(\pt)\times KU^{-m}(\pt)$ around the diagram, we find that the element $\beta^{n/2}\beta^{m/2}=\beta^{(n+m)/2}\in KU^{-n-m}(\pt)$ lies in the image of the composition
\begin{equation*}
    {KU^0_{\C-\alg}(\sph^n\times \sph^m,\sph^n\vee \sph^m)}\hookrightarrow KU^0(\sph^n\times \sph^m,\sph^n\vee \sph^m)\xrightarrow{\cong} KU^{-n-m}(\pt).
\end{equation*}
Since $\beta^{(n+m)/2}$ generates $KU^{-n-m}(\pt)$ as a group, it follows that we have 
\begin{equation*}
    KU^0_{\C-\alg}(\sph^n\times \sph^m,\sph^n\vee \sph^m)=KU^0(\sph^n\times \sph^m,\sph^n\vee \sph^m).
\end{equation*}
Hence $KU^0_{\C}(\sph^n\times \sph^m,\sph^n\vee \sph^m)=KU^0(\sph^n\times \sph^m,\sph^n\vee \sph^m)$
as well.
\end{proof}
\end{thm}

The groups $KO^0_{\C-\alg}(\sph^n\times \sph^m,\sph^n\vee \sph^m)$ and $KO^0_{\C}(\sph^n\times \sph^m,\sph^n\vee \sph^m)$ can now be described in a similar manner:
\begin{thm}\label{thm:KOCSphereKOCalgSpheres}
Let $n$ and $m$ be nonnegative integers. Then, the groups $KO^0_{\C-\alg}(\sph^n\times \sph^m,\sph^n\vee \sph^m)$ and $KO^0_{\C}(\sph^n\times \sph^m,\sph^n\vee \sph^m)$ are equal as subgroups of $KO^0(\sph^n\times \sph^m,\sph^n\vee \sph^m)$. Moreover, the image of $KO^0_{\C}(\sph^n\times \sph^m,\sph^n\vee \sph^m)$ under the isomorphism
\begin{equation*}
        KO^0(\sph^n\times \sph^m,\sph^n\vee \sph^m)\xrightarrow{\cong}\widetilde{KO}^{-m}(\sph^n)
    \end{equation*}
is precisely equal to $\widetilde{KO}^{-m}_\C(\sph^n)$. Hence, in view of Corollary \ref{cor:KOCSphereIndex} the index
\begin{align*}
    &[KO^0(\sph^n\times \sph^m,\sph^n\vee \sph^m):KO^0_{\C-\alg}(\sph^n\times \sph^m,\sph^n\vee \sph^m)]=\\
    =&[KO^0(\sph^n\times \sph^m,\sph^n\vee \sph^m):KO^0_{\C}(\sph^n\times \sph^m,\sph^n\vee \sph^m)]
\end{align*}
is equal to $\varphi(n,m)$.
\begin{proof}
Let us fix $n\geq 0$. We still use the notation $I^n$ for the ideal described in Definition \ref{defi:In}.

For $p\leq -n$, define $A^n_{(p)}\subset KO^p(\pt)$ as the image of the group $KO^0_{\C-\alg}(\sph^n\times \sph^{-p-n},\sph^n\vee \sph^{-p-n})$ through the composition of the isomorphisms
\begin{equation}\label{eq:KOCIso}
    KO^0(\sph^n\times \sph^{-p-n},\sph^n\vee \sph^{-p-n})\xrightarrow{\cong}\widetilde{KO}^{p+n}(\sph^n)\xrightarrow{\cong} KO^{p}(\pt).
\end{equation}
Define $B^n_{(p)}\subset KO^p(\pt)$ similarly as the image of $KO^0_{\C}(\sph^n\times \sph^{-p-n},\sph^n\vee \sph^{-p-n})$ through the isomorphism \eqref{eq:KOCIso}. Extend these definitions to all $p\in \Z$, by defining 
\begin{equation*}
    A^n_{(p)}=B^n_{(p)}:=I^n_{(p)}
\end{equation*}
for $p>-n$.

Finally, define the graded groups
\begin{equation*}
    A^n:=\bigoplus_{p=-\infty}^\infty A^n_{(p)}\subset KO^\ast(\pt),\quad B^n:=\bigoplus_{p=-\infty}^\infty B^n_{(p)}\subset KO^\ast(\pt).
\end{equation*}
It follows from Proposition \ref{prop:KCalgInKC}, that $A^n\subset B^n$. Moreover, it follows from Proposition \ref{prop:suspensionFC}, that $B^n\subset I^n$. 

Now, the desired equality of $KO^0_{\C-\alg}(\sph^n\times \sph^m,\sph^n\vee \sph^m)$ and $KO^0_{\C}(\sph^n\times \sph^m,\sph^n\vee \sph^m)$ is equivalent to $A^n_{(-n-m)}$ being equal to $B^n_{(-n-m)}$. Similarly, the desired equality of the image of $KO^0_{\C}(\sph^n\times \sph^m,\sph^n\vee \sph^m)$ and $\widetilde{KO}^{-m}(\sph^n)$ is equivalent to the equality $B^n_{(-n-m)}=I^n_{(-n-m)}$. Therefore, keeping in mind the chain of inclusions $A^n\subset B^n\subset I^n$, it suffices to show that $I^n\subset A^n$ for all $n$. We consider four cases, depending on the congruence of $n$ modulo $4$.

If $n$ is congruent to $3$ modulo $4$, then as computed in Theorem \ref{thm:KOCSphere}, we have $I^n=(0)$. The containment $I^n\subset A^n$ is then obvious.

If $n$ is congruent to $0$ or $1$ modulo $4$, then it follows from the description of $I^n$ in Theorem \ref{thm:KOCSphere}, that as an ideal it is generated by elements of degree $-n$ and $-n-4$. Hence, in these two cases the conclusion follows from the following lemma:

\begin{lem}\label{lem:proofKO1}
For every $n$ congruent to $0$ or $1$ modulo $4$, the ideal of $KO^\ast(\pt)$ generated by all the elements of degree $-n$ is contained in $A^n$. Similarly, the ideal of $KO^\ast(\pt)$ generated by all the elements of degree $-n-4$ is also contained in $A^n$.
\begin{proof}
Let us begin with the first part, regarding elements of degree $-n$. Let $\gamma\in KO^{-n}(\pt)$ be a generator of $KO^{-n}(\pt)$ (as a group). It suffices to show, that for all $m\in \Z$, and for all elements $\delta \in KO^{-m}(\pt)$ of degree $-m$, the product $\gamma \delta$ belongs to $A^n_{(-n-m)}$. This is clear for $m<0$, as in this case $A^n_{(-n-m)}=I^n_{(-n-m)}$ by definition. Let us hence assume that $m\geq 0$. Due to Observation \ref{obs:productsExist} and the compatibility of the cross product with suspension, we have the commutative diagram:
\begin{center}
\begin{tikzcd}
\widetilde{KO}_{\C-\alg}^0(\sph^n)\times \widetilde{KO}_{\C-\alg}^0(\sph^m) \arrow[d, "\cong"] \arrow[r] & {KO^0_{\C-\alg}(\sph^n\times \sph^m,\sph^n\vee \sph^m)} \arrow[d, hook] \\
\widetilde{KO}^0(\sph^n)\times \widetilde{KO}^0(\sph^m) \arrow[r] \arrow[d, "\cong"]                  & {KO^0(\sph^n\times \sph^m,\sph^n\vee \sph^m)} \arrow[d, "\cong"]        \\
KO^{-n}(\pt)\times KO^{-m}(\pt) \arrow[r]                                                                 & KO^{-n-m}(\pt)                                                                    
\end{tikzcd}
\end{center}
Here, the upper left arrow is a bijection due to Theorem \ref{thm:KCalgSphere}. Chasing the pair $(\gamma,\delta)\in KO^{-n}(\pt)\times KO^{-m}(\pt)$ around the diagram, we verify that the product of its coordinates $\gamma\delta\in KO^{-n-m}(\pt)$ is in the image of the composition on the right 
\begin{equation*}
    KO^0_{\C-\alg}(\sph^n\times \sph^m,\sph^n\vee \sph^m)\rightarrow KO^{-n-m}(\pt).
\end{equation*}
Hence $\gamma\delta\in A^n_{(-n-m)}$ by definition.

Let us now consider the second part, regarding elements of degree $-n-4$. Let $\gamma'\in KO^{-n-4}(\pt)$ be a generator of $KO^{-n-4}(\pt)$. It suffices to show, that for all $m\in \Z$, and for all elements $\delta' \in KO^{-m+4}(\pt)$ of degree $-m+4$, the product $\gamma' \delta'$ belongs to $A^n_{(-n-m)}$. The case $m<0$ is again trivial, so let us assume that $m\geq 0$. Due to Observation \ref{obs:productsExist} and the compatibility of the cross product with suspension, we have the commutative diagram:
\begin{center}
\begin{tikzcd}
\widetilde{KSp}_{\C-\alg}^0(\sph^n)\times \widetilde{KSp}_{\C-\alg}^0(\sph^m) \arrow[d, "\cong"] \arrow[r] & {KO^0_{\C-\alg}(\sph^n\times \sph^m,\sph^n\vee \sph^m)} \arrow[d, hook] \\
\widetilde{KSp}^0(\sph^n)\times \widetilde{KSp}^0(\sph^m) \arrow[r] \arrow[d, "\cong"]                     & {KO^0(\sph^n\times \sph^m,\sph^n\vee \sph^m)} \arrow[d, "\cong"]        \\
KSp^{-n}(\pt)\times KSp^{-m}(\pt) \arrow[r]                                                                    & KO^{-n-m}(\pt)                                                                     
                                 
\end{tikzcd}
\end{center}
Here, the upper left vertical arrow is a bijection due to Theorem \ref{thm:KCalgSphere}.

Chasing the class $(\lambda^{-1}\gamma'\theta,\delta'\theta)\in KSp^{-n}(\pt)\times KSp^{-m}(\pt)$ around the diagram, we verify, that the product of its coordinates $\gamma'\delta'\in KO^{-n-m}(\pt)$ is in the image of the composition on the right 
\begin{equation*}
    KO^0_{\C-\alg}(\sph^n\times \sph^m,\sph^n\vee \sph^m)\rightarrow KO^{-n-m}(\pt).
\end{equation*}
Hence $\gamma'\delta'\in A^n_{(-n-m)}$ by definition.
\end{proof}
\end{lem}

We are now left with the case, when $n$ is congruent to $2$ modulo $4$. According to Theorem \ref{thm:KOCSphere}, in this case we have that $I^n$ is equal to the image of the realification homomorphism 
\begin{equation*}
    r:KU^\ast(\pt)\rightarrow KO^\ast(\pt).
\end{equation*}
To conclude, it hence suffices to prove that this image is contained in $A^n$.

Since $n$ is even, according to Theorem \ref{thm:KUCalgSpheres} we have
\begin{equation*}
    KU^0_{\C-\alg}(\sph^n\times \sph^m,\sph^n\vee \sph^m)=KU^0(\sph^n\times \sph^m,\sph^n\vee \sph^m)
\end{equation*}
for all $m\geq 0$. We have the commutative diagram
\begin{center}
\begin{tikzcd}
{KU^0_{\C-\alg}(\sph^n\times \sph^m,\sph^n\vee \sph^m)} \arrow[r, "r"] \arrow[d, "\cong"] & {KO^0_{\C-\alg}(\sph^n\times \sph^m,\sph^n\vee \sph^m)} \arrow[d] \\
{KU^0(\sph^n\times \sph^m,\sph^n\vee \sph^m)} \arrow[d,"\cong"] \arrow[r, "r"]          & {KO^0(\sph^n\times \sph^m,\sph^n\vee \sph^m)} \arrow[d] \\
KU^{-n-m}(\pt) \arrow[r, "r"]                                                              & KO^{-n-m}(\pt)                                                    
\end{tikzcd}
\end{center}
It follows that
\begin{equation*}
    r(KU^{-n-m}(\pt))\subset A^n_{(-n-m)}
\end{equation*}
for $m\geq 0$. For $m<0$ this is obvious, as $A^n_{(-n-m)}=I^n_{(-n-m)}$.

This way we have considered all the possible cases and the proof has been finished.
\end{proof}
\end{thm}
Using a similar method, we can also describe the groups $KSp^0_{\C-\alg}(\sph^n\times \sph^m,\sph^n\vee\sph^m)$ and $KO^{-4}(\sph^n\times \sph^m,\sph^n\vee\sph^m)$:
\begin{thm}\label{thm:KSpCalgSpheres}
Let $n$ and $m$ be nonnegative integers. Then, the image of the group $KSp^0_{\C-\alg}(\sph^n\times \sph^m,\sph^n\vee \sph^m)$ under the isomorphism
\begin{equation*}
    KSp^0(\sph^n\times \sph^m,\sph^n\vee \sph^m)\xrightarrow{\cong} KO^{-4}(\sph^n\times \sph^m,\sph^n\vee \sph^m)
\end{equation*}
is equal to $KO^{-4}_{\C}(\sph^n\times \sph^m,\sph^n\vee \sph^m)$. Moreover, the image of $KO^{-4}_{\C}(\sph^n\times \sph^m,\sph^n\vee \sph^m)$ under the isomorphism
\begin{equation*}
        KO^{-4}(\sph^n\times \sph^m,\sph^n\vee \sph^m)\xrightarrow{\cong}\widetilde{KO}^{-m-4}(\sph^n)
    \end{equation*}
is equal to $\widetilde{KO}^{-m-4}_\C(\sph^n)$. Hence, in view of Corollary \ref{cor:KOCSphereIndex} the index
\begin{align*}
    &[KSp^0(\sph^n\times \sph^m,\sph^n\vee \sph^m):KSp^0_{\C-\alg}(\sph^n\times \sph^m,\sph^n\vee \sph^m)]=\\
    =&[KO^{-4}(\sph^n\times \sph^m,\sph^n\vee \sph^m):KO^{-4}_{\C}(\sph^n\times \sph^m,\sph^n\vee \sph^m)]
\end{align*}
is equal to $\varphi(n,m+4)$.
\begin{proof}
The proof is analogous to the proof of Theorem \ref{thm:KOCSphereKOCalgSpheres}. For $p\leq -n-4$, define $\widetilde {A}^n_{(p)}\subset KO^p(\pt)$ as the image of the group $KSp^0_{\C-\alg}(\sph^n\times \sph^{-p-n-4},\sph^n\vee \sph^{-p-n-4})$ through the composition of the isomorphisms
\begin{equation*}
    KSp^0(\sph^n\times \sph^{-p-n-4},\sph^n\vee \sph^{-p-n-4})\xrightarrow{\cong}KSp^{p+4}(\pt)\xrightarrow{\cong} KO^{p}(\pt).
\end{equation*}
Define $\widetilde {B}^n_{(p)}\subset KO^p(\pt)$ similarly as the image of $KO^{-4}_{\C}(\sph^n\times \sph^{-p-n-4},\sph^n\vee \sph^{-p-n-4})$ through the isomorphism 
\begin{equation*}
    KO^{-4}(\sph^n\times \sph^{-p-n-4},\sph^n\vee \sph^{-p-n-4})\xrightarrow{\cong} KO^p(\pt).
\end{equation*}
Extend these definitions to all $p\in \Z$, by defining 
\begin{equation*}
    \widetilde A^n_{(p)}=\widetilde B^n_{(p)}:=I^n_{(p)}
\end{equation*}
for $p>-n-4$. Define then the graded groups
\begin{equation*}
    \widetilde {A}^n:=\bigoplus_{p=-\infty}^\infty \widetilde {A}^n_{(p)}\subset KO^\ast(\pt),\quad \widetilde {B}^n:=\bigoplus_{p=-\infty}^\infty \widetilde {B}^n_{(p)}\subset KO^\ast(\pt).
\end{equation*}
It follows from Proposition \ref{prop:quaternionCycleClass}, that $\widetilde {A}^n\subset \widetilde {B}^n$. Moreover, it follows from Proposition \ref{prop:suspensionFC}, that $\widetilde {B}^n\subset I^n$. 

Similarly as before, it hence suffices to show that $I^n\subset \widetilde {A}^n$ for all $n$. The case when $n$ is congruent to $3$ modulo $4$ is again trivial, as $I^n=0$. 

The cases when $n$ is congruent to $0$ or $1$ modulo $4$ follow from the following analogue of Lemma \ref{lem:proofKO1}:
\begin{lem}
For every $n$ congruent to $0$ or $1$ modulo $4$, the ideal of $KO^\ast(\pt)$ generated by all the elements of degree $-n$ is contained in $\widetilde{A}^n$. Similarly, the ideal of $KO^\ast(\pt)$ generated by all the elements of degree $-n-4$ is also contained in $\widetilde{A}^n$.
\begin{proof}
Let us begin with the first part, regarding elements of degree $-n$. Let $\gamma\in KO^{-n}(\pt)$ be a generator of $KO^{-n}(\pt)$ (as a group). It suffices to show, that for all $m\in \Z$, and for all elements $\delta \in KO^{-m-4}(\pt)$ of degree $-m-4$, the product $\gamma \delta$ belongs to $\widetilde {A}^n_{(-n-m-4)}$. This is clear for $m<0$, as in this case $\widetilde {A}^n_{(-n-m-4)}=I^n_{(-n-m-4)}$ by definition. Let us hence assume that $m\geq 0$. Due to Observation \ref{obs:productsExist} and the compatibility of the cross product with suspension, we have the commutative diagram:
\begin{center}
\begin{tikzcd}
\widetilde{KO}_{\C-\alg}^0(\sph^n)\times \widetilde{KSp}_{\C-\alg}^0(\sph^m) \arrow[d, "\cong"] \arrow[r] & {KSp^0_{\C-\alg}(\sph^n\times \sph^m,\sph^n\vee \sph^m)} \arrow[d, hook] \\
\widetilde{KO}^0(\sph^n)\times \widetilde{KSp}^0(\sph^m) \arrow[r] \arrow[d, "\cong"]                     & {KSp^0(\sph^n\times \sph^m,\sph^n\vee \sph^m)} \arrow[d, "\cong"]        \\
KO^{-n}(\pt)\times KSp^{-m}(\pt) \arrow[r] \arrow[d, "\cong"]                                                 & KSp^{-n-m}(\pt) \arrow[d]                                                           \\
KO^{-n}(\pt)\times KO^{-m-4}(\pt) \arrow[r]                                                                   & KO^{-n-m-4}(\pt)                                                                   
\end{tikzcd}
\end{center}
Here, the upper left arrow is a bijection due to Theorem \ref{thm:KCalgSphere}. Chasing the pair $(\gamma,\delta)\in KO^{-n}(\pt)\times KO^{-m-4}(\pt)$ we deduce that the product $\gamma\delta\in KO^{-n-m-4}(\pt)$ is in the image of the homomorphism 
\begin{equation*}
    KSp^0_{\C-\alg}(\sph^n\times \sph^m,\sph^n\vee \sph^m)\rightarrow KO^{-n-m-4}(\pt).
\end{equation*}
Hence $\gamma\delta\in \widetilde {A}^n_{(-n-m-4)}$ by definition.

Let us now consider the second part, regarding elements of degree $-n-4$. Let $\gamma'\in KO^{-n-4}(\pt)$ be a generator of $KO^{-n-4}(\pt)$. It suffices to show, that for all $m\in \Z$, and for all elements $\delta' \in KO^{-m}(\pt)$ of degree $-m$, the product $\gamma' \delta'$ belongs to $\widetilde {A}^n_{(-n-m-4)}$. The case $m<0$ is again trivial, so let us assume that $m\geq 0$. Due to Observation \ref{obs:productsExist} and the compatibility of the cross product with suspension, we have the commutative diagram:
\begin{center}
\begin{tikzcd}
\widetilde{KSp}_{\C-\alg}^0(\sph^n)\times \widetilde{KO}_{\C-\alg}^0(\sph^m) \arrow[d, "\cong"] \arrow[r] & {KSp^0_{\C-\alg}(\sph^n\times \sph^m,\sph^n\vee \sph^m)} \arrow[d, hook] \\
\widetilde{KSp}^0(\sph^n)\times \widetilde{KO}^0(\sph^m) \arrow[r] \arrow[d, "\cong"]                     & {KSp^0(\sph^n\times \sph^m,\sph^n\vee \sph^m)} \arrow[d, "\cong"]        \\
KSp^{-n}(\pt)\times KO^{-m}(\pt) \arrow[r] \arrow[d, "\cong"]                                                 & KSp^{-n-m}(\pt) \arrow[d]                                                           \\
KO^{-n-4}(\pt)\times KO^{-m}(\pt) \arrow[r]                                                                   & KO^{-n-m-4}(\pt)                                                                   
\end{tikzcd}
\end{center}
Here, the upper left vertical arrow is a bijection due to Theorem \ref{thm:KCalgSphere}. Chasing the pair $(\gamma',\delta')\in KO^{-n-4}(\pt)\times KO^{-m}(\pt)$ we deduce that the product $\gamma'\delta'\in KO^{-n-m-4}(\pt)$ is in the image of the homomorphism 
\begin{equation*}
    KSp^0_{\C-\alg}(\sph^n\times \sph^m,\sph^n\vee \sph^m)\rightarrow KO^{-n-m-4}(\pt).
\end{equation*}
Hence $\gamma'\delta'\in \widetilde {A}^n_{(-n-m-4)}$ by definition.
\end{proof}
\end{lem}

In the remaining case of $n$ congruent to $2$ modulo $4$, we need to show that the ideal $(2,\eta^2,\alpha)$ is contained in $\widetilde{A}^n$. As stated in Proposition \ref{prop:quaternionification}, this ideal is equal to the image of the composition
\begin{equation*}
    \theta h:KU^\ast(\pt)\xrightarrow{h} KSp^\ast(\pt)\xrightarrow{\cdot \theta} KO^{\ast-4}(\pt).
\end{equation*}
Hence we need to show that for all $m$, we have $\theta h(KU^{-n-m}(\pt))\subset \widetilde {A}^n_{(-n-m-4)}$. This is clear for $m<0$, as in this case $\widetilde {A}^n_{(-n-m-4)}=I^n_{(-n-m-4)}$. Let us then assume that $m\geq 0$.

Since $n$ is even, then according to Theorem \ref{thm:KUCalgSpheres}, for $m\geq 0$ we have
\begin{equation*}
    KU^0_{\C-\alg}(\sph^n\times \sph^m,\sph^n\vee \sph^m)=KU^0(\sph^n\times \sph^m,\sph^n\vee \sph^m).
\end{equation*}
We also have the commutative diagram
\begin{center}
\begin{tikzcd}
{KU_{\C-\alg}^0(\sph^n\times \sph^m,\sph^n\vee \sph^m)} \arrow[r, "h"] \arrow[d, "\cong"] & {KSp_{\C-\alg}^{0}(\sph^n\times \sph^m,\sph^n\vee \sph^m)} \arrow[d] \\
{KU^0(\sph^n\times \sph^m,\sph^n\vee \sph^m)} \arrow[d, "\cong"] \arrow[r, "h"]           & {KSp^0(\sph^n\times \sph^m,\sph^n\vee \sph^m)} \arrow[d,"\cong"]             \\
KU^{-n-m}(\pt) \arrow[r, "h"] \arrow[rd, "\theta h"]                                                 & KSp^{-n-m}(\pt) \arrow[d, "\cong"]                                              \\
                                                                                                      & KO^{-n-m-4}(\pt)                                                               
\end{tikzcd}
\end{center}
It follows that $\theta h(KU^{-n-m}(\pt))$ is contained in the image of the composition on the right, so $\theta h(KU^{-n-m}(\pt))\subset \widetilde {A}^n_{(-n-m-4)}$ as claimed.
\end{proof}
\end{thm}

From these results we deduce the description of the algebraic K-groups stated in the introduction:
\begin{proof}[Proof of Theorems \ref{thm:introKU} and \ref{thm:introKOKSp}]
Recall from Proposition \ref{prop:splittings} that there are splittings
\begin{align*}
    \widetilde{KO}_{\C-\alg}^0(\sph^n\times \sph^m)&\cong \widetilde{KO}_{\C-\alg}^0(\sph^n\times \sph^m,\sph^n\vee \sph^m)\oplus \widetilde{KO}_{\C-\alg}^0(\sph^n)\oplus \widetilde{KO}_{\C-\alg}^0(\sph^m),\\
    \widetilde{KU}_{\C-\alg}^0(\sph^n\times \sph^m)&\cong \widetilde{KU}_{\C-\alg}^0(\sph^n\times \sph^m,\sph^n\vee \sph^m)\oplus \widetilde{KU}_{\C-\alg}^0(\sph^n)\oplus \widetilde{KU}_{\C-\alg}^0(\sph^m),\\
    \widetilde{KSp}_{\C-\alg}^0(\sph^n\times \sph^m)&\cong \widetilde{KSp}_{\C-\alg}^0(\sph^n\times \sph^m,\sph^n\vee \sph^m)\oplus \widetilde{KSp}_{\C-\alg}^0(\sph^n)\oplus \widetilde{KSp}_{\C-\alg}^0(\sph^m).
\end{align*}
Every summand on the right is described by Theorem \ref{thm:KCalgSphere}, Theorem \ref{thm:KUCalgSpheres}, Theorem \ref{thm:KOCSphereKOCalgSpheres} or Theorem \ref{thm:KSpCalgSpheres}. This yields the conclusion.
\end{proof} 
\section{Maps from products of spheres into spheres}
\subsection{K-theoretic invariants as obstructions}
We can now apply our results to obtain negative results about the existence of regular maps into spheres. First, we formulate a general obstruction to the existence of a regular map from $\sph^n\times \sph^m$ into $\sph^{n+m}$ of topological degree one:
\begin{prop}\label{prop:conditionForNoMapSpheres}
Let $n$ and $m$ be positive integers. Let $F^\star$ be a Real-oriented equivariant cohomology theory. Assume that there exists a regular map from $\sph^n\times \sph^m$ into $\sph^{n+m}$ of topological degree one. Then, the suspension isomorphism
\begin{equation*}
    \widetilde{F}^\star(\sph^{n+m})\xrightarrow{ \cong}\widetilde{F}^{\star-m}(\sph^n)
\end{equation*}
carries $\widetilde{F}_\C^\star(\sph^{n+m})$ into $\widetilde{F}_\C^{\star-m}(\sph^{n})$. The same is true in the complex-oriented case.
\begin{proof}
We consider only the Real-oriented case. Let $f:\sph^n\times \sph^m\rightarrow \sph^{n+m}$ be the map collapsing $\sph^n\vee \sph^m$ into a point. It is of topological degree one, so by the assumption and Hopf theorem it is homotopic to a regular map. The suspension isomorphism can then be written as
\begin{equation*}
    \widetilde{F}^\star(\sph^{n+m})\xrightarrow{f^\ast} F^\star(\sph^n\times \sph^m,\sph^n\vee \sph^m)\rightarrow\widetilde{F}^{\star-m}(\sph^n).
\end{equation*}
Since $f$ is homotopic to a regular map, the conclusion follows from Observation \ref{obs:inclusion} and Proposition \ref{prop:suspensionFC}.
\end{proof}
\end{prop}
\begin{proof}[Proof of Theorem \ref{thm:23mod4noMap}]
We apply Proposition \ref{prop:conditionForNoMapSpheres} to the Real-oriented theory $K\R^\star$. The conclusion of Proposition \ref{prop:conditionForNoMapSpheres} is equivalent to the inclusion $I^{n+m}\subset I^n$, where $I^{n+m}$ and $I^n$ are defined as in Definition \ref{defi:In}. Using Theorem \ref{thm:KOCSphere}, one verifies that if one of the numbers $n,m$ is congruent to $2$ modulo $4$ and the other is congruent to $2$ or $3$ modulo $4$, then $\eta \in I^{n+m}$, but $\eta\not\in I^n$. This shows that under this assumption there is no regular map from $\sph^n\times \sph^m$ into $\sph^{n+m}$ of topological degree one. in view of Proposition \ref{prop:dichotomy} this finishes the proof.
\end{proof}
\begin{rem}
Using a similar argument involving the functor $K\R_\C^\ast$ one can provide a new proof of Theorem \ref{thm:BKNoMap}.
\end{rem}
\subsection{A construction involving normed bilinear maps}\label{sec:products:bilinear}
It was already noticed in 1973 by Loday \cite{lodayApplicationsAlgebriquesTore1973}, that regular (actually, polynomial) maps from products of spheres into spheres sometimes arise as restrictions of appropriate bilinear maps. In this subsection we generalise the construction of Loday, constructing new pairs of positive integers $(n,m)$ for which there exists a regular map $f:\sph^n \times \sph^m\rightarrow \sph^{n+m}$ of topological degree one.

The following definition is classical:
\begin{defi}
Let $n,m,k$ be positive integers. A bilinear map $F:\R^n\times \R^m\rightarrow \R^k$ is said to be \emph{normed}, if 
\begin{equation*}
    \Vert F(x,y)\Vert =\Vert x\Vert \Vert y\Vert,
\end{equation*}
for all $x\in \R^n, y\in \R^m$, where $\Vert \cdot \Vert$ denotes the Euclidean norm.
\end{defi}
We also introduce the following simple condition which plays an important role in our construction
\begin{defi}
A bilinear map $F:\R^n\times \R^m\rightarrow \R^k$ is said to be \emph{nice}, if its first coordinate satisfies
\begin{equation*}
    F_1(x,y)=x_1y_1
\end{equation*}
for all $x\in \R^n, y\in \R^m$.
\end{defi}
\begin{lem}\label{lem:bilinear+odd=>existsRegular}
Let $n$ and $m$ be two positive integers. Assume that the following conditions are satisfied:
\begin{enumerate}
    \item there exists a nice normed bilinear map $F:\R^{n+1}\times \R^{m+1} \rightarrow \R^{n+m+1}$,
    \item the number $\binom{n+m}{n}$ is odd.
\end{enumerate}
Then, every map $f:\sph^{n}\times \sph^m\rightarrow \sph^{n+m}$ is homotopic to a regular one.
\begin{proof}
The second condition forces at least one of the numbers $n,m$ to be even, so in view of Proposition \ref{prop:dichotomy}, it suffices to construct a regular map $\sph^{n}\times \sph^m\rightarrow \sph^{n+m}$ of odd topological degree. In the course of the proof, for every $k$, through a shift, we identify $\sph^k$ with the unit sphere in $\R^{k+1}$ centred at the point $(1,0,\dots,0)$. That is, through a biregular isomorphism we identify $\sph^k$ with the algebraic subset of $\R^{k+1}$ given by the equation
\begin{equation*}
    (x_1-1)^2+\sum_{i=2}^{k+1}x_i^2=1,
\end{equation*}
which can be rewritten as 
\begin{equation*}
    \Vert x\Vert^2=2x_1.
\end{equation*}
We claim that the scaled restriction
\begin{equation*}
    f:\sph^n\times \sph^m\rightarrow \R^{n+m+1},\quad f:=\frac{1}{2}F\vert_{\sph^n\times \sph^m}
\end{equation*}
has its image contained in $\sph^{n+m}$, and that as a map into $\sph^{n+m}$ it is of odd topological degree. Since it is obviously regular (indeed polynomial) the conclusion follows from this claim.

For $x\in \sph^n,y\in \sph^m$ we have
\begin{equation*}
    \Vert f(x,y)\Vert^2=\frac{1}{4}\Vert F(x,y)\Vert^2=\frac{1}{4}\Vert x\Vert^2 \Vert y\Vert ^2=x_1y_1=F_1(x,y)=2f_1(x,y).
\end{equation*}
Therefore, the image of $f$ is indeed contained in $\sph^{n+m}$. 

Now, consider the associated map of projective spaces $\varphi:\R P^n\times \R P^m\rightarrow \R P^{n+m}$ which to a pair of points in homogeneous coordinates $(x,y)=(x_1:\dots:x_{n+1},y_1:\dots:y_{m+1})$ associates the point written in homogeneous coordinates as $F(x,y)$. It is well defined, as due to $F$ being normed the point $F(x,y)$ is nonzero if both $x$ and $y$ are nonzero.

For every $k$, let $\psi_k:\R P^k\rightarrow \sph^k$ be the blowup, explicitly given by 
\begin{equation*}
    \psi_k(x_1:\dots:x_{k+1})=\frac{2x_1}{\Vert x\Vert^2}x.
\end{equation*}
Geometrically, it collapses the hypersurface $x_1=0$ to the origin, and to every point $x\in \R P^{k}\backslash \R P^{k-1}$ it associates the second intersection point of the line $\R x$ with $\sph^k$ (other than the origin). We assert that the following diagram is commutative:
\begin{center}
\begin{tikzcd}
\R P^n\times \R P^m \arrow[d, "\psi_n\times \psi_m"] \arrow[r, "\varphi"] & \R P^{n+m} \arrow[d, "\psi_{n+m}"] \\
\sph^n\times \sph^m \arrow[r, "f"]                                      & \sph^{n+m}                       
\end{tikzcd}
\end{center}
This can be seen geometrically, but we provide an explicit computation instead. Starting with a point $(x,y)\in \R P^n\times \R P^m$ written in homogeneous coordinates, the composition of the left vertical and bottom horizontal arrows maps it to 
\begin{equation*}
    f(\psi_n(x),\psi_m(y))=\frac{1}{2}F\left(\frac{2x_1}{\Vert x\Vert ^2}x,\frac{2y_1}{\Vert y\Vert ^2}y\right)=\frac{2x_1y_1}{\Vert x\Vert^2\Vert y\Vert^2}F(x,y).
\end{equation*}
The other composition maps it to the same point as
\begin{equation*}
    \psi_{n+m}(\varphi(x,y))=\frac{2F_1(x,y)}{\Vert F(x,y)\Vert ^2}F(x,y)=\frac{2x_1y_1}{\Vert x\Vert^2\Vert y\Vert^2}F(x,y).
\end{equation*}
Since the maps $\psi_n\times \psi_m$ and $\psi_{n+m}$ are of mod two topological degree one, to conclude it suffices to show that $\varphi$ is of mod two topological degree one as well.

Denote by $a,b$ and $c$ the generators of the cohomology rings with coefficients in $\F_2$ of the projective spaces $\R P^n,\R P^m$ and $\R P^{n+m}$ respectively as $\F_2$-algebras. Thus, 
\begin{align*}
    H^\ast(\R P^n;\F_2)&=\F_2[a]/(a^{n+1}),\\
    H^\ast(\R P^m;\F_2)&=\F_2[b]/(b^{m+1}),\\ H^\ast(\R P^{n+m};\F_2)&=\F_2[c]/(c^{n+m+1}).
\end{align*}
Associating $a$ and $b$ with their pullbacks in $H^\ast(\R P^n\times \R P^m;\F_2)$ we also have
\begin{equation*}
    H^\ast(\R P^n\times \R P^m;\F_2)=\F_2[a,b]/(a^{n+1},b^{m+1}).
\end{equation*}
The following claim is classical; it goes back to Hopf \cite{hopfTopologischerBeitragZur1940}:
\begin{claim*}
We have
\begin{equation*}
    \varphi^\ast c=a+b.
\end{equation*}
\begin{proof}
For every $k$, denote by $\gamma_k$ the tautological line bundle over $\R P^k$. We begin by showing that the bundles $\pi_1^\ast\gamma_n\otimes \pi_2^\ast\gamma_m$ and $\varphi^\ast \gamma_{n+m}$ over $\R P^n \times \R P^m$ are isomorphic, where $\pi_1$ and $\pi_2$ are the projections from $\R P^n\times \R P^m$ onto the first and second factor respectively. The isomorphism is given in a natural manner; since $F$ is bilinear, for $(x,y)\in \R P^n \times \R P^m$ it induces a well defined map of fibres of the corresponding bundles
\begin{equation*}
    \R x\otimes \R y\rightarrow \R F(x,y).
\end{equation*}
The fact that this map is an isomorphism follows as $F$ is normed. Now, taking the first Stiefel-Whitney class we obtain the conclusion:
\begin{equation*}
    \varphi^\ast c=\varphi^\ast w_1(\gamma_{n+m})=w_1(\gamma_n\otimes \gamma_m)=a+b.
\end{equation*}
\end{proof}
\end{claim*}
From the claim and the assumption about the parity of $\binom{n+m}{n}$ we get
\begin{equation*}
    \varphi^\ast c^{n+m}=(a+b)^{n+m}=\sum_{i+j=n+m}\binom{n+m}{i}a^ib^j=\binom{n+m}{n}a^nb^m=a^nb^m.
\end{equation*}
Hence, the induced morphism
\begin{equation*}
    \varphi^\ast:H^{n+m}(\R P^{n+m};\F_2)\rightarrow H^{n+m}(\R P^n\times \R P^m;\F_2)
\end{equation*}
is an isomorphism. From the universal coefficient theorem it follows that the morphism
\begin{equation*}
    \varphi_\ast:H_{n+m}(\R P^n\times \R P^m;\F_2)\rightarrow H_{n+m}(\R P^{n+m};\F_2)
\end{equation*}
is an isomorphism as well, i.e. $\varphi$ is of mod two topological degree one, which was to be shown.
\end{proof}
\end{lem}
\begin{rem}
We do not know what the exact relation between the two conditions in the lemma is. Despite some effort, we were not even able to show that they are not equivalent. This seems rather unlikely though.
\end{rem}

The next step we need to take to obtain a positive result about the existence of regular maps from products of spheres into spheres is to prove an existence result about nice normed bilinear maps. Even though normed bilinear maps have been studied extensively throughout the last century, it seems like the additional niceness condition has not been considered in the literature. For this reason, we need to come up with a construction on our own. 

Recall the following definition from the introduction:
\begin{defi}
Let $n$ be a positive integer. The \emph{Radon-Hurwitz} number $\rho(n)$ is defined to be equal to $8a+2^b$, where $a\geq 0$ and $ 3\geq b\geq 0$ are such that
\begin{equation*}
    n=2^{4a+b}(2c+1)
\end{equation*}
for some $c\geq 0$.
\end{defi}
To state our result cleanly, we also introduce the following definition
\begin{defi}
A pair of nonnegative integers $(n,m)$ is said to be \emph{nice}, if it satisfies the assumptions of Lemma \ref{lem:bilinear+odd=>existsRegular}, i.e. if the following conditions are satisfied:
\begin{enumerate}
    \item there exists a nice normed bilinear map $F:\R^{n+1}\times \R^{m+1} \rightarrow \R^{n+m+1}$,
    \item the number $\binom{n+m}{n}$ is odd.
\end{enumerate}
\end{defi}
In particular, if $(n,m)$ is nice and the numbers are positive then every continuous map $\sph^{n}\times \sph^{m}\rightarrow \sph^{n+m}$ is homotopic to a regular one. Of course, the property of being nice is symmetric: if $(n,m)$ is nice then $(m,n)$ is nice as well.

A trivial example of a nice pair is $(0,0)$; the map $F$ can then be simply defined as $F(x,y)=xy$. More interesting examples of nice pairs can be constructed by repeatedly applying the following proposition:
\begin{prop}\label{prop:niceRecursion}
Let $(n,m)$ be a nice pair, and let $k$ be a natural number satisfying $\rho(k)>m$. Then, $(n+k,m)$ is a nice pair as well.
\begin{proof}
We begin by constructing a nice normed bilinear map
\begin{equation*}
    F:\R^{n+k+1}\times \R^{m+1}\rightarrow \R^{n+m+k+1}.
\end{equation*}
First of all, by assumption, there exists a nice normed bilinear map
\begin{equation*}
    G:\R^{n+1}\times \R^{m+1}\rightarrow \R^{n+m+1}.
\end{equation*}
Since $\rho(k)>m$, by the Hurwitz-Radon theorem (see \cite[Theorem 13.1.6]{bochnakRealAlgebraicGeometry1998}) there exists also a normed (not necessarily nice) bilinear map
\begin{equation*}
    H:\R^{k}\times \R^{m+1}\rightarrow \R^{k}.
\end{equation*}
For $x\in \R^{n+k+1}$ write $x=(\tilde x,\hat x)$, where $\tilde x\in \R^{n+1}$ and $\hat x \in  \R^{k}$. Finally define
\begin{align*}
     F&:\R^{n+k+1}\times \R^{m+1}\rightarrow \R^{n+k+m+1},\\
     F(x,y)&:=(G(\tilde x,y),H(\hat x,y)).
\end{align*}
Clearly, $F$ is a bilinear map. The fact that it is nice is obvious, as $G$ is nice by assumption. The fact that it is normed follows from the following computation:
\begin{equation*}
    \Vert F(x,y)\Vert^2=\Vert G(\tilde x, y)\Vert^2+\Vert H(\hat x, y)\Vert^2=\Vert \tilde x\Vert^2\Vert y\Vert^2+\Vert \hat x\Vert^2\Vert y\Vert^2=\Vert x\Vert^2\Vert y\Vert^2.
\end{equation*}

Let us now verify that the number $\binom{n+k+m}{n+k}$ is odd. The verification will rely on the following classical result
\begin{thm}[Kummer's theorem]
Let $a$ and $b$ be positive integers, expressed in binary as
\begin{align*}
    a=\varepsilon_0 +2\varepsilon_1 +\dots +2^{l}\varepsilon_{l},\\
    b=\varepsilon'_0 +2\varepsilon'_1 +\dots +2^l\varepsilon'_{l},
\end{align*}
with $\varepsilon_0,\dots,\varepsilon_l,\varepsilon_0',\dots,\varepsilon_l'\in\{0,1\}$ and $l\geq \log_2(a),\log_2(b)$. Then, the number $\binom{a+b}{a}$ is odd if and only if all the products $\varepsilon_i\varepsilon_i'$ vanish.
\end{thm}
Choose any $l\geq \log_2(m),\log_2(n+k)$ and express $m,n$ and $n+k$ in binary:
\begin{align*}
    m&=\varepsilon_0 +2\varepsilon_1 +\dots +2^{l}\varepsilon_{l},\\
    n&=\varepsilon'_0 +2\varepsilon'_1 +\dots +2^l\varepsilon'_{l}, \\
    n+k&=\varepsilon''_0 +2\varepsilon''_1 +\dots +2^l\varepsilon''_{l}.
\end{align*}
By assumption, the products $\varepsilon_i\varepsilon_i'$ vanish. 

It is easy to see from the definition of the Radon-Hurwitz number, that if we set $t$ to be equal to the $2$-adic exponent of $k$, then $m<\rho(k)\leq 2^t$. Hence, the numbers $\varepsilon_i$ are zero for $i\geq t$, so $\varepsilon_i\varepsilon_i''=0$ for $i\geq t$. On the other hand, since the difference between $n$ and $n+k$ is divisible by $2^t$ we have that $\varepsilon_i''=\varepsilon_i'$ for $i<t$. Hence $\varepsilon_i\varepsilon_i''=\varepsilon_i\varepsilon_i'=0$ for $i<t$ as well. The conclusion follows from another application of Kummer's theorem.
\end{proof}
\end{prop}

Loday's result (Theorem \ref{thm:Loday}) follows immediately from Proposition \ref{prop:niceRecursion}. Indeed, applying it once we get that $(0,m)$ is a nice pair for every $m>0$. Applying it once more we get that $(n,m)$ is a nice pair whenever $m<\rho(n)$.

Applying Proposition \ref{prop:niceRecursion} more than twice we obtain new results beyond Loday's construction. In particular, we can prove the following result from the introduction:
\begin{cor}
Let $n$ and $m$ be two positive integers. 
\begin{enumerate}
    \item Assume that $n\leq 8$ and that the number $\binom{n+m}{n}$ is odd. Then, the pair $(n,m)$ is nice and hence every continuous map from $\sph^n\times \sph^m$ into $\sph^{n+m}$ is homotopic to a regular one.
    \item Assume that $n\leq 3$. Then, the number $\binom{n+m}{n}$ is odd if and only if every continuous map from $\sph^n\times \sph^m$ into $\sph^{n+m}$ is homotopic to a regular one.
\end{enumerate}
\begin{proof}
Let us begin with the first part. Assume first that both $n$ and $m$ are strictly smaller than $8$, so that they can be written as
\begin{align*}
    n=\varepsilon_0+2\varepsilon_1+4\varepsilon_2, \\
    m=\varepsilon_0'+2\varepsilon_1'+4\varepsilon_2',
\end{align*}
where the coefficients are from $\{0,1\}$. By Kummer's theorem, we have $\varepsilon_i\varepsilon_i'=0$ for $i=0,1,2$. Since $\rho(1)=1$, applying Proposition \ref{prop:niceRecursion} we get that the pair $(\varepsilon_0,\varepsilon_0')$ is nice. Since $\rho(2)=2$ applying it once again we get that the pair $(\varepsilon_0+2\varepsilon_1,\varepsilon_0'+2\varepsilon_1')$ is nice. Since $\rho(4)=4$, applying it for the third time we get that $(n,m)$ is nice. 

Now assume that $n<8$ and $m$ is arbitrary. Then, $m$ can be written as $m=m'+m''$, where $m'<8$ and $m''$ is a multiple of $8$. By the previously covered case, we have that $(n,m')$ is nice. If $m''=0$ we are done, otherwise by the definition of $\rho$ we have $\rho(m'')\geq 8>n$, so from Proposition \ref{prop:niceRecursion} we conclude that $(n,m)$ is nice.

Finally, assume that $n=8$ and $m$ is arbitrary. Then, once more $m$ can be written as $m=m'+m''$ with $m'<8$ and $m''$ a multiple of $8$. Actually, due to Kummer's theorem the third binary digit of $m$ is zero, so $m''$ is divisible by $16$. Since $\rho(8)=8$, applying Proposition \ref{prop:niceRecursion} twice we get that the pair $(n,m')$ is nice. If $m''=0$ we are done, otherwise applying Proposition \ref{prop:niceRecursion} for the third time we find that $(n,m)$ is nice, since by definition of the Radon-Hurwitz number $\rho(m'')\geq \rho(16)=9>n$.

To obtain the second part of the conclusion it suffices to verify that if $n\leq 3$ and $\binom{n+m}{n}$ is even, then there is no regular map from $\sph^n\times \sph^m$ into $\sph^{n+m}$ of topological degree one. This follows from Kummer's theorem and Theorems \ref{thm:BKNoMap} and \ref{thm:23mod4noMap}.
\end{proof}
\end{cor}

As some more partial evidence for Conjecture \ref{conj:main}, we include the following result:
\begin{obs}
Let $n,m>0$ be such that $\binom{n+m}{n}$ is odd. Then, there exists a regular map
\begin{equation*}
    f:\R P^n\times \R P^m\rightarrow \sph^{n+m}
\end{equation*}
of mod two topological degree one.
\begin{proof}
As in the proof of Lemma \ref{lem:bilinear+odd=>existsRegular}, denote by $a,b$ and $c$ the cohomology classes, which generate the cohomology rings of $\R P^n,\R P^m$ and $\R P^{n+m}$ respectively. Thus as before we have
\begin{align*}
    H^\ast(\R P^n\times \R P^m;\F_2)&=\F_2[a,b]/(a^{n+1},b^{m+1}),\\
    H^\ast(\R P^{n+m};\F_2)&=\F_2[c]/(c^{n+m+1}).
\end{align*}
Since $\R P^{\infty}$ is an Eilenberg-MacLane space $K(\Z/2,1)$, there exists a map $g:\R P^n\times \R P^m\rightarrow \R P^{\infty}$, which satisfies $g^\ast w=a+b$, where $w$ is the generator of the cohomology ring of $\R P^\infty$. From the cellular approximation theorem we may assume that $g$ has its image contained in $\R P^{n+m}$. Thus, as a map into $\R P^{n+m}$ it satisfies $g^\ast c=a+b$. Now, according to \cite[Theorems 12.4.6 and 13.3.1]{bochnakRealAlgebraicGeometry1998}, for the map $g$ to be homotopic to a regular one it is necessary and sufficient that $g^\ast c\in H^1_\alg(\R P^n\times \R P^m;\F_2)$. This condition is satisfied (it follows for example from \cite[Theorem 12.4.6]{bochnakRealAlgebraicGeometry1998}), so we can assume that $g$ is a regular map. Now, as it was verified in the proof of Lemma \ref{lem:bilinear+odd=>existsRegular}, the condition $g^\ast c=a+b$ and the assumption that $\binom{n+m}{n}$ is odd force $g$ to be of mod two topological degree one. One can then compose $g$ with the blowup $\psi_{n+m}:\R P^{n+m}\rightarrow \sph^{n+m}$ to obtain the desired regular map $f:\R P^n\times \R P^m\rightarrow \sph^{n+m}$ of mod two topological degree one.
\end{proof}
\end{obs}

\section{Obstructions to algebraic approximation}\label{sec:obstructionsApproximation}
Throughout this section a smooth manifold is always considered without boundary, unless explicitly stated otherwise. 

Let $\var{X}$ be a nonsingular real affine variety. Let $M\subset \var{X}$ be a smooth compact submanifold of $\var{X}$. Recall from the introduction, that $M$ is said to \emph{admit an algebraic approximation} in $\var{X}$, if for every neighbourhood $\mathcal U$ of the inclusion $j:M\hookrightarrow \var{X}$ in the $\mathcal C^\infty$-topology there is a map $j':M\hookrightarrow \var{X}$ in $\mathcal U$ whose image $j'(M)$ is a nonsingular Zariski closed subset of $\var{X}$. Similarly, $M$ is said to \emph{admit a weak algebraic approximation} in $\var{X}$, if for every neighbourhood $\mathcal U$ of the inclusion $j:M\hookrightarrow \var{X}$ in the $\mathcal C^\infty$-topology there is a map $j':M\hookrightarrow \var{X}$ in $\mathcal U$ whose image $j'(M)$ is equal to the nonsingular locus of a Zariski closed subset of $\var{X}$.
\subsection{General obstructions to (weak) algebraic approximation}
We begin by stating some necessary conditions for the existence of a (weak) algebraic approximation, formulated in terms of the new invariants introduced in this paper. Later we will apply one of these conditions in the particular case of the complex-oriented cohomology theory $H_\C^\ast(-;\Z)$ to prove Theorem \ref{thm:newCounterex}. Nonetheless, we believe that the general form of the criteria in Observation \ref{obs:ObstructionsToApprox} and Proposition \ref{prop:generalObstructionWeak} might find use in future, in cases where the weaker invariants based on singular cohomology fail to obstruct algebraic approximation.

The obstruction to algebraic approximation is defined using the following construction:
\begin{defi}
Let $\var{X}$ be a nonsingular real affine variety, and let $M\subset \var{X}$ be its smooth compact submanifold of codimension $c$. Denote by $\nu$ the normal bundle to $M$ in $\var{X}$. Let 
\begin{equation*}
    \PT:\var{X}_+\rightarrow \Th(\nu)
\end{equation*}
be the Pontryagin-Thom collapse map onto the Thom space $\Th(\nu)$ of $\nu$, where $\var{X}_+:=\var{X}\sqcup\{\pt\}$ is the disjoint union of $\var{X}$ and a base point and $\PT$ is considered as a map between pointed spaces. Denote by $\nu_\C$ the complexification of $\nu$, that is, the complex bundle $\nu \otimes_\R \C$ over $M$. The natural inclusion $\nu \hookrightarrow \nu_\C$ gives an inclusion of Thom spaces 
\begin{equation*}
    \iota:\Th(\nu)\rightarrow \Th(\nu_\C).
\end{equation*}

Let $E^\ast$ be a complex-oriented generalised cohomology theory. As described in Section \ref{sec:preliminaries:cohomology}, the bundle $\nu_\C$ admits the natural Thom class $u_{\nu_\C}^E\in E^{2c}(\nu_\C,\nu_\C\backslash M)\cong \widetilde{E}^{2c}(\Th(\nu_\C))$ in degree $2c$. We denote its image through the following composition by $\gamma_M^E\in E^{2c}(\var{X})$:
\begin{equation*}
    \widetilde E^{2c}(\Th(\nu_\C))\xrightarrow{\iota^\ast}\widetilde E^{2c}(\Th(\nu))\xrightarrow{\PT^\ast} \widetilde E^{2c}(\var{X}_+)\cong E^{2c}(\var{X}). 
\end{equation*}

Now, let $F^\star$ be a Real-oriented equivariant cohomology theory. If we treat $M$ as a Real space with the trivial involution, we can also treat $\nu_\C$ as a Real vector bundle, with the involution induced by complex conjugation on $\C$. This way, $\nu$ is the fixed point subbundle of $\nu_\C$, and $\iota$ is a map of Real spaces if $\Th(\nu)$ is considered with the trivial involution. The Real bundle $\nu_\C$ carries the natural Thom class $u_{\nu_\C}^F\in \widetilde F^{c+c\tau}(\Th(\nu_\C))$. Similarly, we denote its image through the following composition by $\gamma_M^F\in F^{c+c\tau}(\var{X})$:
\begin{equation*}
    \widetilde F^{c+c\tau}(\Th(\nu_\C))\xrightarrow{\iota^\ast}\widetilde F^{c+c\tau}(\Th(\nu))\xrightarrow{\PT^\ast} \widetilde{F}^{c+c\tau}(\var{X}_+)\cong F^{c+c\tau}(\var{X}). 
\end{equation*}
\end{defi}

These classes can obstruct algebraic approximation of $M$ due to the following observation:
\begin{obs}\label{obs:ObstructionsToApprox}
Let $F^\star$ be a Real-oriented equivariant cohomology theory. Let $\var{X}$ be a complete nonsingular real affine variety, and let $M\subset \var{X}$ be a smooth submanifold of $\var{X}$ of codimension $c$. If $M$ admits an algebraic approximation in $\var{X}$, then $\gamma_M^F\in F^{c+c\tau}_\C(\var{X})$.

If we start with a complex-oriented cohomology theory $E^\ast$ instead, then similarly $\gamma_M^E\in E^{2c}_\C(\var{X})$.
\begin{proof}
We consider only the case of a Real-oriented cohomology theory, as the latter case is similar.

Let $j':M\rightarrow \var{X}$ be a smooth embedding, such that the set $\var{Y}:=j'(M)$ is a nonsingular Zariski closed subset of $\var{X}$. Let $\nu'$ be the normal bundle to $\var{Y}$ in $\var{X}$. If $j'$ is sufficiently close to $j$ in the $\mathcal C^\infty$-topology, then the bundles $\nu$ and $j'^\ast \nu'$ are isomorphic, and there is an isomorphism between the corresponding Thom spaces making the diagram below commutative up to homotopy:
\begin{center}
\begin{tikzcd}
\var{X}_+ \arrow[r] \arrow[rd] & \Th(\nu) \arrow[d] \\
                          & \Th(\nu')         
\end{tikzcd}
\end{center}

It follows that the classes $\gamma_M^F$ and $\gamma_{\var{Y}}^F$ are equal. Therefore, without loss of generality we may assume that $M=\var{Y}$ is a nonsingular Zariski closed subset of $\var{X}$.

After resolving singularities, we find a nonsingular projective complexification $\Rvar{X}$ of $\var{X}$, such that the Zariski closure of $\var{Y}$ in $\Rvar{X}$ is a Zariski closed nonsingular subset $\Rvar{Y}$ of $\Rvar{X}$, i.e. $\Rvar{Y}$ is a nonsingular projective complexification of $\var{Y}$. The normal bundle $\xi$ to $\Rvar{Y}$ in $\Rvar{X}$ has a natural structure of a Real bundle. In the usual way, we get an equivariant Pontryagin-Thom collapse map onto the Thom space of this bundle:
\begin{equation*}
    \PT_\xi:\Rvar{X}_+\rightarrow \Th(\xi).
\end{equation*}
We have $\xi\vert_{\var{Y}}\cong\nu_\C$ as Real vector bundles. The bundle $\xi$ carries a canonical Thom class $u_\xi^F\in \widetilde F^{c+c\tau}(\Th(\xi))$, whose restriction to $\Th(\nu_\C)$ by naturality is equal to its canonical Thom class $u^F_{\nu_\C}$. We have the commutative diagram
\begin{center}
\begin{tikzcd}
\widetilde F^{c+c\tau}(\Rvar{X}_+) \arrow[d, "i^\ast"] & \widetilde F^{c+c\tau}(\Th(\xi)) \arrow[l, "\PT_\xi^\ast"'] \arrow[d] \\
\widetilde F^{c+c\tau}(\var{X}_+)                  & \widetilde F^{c+c\tau}(\Th(\nu_\C)) \arrow[l, "(\iota\circ \PT)^\ast"']
\end{tikzcd}
\end{center}
It follows that $\gamma^F_M=i^\ast( \PT_{\xi}^\ast u^F_{\xi})$, where $i:\var{X}\hookrightarrow \Rvar{X}$ is the inclusion. Hence, by definition $\gamma^F_M\in F^{c+c\tau}_\C(\var{X})$.
\end{proof}
\end{obs}

In this setting we introduce another cohomology class, which will turn out to be an obstruction to weak algebraic approximation:
\begin{defi}
Let $\var{X}$ be a nonsingular real affine variety, let $M\subset \var{X}$ be its smooth compact submanifold of codimension $c$ and let $E^\ast$ be a complex-oriented cohomology theory. We denote by $\delta^E_M\in E^{2c}(M)$ the pullback of $\gamma^E_M\in E^{2c}(\var{X})$ through the inclusion $j:M\hookrightarrow \var{X}$. 

Similarly, if $F^\star$ is a Real-oriented equivariant cohomology theory, we denote the pullback of $\gamma^F_M\in F^{c+c\tau}(\var{X})$ through $j$ by $\delta^F_M\in F^{c+c\tau}(M)$.
\end{defi}
Even though this is not important for our applications, we note that these classes can be defined in a different way:
\begin{obs}
Let $\var{X}$ be a nonsingular real affine variety, let $M\subset \var{X}$ be its smooth compact submanifold of codimension $c$ and let $E^\ast$ be a complex-oriented cohomology theory. Then $\delta^E_M$ is equal to the $c$-th Conner-Floyd Chern class $c_{c}^E(\nu_\C)$ of the bundle $\nu_\C$ (see for example \cite[Section I.4]{adamsStableHomotopyGeneralised1974} for the definition of Conner-Floyd Chern classes in $E^\ast$). 

Similarly, if $F^\star$ is a Real-oriented cohomology theory, then $\delta^F_M=c_{c}^F(\nu_\C)\in F^{c+c\tau}(M)$, where $c_{c}$ is the $c$-th Conner-Floyd Chern class in $F^\star$ (see \cite[(4.5)]{arakiOrientationsTcohomologyTheories1979} for its definition).
\begin{proof}
We consider only the Real-oriented case. By construction $\delta^F_M$ is the pullback of the Thom class of the bundle $\nu_\C$ over $M$, i.e. its Euler class. According to \cite[Proposition 5.2]{arakiOrientationsTcohomologyTheories1979}, just like in the nonequivariant case, the Euler class is equal to $c_c^F(\nu_\C)$.
\end{proof}
\end{obs}

The classes $\delta^E_M$ and $\delta^F_M$ can obstruct weak algebraic approximation due to the following:
\begin{prop}\label{prop:generalObstructionWeak}
Let $E^\ast$ be a complex-oriented cohomology theory. Let $\var{X}$ be a complete nonsingular real affine variety, and let $M\subset \var{X}$ be a smooth submanifold of $\var{X}$ of codimension $c$. If $M$ admits a weak algebraic approximation in $\var{X}$, then there exists a class $\beta \in E^{2c}_\C(\var{X})$, whose restriction $j^\ast \beta$ to $M$ is equal to $\delta^E_M$.

Similarly, let $F^\star$ be a Real-oriented equivariant cohomology theory. If $M$ admits a weak algebraic approximation in $\var{X}$, then there exists a class $\beta \in F^{c+c\tau}_\C(\var{X})$, whose restriction to $M$ is equal to $\delta^F_M$.
\begin{proof}
We consider only the Real-oriented case. By a similar argument as in the proof of Observation \ref{obs:ObstructionsToApprox}, we may assume that $M$ is equal to the nonsingular locus of a Zariski closed subset $\var{Y}$ of $\var{X}$. We can then assume that $\var{Y}$ is in fact equal to the Zariski closure of $M$ in $\var{X}$, and hence that each irreducible component of $\var{Y}$ is of codimension $c$ in $\var{X}$.

Fix a nonsingular projective complexification $\Rvar{X}$ of $\var{X}$ and denote by $\Rvar{Y}$ the Zariski closure of $\var{Y}$ in $\Rvar{X}$. Applying the resolution of singularities theorem, we find a nonsingular projective $\R$-variety $\widetilde{\Rvar{Y}}$ and a regular map 
\begin{equation*}
    \varphi:\widetilde{\Rvar{Y}}\rightarrow \Rvar{Y}
\end{equation*}
which induces a birational equivalence between the irreducible components of $\widetilde{\Rvar{Y}}$ and the respective irreducible components of $\Rvar{Y}$. Moreover, since $\Rvar{Y}$ is nonsingular at points of $M$, we can assume that there is a Zariski open neighbourhood $M\subset U\subset \Rvar{Y}$ of $M$ in $\Rvar{Y}$ such that $\varphi\vert_{\varphi^{-1}(U)}$ is a biregular isomorphism onto $U$. 

The $\R$-variety $\widetilde{\Rvar{Y}}$ by assumption can be embedded as a nonsingular Zariski closed subset of $\C P^n$ for some $n$. Identifying $\C P^n$ with the space 
\begin{equation*}
    \{P\in \mathcal M_{(n+1)\times (n+1)}:P^\ast=P,P^2=P,\tr P=1\}
\end{equation*}
of hermitian idempotent matrices of trace one, we get an equivariant embedding of $\widetilde{\Rvar{Y}}$ as a Real submanifold of $\R^{m+m\tau}$, where $m=(n+1)^2$. Denote this embedding by $\eta:\widetilde{\Rvar{Y}}\hookrightarrow \R^{m+m\tau}$. 

Let $\theta:\widetilde{\Rvar{Y}}\rightarrow \R$ be a smooth map constantly equal to one on a Euclidean neighbourhood of $\varphi^{-1}(M)$ and satisfying $\supp \theta\subset \varphi^{-1}(U)$. After substituting $\frac{1}{2}(\theta+\theta\circ \sigma)$ for $\theta$, where $\sigma$ is the conjugation on $\widetilde{\Rvar{Y}}$, we may assume that $\theta$ satisfies $\theta\circ \sigma =\theta$. Consider the map 
\begin{equation*}
    \psi:\widetilde{\Rvar{Y}}\hookrightarrow \Rvar{X}\times \R^{m+m\tau},\quad \psi(y):=(\varphi(y),(1-\theta(y))\eta(y)).
\end{equation*}
We claim that it is an equivariant embedding of the smooth manifold $\widetilde{\Rvar{Y}}$ into $\Rvar{X}\times \R^{m+m\tau}$. Indeed, it is clearly immersive, and it is one-to-one since $\varphi\vert_{\varphi^{-1}(U)}$ is a bijection onto $U$ and $(1-\theta)\eta$ is injective on the complement of $\varphi^{-1}(U)$. Moreover, by construction there is a Euclidean neighbourhood $V$ of $M$ in $\Rvar{X}$, such that 
\begin{equation}\label{weakObstructionGeneral:eq1}
    \psi(\widetilde{\Rvar{Y}})\cap (V\times \R^{m+m\tau})=(\Rvar{Y} \cap V)\times \{0\}.
\end{equation}

Denote the normal bundle to $\psi(\widetilde{\Rvar{Y}})$ in $\Rvar{X}\times \R^{m+m\tau}$ by $\xi$. Consider the equivariant Pontryagin-Thom collapse map from the one-point compactification $\Sigma^{m+m\tau}\Rvar{X}_+$ of $\Rvar{X}\times \R^{m+m\tau}$ onto the Thom space of $\xi$:
\begin{equation*}
    \PT_\xi:\Sigma^{m+m\tau}\Rvar{X}_+\rightarrow \Th(\xi).
\end{equation*}
Let $u_\xi^F\in \widetilde{F}^{(c+m)(1+\tau)}(\Th(\xi))$ be the Thom class of $\xi$. Consider the following commutative diagram, where the vertical arrows are suspension isomorphisms:
\begin{center}
\begin{tikzcd}
\widetilde{F}^{(c+m)(1+\tau)}(\Sigma^{m+m\tau}M_+) \arrow[d, "\cong"] & \widetilde{F}^{(c+m)(1+\tau)}(\Sigma^{m+m\tau}\Rvar{X}_+) \arrow[l] \arrow[d, "\cong"] & \widetilde{F}^{(c+m)(1+\tau)}(\Th(\xi)) \arrow[l, "\PT^\ast_\xi"'] \\
\widetilde{F}^{c+c\tau}(M_+)                                          & \widetilde{F}^{c+c\tau}(\Rvar{X}_+) \arrow[l]                                         &                                                              
\end{tikzcd}
\end{center}
We claim that the image of $u_\xi^F$ in $\widetilde{F}^{c+c\tau}(M_+)=F^{c+c\tau}(M)$ is equal to $\delta^F_M$. This implies the conclusion, as the bottom arrow in the diagram can be factored as
\begin{equation*}
    F^{c+c\tau}(\Rvar{X})\xrightarrow{i^\ast} F^{c+c\tau}(\var{X}) \xrightarrow{j^\ast} F^{c+c\tau}(M),
\end{equation*}
and thus the required class $\beta$ may be chosen in the image of $F^{c+c\tau}(\Rvar{X})\to F^{c+c\tau}(\var{X})$, that is, in $F^{c+c\tau}_\C(\var{X})$.

To verify the claim, note that thanks to \eqref{weakObstructionGeneral:eq1} the restriction of the bundle $\xi$ to $M\subset \psi(\widetilde{\Rvar{Y}})$ is isomorphic to $\nu_\C \oplus \R^{m+m\tau}$, where the second summand denotes the trivial bundle over $M$ with fibre $\R^{m+m\tau}$. It follows that $\Th(\xi\vert_{M})$ can be identified with $\Th(\nu_\C\oplus \R^{m+m\tau})\cong \Sigma^{m+m\tau}\Th(\nu_\C)$, and that we have the commutative diagram
\begin{center}
\begin{tikzcd}
\widetilde F^{(c+m)(1+\tau)}(\Sigma^{m+m\tau}M_+) \arrow[d,"\cong"] & \widetilde F^{(c+m)(1+\tau)}(\Th(\xi)) \arrow[d,"\cong"] \arrow[l] \\
\widetilde F^{c+c\tau}(M_+)                   & \widetilde F^{c+c\tau}(\Th(\nu_\C)) \arrow[l]             
\end{tikzcd}
\end{center}
It follows from multiplicativity of orientation classes in a Real-oriented equivariant cohomology theory (see \cite[p. 412]{arakiOrientationsTcohomologyTheories1979}) that the orientation class $u^F_\xi$ maps to the orientation class $u^F_{\nu_\C}$ under the suspension isomorphism on the right. Hence it maps to the class $\delta^F_M\in F^{c+c\tau}(M)\cong \widetilde F^{c+c\tau}(M_+)$ by definition.
\end{proof}
\end{prop}

\subsection{Obstructions based on singular cohomology}
We will now restrict ourselves to the very special case of the complex-oriented cohomology theory $H^\ast(-;\Z)$. Throughout this section singular cohomology will always be considered with coefficients in $\Z$, unless written otherwise. Moreover, from now on, by $\gamma_M$ and $\delta_M$ we denote the classes $\gamma_M^E$ and $\delta^E_M$ respectively, corresponding to singular cohomology $E^\ast=H^\ast$. These classes can be described in an easier way under an additional orientability assumption:
\begin{lem}\label{lem:orientability=>Square}
Let $\var{X}$ be a complete nonsingular real affine variety, and let $M\subset \var{X}$ be a smooth compact submanifold of $\var{X}$ of codimension $c$. Assume that we are given orientations of the manifolds $M$ and $\var{X}$, so that through Poincar\'e duality $M$ represents a cohomology class $\alpha_M \in H^c(\var{X})$. Then, $\gamma_M$ is up to sign equal to $\alpha_M^2$.
\begin{proof}
The orientations of $\var{X}$ and $M$ induce an orientation of the normal bundle $\nu$ to $M$ in $\var{X}$. This gives a Thom class for this bundle, which we denote by
\begin{equation*}
    u_\nu \in \widetilde{H}^c(\Th(\nu))\cong H^c(\nu,\nu\backslash M).
\end{equation*}
The pullback of $u_\nu$ to $H^c(\var{X})$ under the Pontryagin-Thom collapse map
\begin{equation*}
    \PT:X_+\rightarrow \Th(\nu)
\end{equation*}
is up to sign equal to $\alpha_M$ \cite[Problem 11-C]{milnorCharacteristicClasses1974}. The bundle $\nu_\C$ splits as a direct sum $\nu_\C=\nu \oplus i\nu$. Denote by $u_{i\nu}$ the image of $u_\nu$ in $H^c(i\nu,i\nu\backslash M)$ under the isomorphism induced by $\nu\cong i\nu$. Notice that the pair $(u_\nu,u_{i\nu})$ maps to a Thom class of the bundle $\nu_\C$ through the cross product
\begin{equation*}
    H^c(\nu,\nu\backslash M)\times H^{c}(i\nu,i\nu\backslash M)\rightarrow H^{2c}((\nu,\nu\backslash M) \times (i\nu,i\nu\backslash M))\rightarrow H^{2c}(\nu_\C,\nu_\C\backslash M).
\end{equation*}
To be more explicit, let us fix a point $x\in M$ and a positive basis $e_1,\dots,e_c$ of the fibre of $\nu$ over $x$. Then, the restriction of $u_\nu \times u_{i\nu}$ to the fibre over $x$ is a fundamental cohomology class, which corresponds to the orientation $e_1,\dots,e_c,ie_1,\dots,ie_c$ of the fibre (cf. \cite[Proof of Property 9.6]{milnorCharacteristicClasses1974}). This differs from the standard orientation $e_1,ie_1,\dots,e_c,ie_c$ by the sign $(-1)^{\binom{c}{2}}$. Thus $u_\nu\times u_{i\nu}$ is a Thom class of the bundle $\nu_\C$ and it satisfies, 
\begin{equation*}
    u_{\nu_\C}=(-1)^{\binom{c}{2}}u_\nu \times u_{i\nu}.
\end{equation*}

Now, we have the commutative diagram:
\begin{center}
\begin{tikzcd}
{H^c(\nu,\nu\backslash M)\times H^c(i\nu,i\nu\backslash M)} \arrow[d] \arrow[r] & {H^{2c}(\nu_\C,\nu_\C\backslash M)} \arrow[d] \arrow[dd, "\iota^\ast", bend left=80] \\
{H^c(\nu,\nu\backslash M)\times H^c(i\nu)} \arrow[r] \arrow[d]                  & {H^{2c}(\nu_\C,\nu_\C\backslash i\nu)} \arrow[d]                       \\
{H^c(\nu,\nu\backslash M)\times H^c(M)} \arrow[r]                               & H^{2c}(\nu,\nu\backslash M)                                                                         
\end{tikzcd}
\end{center}
It follows that $\iota^\ast u_{\nu_\C}$ is up to sign equal to $u_\nu\smile p^\ast(e_{i\nu})$, where $e_{i\nu}$ is the Euler class of $i\nu$ and $p:\nu\rightarrow M$ is the bundle projection. The class $e_{i\nu}$ is equal to the Euler class $e_\nu$ of $\nu$, as the two bundles are isomorphic. From commutativity of the diagram
\begin{center}
\begin{tikzcd}
{H^c(\nu,\nu\backslash M)\times H^c(\nu,\nu\backslash M)} \arrow[d] \arrow[rd] &                               \\
{H^c(\nu,\nu\backslash M)\times H^c(M)} \arrow[r]                              & {H^{2c}(\nu,\nu\backslash M)}
\end{tikzcd}
\end{center}
we then deduce that $\iota^\ast u_{\nu_\C}$ up to sign equals $u_{\nu}^2$. Pulling back to $\var{X}$ we conclude that $\gamma_M$ up to sign equals $\alpha_M^2$, which was to be shown.
\end{proof}
\end{lem}

To summarize, from Observation \ref{obs:ObstructionsToApprox} and Lemma \ref{lem:orientability=>Square} we deduce the following
\begin{cor}\label{cor:squareNoApproximation}
Let $\var{X}$ be a complete nonsingular real affine variety, and let $M\subset \var{X}$ be a smooth compact submanifold of $\var{X}$ of codimension $c$. Assume that we are given orientations of the manifolds $M$ and $\var{X}$, so that $M$ represents a cohomology class $\alpha_M \in H^c(\var{X})$. If $M$ admits an algebraic approximation in $\var{X}$, then $\alpha_M^2\in H^{2c}_{\C}(\var{X})$.
\end{cor}

Using this, one can already construct interesting examples of submanifolds of real affine varieties which do not admit algebraic approximations, showing that the assumption about the dimension in Theorem \ref{thm:BenoistApprox} is essential. To obstruct weak algebraic approximation we make use of the following version of Proposition \ref{prop:generalObstructionWeak}:
\begin{prop}\label{prop:cubeNoWeakApprox}
Let $\var{X}$ be a complete nonsingular real affine variety, and let $M\subset \var{X}$ be a smooth compact submanifold of $\var{X}$ of codimension $c$. Assume that we are given orientations of the manifolds $M$ and $\var{X}$, so that $M$ represents a cohomology class $\alpha_M \in H^c(\var{X})$. If $M$ admits a weak algebraic approximation in $\var{X}$, then there exists a class $\beta\in  H^{2c}_{\C}(\var{X})$, such that
\begin{equation}\label{cubeCondition}
    \beta \alpha_M =\alpha_M^3.
\end{equation}
\begin{proof}
Denote by $j:M\hookrightarrow \var{X}$ the inclusion. From Proposition \ref{prop:generalObstructionWeak} and Lemma \ref{lem:orientability=>Square} we deduce that there exists a class $\beta\in  H^{2c}_{\C}(\var{X})$, which satisfies
\begin{equation*}
    j^\ast \beta = j^\ast \alpha_M^2.
\end{equation*}
Capping both sides of the equation with the fundamental class of $M$ and applying the pushforward $j_{\ast}$ we get
\begin{equation*}
    \beta \frown j_\ast [M]=j_\ast(j^\ast\beta \frown [M])=j_\ast(j^\ast \alpha_M^2\frown [M])=\alpha_M^2\frown j_\ast[M].
\end{equation*}
By definition we have $\alpha_M\frown [\var{X}]=[M]$, so
\begin{equation*}
    \beta \alpha_M \frown [\var{X}]=\alpha_M^3 \frown [\var{X}].
\end{equation*}
The conclusion follows as the Poincar\'e duality map is an isomorphism.
\end{proof}
\end{prop}

We can now begin the construction of the example discussed in Theorem \ref{thm:newCounterex}. It will be based on the following lemma:
\begin{lem}\label{lem:constructionDouble}
Let $N$ be a smooth connected compact manifold of positive dimension. Then, there exists a nonsingular real affine variety $\var{X}$ diffeomorphic to the disjoint union of two copies of $N$ through a diffeomorphism $\psi:\var{X}\rightarrow N\sqcup N$, with the property that for every $0\leq k\leq \dim N-1$, $H^{k}_{\C}(\var{X})$ is contained in the image of the homomorphism
\begin{equation*}
    H^{k}(N)\xrightarrow{p^\ast} H^{k}(N\sqcup N)\xrightarrow{\psi^\ast} H^{k}(\var{X}),
\end{equation*}
where the first arrow is induced by the projection $p:N\sqcup N\cong N\times \{-1,1\}\rightarrow N$.
\begin{proof}
According to the Nash-Tognoli theorem, there exists a nonsingular real affine variety $\var{Y}$ diffeomorphic to $N$. Consider the variety $\var{Z}:=\var{Y}\times \sph^1$. The set $\var{Y}\times \{-1,1\}$ is a smooth codimension one submanifold of $\var{Z}$ homologous to zero. This allows us to apply \cite[Proposition 2.7]{KucharzKurdykaComplexification2011} to find an embedding $j':\var{Y}\times \{-1,1\} \rightarrow \var{Z}$ arbitrarily close to the inclusion $j:\var{Y}\times \{-1,1\} \hookrightarrow \var{Z}$ in the $\mathcal C^\infty$-topology, such that the set $\var{X}:=j'(\var{Y}\times \{-1,1\})$ is a nonsingular Zariski closed subset of $\var{Z}$, and such that 
\begin{equation}\label{constrEq1}
    H^{k}_{\C}(\var{X})= \kappa^\ast H^{k}_\C(\var{Z})\subset \kappa^\ast H^k(\var{Z})\text{ for }0\leq k\leq \dim N-1,
\end{equation}
where 
\begin{equation*}
    \kappa:\var{X}\hookrightarrow \var{Z}
\end{equation*}
is the inclusion map. We can assume that $j'$ is so close to $j$ that the two maps are homotopic. Denote by $\psi:\var{X}\rightarrow N\sqcup N$ the inverse of $j'$, where we identify $\var{Y}\times \{-1,1\}$ with $N\sqcup N$. It follows that
\begin{equation}\label{constrEq2}
    \kappa^\ast H^{\ast}(\var{Z})=(j\circ \psi)^\ast H^{\ast}(\var{Z}).
\end{equation}
Moreover, the inclusion $j:\var{Y}\times \{-1,1\}\hookrightarrow \var{Z}$ up to homotopy factors through the projection $p:\var{Y}\times \{-1,1\}\rightarrow \var{Y}$, so
\begin{equation*}
    j^\ast H^\ast(\var{Z})\subset p^\ast H^\ast(\var{Y}).
\end{equation*}
Combining this with \eqref{constrEq1} and \eqref{constrEq2} yields the conclusion.
\end{proof}
\end{lem}

\begin{proof}[Proof of Theorem \ref{thm:newCounterex}]
Consider the smooth manifold $\C P^3$. Apply Lemma \ref{lem:constructionDouble} to find a nonsingular real affine variety $\var{X}$ diffeomorphic to the disjoint union $\C P^3\sqcup \C P^3$, such that $H_{\C}^4(\var{X})$ is contained in the image of the map
\begin{equation*}
    H^4(\C P^3)\xrightarrow{p^\ast} H^4(\var{X}), 
\end{equation*}
where we identify $\var{X}$ with $\C P^3\sqcup \C P^3$ and $p$ is the projection $p:\C P^3\times \{-1,1\}\rightarrow \C P^3$. The cohomology ring of $\var{X}$ is the cartesian product of two copies of $H^\ast(\C P^3)$, so it may be written as 
\begin{equation*}
    H^\ast(\var{X})=\Z[u_1]/(u_1^4)\times \Z[u_2]/(u_2^4), 
\end{equation*}
where $|u_1|=|u_2|=2$. With this notation we have $H^4_{\C}(\var{X})\subset p^\ast H^4(\C P^3)=\Z (u_1^2+ u_2^2)$.

Let $M_1$ be a smooth complex degree two hypersurface in the first copy of $\C P^3$, and let $M_2$ be a smooth complex degree four hypersurface in the second copy of $\C P^3$, both treated as smooth manifolds. Define $M:=M_1\sqcup M_2\subset \var{X}$. We claim that the pair $(\var{X},M)$ satisfies the desired conditions.

Let us first verify that $M$ does not admit a weak algebraic approximation in $\var{X}$. The cohomology class $\alpha_M$ represented by $M$ is equal to $2u_1+4u_2$. Assume that $M$ does admit a weak algebraic approximation in $\var{X}$. Then, according to Proposition \ref{prop:cubeNoWeakApprox}, there exists an element $\beta\in H^4_{\C}(\var{X})$ such that
\begin{equation*}
    \beta\alpha_M=\alpha_M^3.
\end{equation*}
Since $\beta \in \Z (u_1^2+ u_2^2)$, it can be written as $n(u_1^2+u_2^2)$ for some $n\in \Z$. We then have
\begin{equation*}
    2nu_1^3+4nu_2^3=8u_1^3+64u_2^3.
\end{equation*}
It follows that $4=n=16$, which is a contradiction.

Let us now verify that the bordism class of the inclusion $[j:M\hookrightarrow \var{X}]$ is zero. It suffices to show that the bordism classes of the inclusions of the two components $[j_1:M_1\hookrightarrow \C P^3]$ and $[j_2:M_2\hookrightarrow \C P^3]$ are zero. We will show this only for the first one, as the second case is completely analogous.

According to \cite[Theorem 17.2]{connerDifferentiablePeriodicMaps1979}, the bordism class $[j_1:M_1\hookrightarrow \C P^3]$ being zero is equivalent to the equation
\begin{equation}\label{ex:conditionVanishingSW}
    (j_1^\ast a \smile w_I(TM_1))\frown[M_1]_{\Z/2}=0
\end{equation}
being true for all $a$ and $w_I$, where $w_I(TM_1)$ is a monomial in Stiefel-Whitney classes of the tangent bundle $w_I(TM_1)=(w_1(TM_1))^{i_1}\dots (w_{4}(TM_1))^{i_{4}}$, $a  \in H^\ast(\C P^{3};\Z/2)$ is a cohomology class such that $|a|+|w_I(TM_1)|=4$ and $[M_1]_{\Z/2}$ is the fundamental class of $M_1$ with $\Z/2$-coefficients.

The variety $M_1$ can be described as the zero locus of some section of the complex line bundle $\mathcal O(2)$ over $\C P^{3}$ transverse to zero. Hence, the normal bundle $\nu$ to $M_1$ in $\C P^{3}$ is isomorphic to $\mathcal O(2)\vert_{M_1}$. As $\mathcal O(2)$ is a complex line bundle whose first Chern class is even, it follows that its Stiefel-Whitney classes vanish (cf. \cite[Problem 14-B]{milnorCharacteristicClasses1974}). From the Whitney sum formula it follows that 
\begin{equation*}
    w_i(TM_1)=w_i(j_1^\ast T(\C P^{3}))=j_1^\ast w_i(T(\C P^{3}))    
\end{equation*}
for all $i$. In particular, all the Stiefel-Whitney classes of the bundle $TM_1$ are pullbacks of classes from $H^\ast(\C P^{3};\Z/2)$. The equation \eqref{ex:conditionVanishingSW} now reduces to the following:
\begin{equation}\label{ex:conditionVanishingSW2}
    j_1^\ast a \frown [M_1]_{\Z/2}=0\text{ for all }a\in H^{4}(\C P^{3};\Z/2).
\end{equation}
Using the projection formula, we find that
\begin{equation*}
    j_{1\ast} (j_1^\ast a \frown [M_1]_{\Z/2})=a \frown j_{1\ast}[M_1]_{\Z/2}=a\frown 0=0, 
\end{equation*}
since $M_1$ is a hypersurface of even degree. As the complex hypersurface $M_1$ is connected, the morphism on zeroth homology groups induced by $j_1$ is injective. This proves \eqref{ex:conditionVanishingSW2} and finishes the proof.
\end{proof}
\begin{rem}
In the example just constructed the smooth manifold $M$ has two connected components. A variant of the construction, where one considers the connected sum of complex projective spaces instead of their disjoint sum, should also yield examples where both $\var{X}$ and $M$ are connected.
\end{rem}
\section*{Acknowledgments}
The author was partially supported by the Polish National Science Center (NCN) under grant number 2024/53/N/ST1/02481.
\printbibliography
\end{document}